\documentclass[a4paper,11pt]{amsart}
\usepackage{epsfig}
\usepackage{graphicx,color}
\usepackage{amsmath,amsfonts,amssymb,eufrak}
\usepackage{float}
\usepackage{graphics}
\usepackage{latexsym}
\usepackage{float}
\usepackage{verbatim}
\usepackage{tikz}
\usepackage{psfrag}
\usepackage[all]{xy} 
\usepackage{mathdots}
\usepackage{lscape}

\xyoption{all}         

\newtheorem{lem}{Lemma}[section]
\newtheorem{prop}[lem]{Proposition}
\newtheorem{cor}[lem]{Corollary}
\newtheorem{thm}[lem]{Theorem}

\theoremstyle{definition}
\newtheorem{dfn}[lem]{Definition}
\newtheorem{rem}[lem]{Remark}

\newcommand{\Rad}{\operatorname{rad}\nolimits}

\newcommand{\mo}{\operatorname{mod}\nolimits}
\newcommand{\T}{\operatorname{\mathcal T}\nolimits}
\newcommand{\U}{\operatorname{\mathbb U}\nolimits}
\newcommand{\Cyl}{\operatorname{Cyl}\nolimits}
\newcommand{\Pp}{\mathcal{P}}
\newcommand{\Ip}{\mathcal{I}}
\newcommand{\GG}{\operatorname{\overline{\Gamma}}\nolimits}
\newcommand{\GGG}{\operatorname{\underline{\Gamma}}\nolimits}

\begin{document}
\title[A geometric realization of categories of type $\tilde{A}$]{A geometric realization of tame categories}
\author{Karin Baur}
\address{Institut f\"{u}r Mathematik und Wissenschaftliches Rechnen, 
Universit\"{a}t Graz, NAWI Graz, Heinrichstrasse 36, 
A-8010 Graz, Austria}
\email{baurk@uni-graz.at}

\author{Hermund Andr\' e Torkildsen}
\address{Faculty of Teacher and Interpreter Education,
S\o{}r-Tr\o{}ndelag University College,
7004 Trondheim, Norway
}
\email{hermund.a.torkildsen@hist.no} 

%\date{version feb 23, 2015}

\begin{abstract}
We give a geometric realization of module categories of type $\tilde{A}_n$. 
We work with oriented arcs to define a translation quiver isomorphic to the 
Auslander-Reiten quiver of the module category of type $\tilde{A}_n$. To get 
a description of the module category, we introduce long moves between 
arcs. These allow us to include the infinite radical in the geometric description. 
Finally, our results can also be used to describe the 
corresponding cluster categories by taking unoriented arcs instead. 
\end{abstract}

\maketitle

%%%%%%%%%%%%%%% NEW SECTION %%%%%%%%%%%%%%%%%%%%%
%
%\section*{comments, to do list}
%
%\begin{enumerate}
%\item
%\textcolor{red}{comment about closed loops ok? (page 2)} \textcolor{orange}{Yes, I believe so. It could be the correct way to interpret homogeneous objects.}
%\item
%\textcolor{red}{are you happy with this: 
%On page 5, I have explained which quasi-simple corresponds to which vertex of $Q$}\textcolor{orange}{Yes, good.}
%\end{enumerate}
%

%%%%%%%%%%%%%%% NEW SECTION %%%%%%%%%%%%%%%%%%%%%

\section*{Introduction}

In this article, we want to describe the module categories and cluster categories 
of type $\tilde{A}_n$ combinatorially. In particular, we will provide a 
geometric combinatorial interpretation of their infinite radicals. 

Let $k$ be an algebraically closed field. 
We will first concentrate on $\mo\tilde{A}=\mo kQ_{g,h}$ for $g\ge h>0$ and 
$\Rad^{\infty}(\mo \tilde{A})=\cap_i \Rad^i(\mo \tilde{A})$. 

Recall that the Auslander Reiten quiver AR$(\mo kQ_{g,h})$ describes the 
isomorphism classes of indecomposable modules and 
the irreducible maps between them. It does not describe the elements of the 
infinite radical $\Rad^{\infty}(\mo\tilde{A})$.  
The shape of the AR-quiver of $\mo\tilde{A}$ is well-known, it consists of several 
connected components (cf. \cite{ringel}). 

We will write $\Pp$, $\Ip$, to 
denote the subcategories of $\mo kQ_{g,h}$ 
of the preprojective, preinjective modules of $\mo kQ_{g,h}$, 
$\T_g$, $\T_h$ and $\T^{\lambda}$ for $\lambda \in k\setminus\{0\}$ for 
the subcategories of the 
modules whose vertices belong to the tubes of rank $g$, $h$ and to the homogeneous 
tubes respectively.  

We describe indecomposable modules in $\Pp$, $\Ip$, $\T_g$ or $\T_h$ 
via oriented arcs in the annulus. This geometric model has been used in various 
articles on (cluster) tubes, module/cluster categories of type $\tilde{A}$, 
cf. \cite{bm}, \cite{bbm}, \cite{bz}, \cite{gehrig}, \cite{to}, \cite{w}, \cite{to}. 
Note that 
the geometric 
interpretation of elements in the homogenous tubes can be achieved via 
closed loops in the annulus, with labels. It seems less natural to us and we thus restrict 
our attention to the other components.

However, we will also focus on the infinite radical of the module and the 
cluster categories of type $\tilde{A}$. We will introduce long moves as a new tool, 
and show how we can interpret them as elements of the infinite radical. 

In \cite{br}, the author introduces a quiver $Q_m$, for every $m\ge 1$, 
by truncating the different components 
in AR$(\mo kQ_{g,h})$ at appropriate levels and then inserting additional 
arrows linking the different components. 
He then defines full subcategories 
$\mathcal J_m Q_{g,h}$ of the module category $\mo kQ_{g,h}$ consisting 
of modules in the union of $\Pp_m$, $\Ip_m$, $(\T_g)_{2m(n+1)}$, $(\T_h)_{2m(n+1)}$, 
$\T^{\lambda}_{2m(n+1)}$ ($\lambda\in k\setminus\{0\}$) (see Section~\ref{sec:bruestle}).
He finally proves that the $k$-category $\mathcal C(Q_m)$ defined by 
$Q_m$, with some relations,  is isomorphic to $\mathcal J_m Q_{g,h}$.

Note that for every module $M\in\mo kQ_{g,h}$, there exists 
$m\ge 1$, such that $M$ is isomorphic to a direct sum of objects of $\mathcal J_m Q_{g,h}$ 
and hence the categories $\mathcal J_m Q_{g,h}$ can be used to describe $\mo kQ_{g,h}$.

In Sections \ref{sec:mod-cat} and \ref{sec:bruestle}, 
we will fix the notation, recall the truncated 
quiver $Q_m$ from \cite{br}. 

In Section \ref{sec:arcs-quiver}, we use oriented arcs in the annulus to 
define the translation quiver 
$\Gamma$ which is isomorphic to AR$(\mo kQ_{g,h})$. 

Geometric models for categories have been studied by various authors in 
recent years. 
A geometric model for cluster categories of type $\widetilde{A}_n$ has been described in 
in \cite{bz} (treating marked surfacds) 
and for $m$-cluster categories of type $\tilde{A}_n$ in \cite{to} in independent work. 
It is based on arcs in an annulus with marked points on the boundary. 
Annuli have also been 
used to describe string modules (\cite{w}) and (cluster) tubes (\cite{gehrig}, \cite{bm}). 

Our 
goal is to 
give a geometric model for the infinite radical in terms of moves between such arcs. 
To do this, we define long moves between arcs of the annulus 
in Section~\ref{sec:inf-rad} and define a quiver $\GG$, whose objects are the same as 
the objects of $\Gamma$. The arrows of $\Gamma$ are kept, and we add new arrows 
for the long moves. 
The long moves are the ones corresponding to elements of the infinite radical. 

We then show in Section~\ref{sec:arquivers} that up to the homogenous components, 
we obtain an isomorphism between 
certain subquivers $\GG_m$ of $\GG$ and the truncated quiver 
$Q_m$ from \cite{br}. 
We thus give a full description of ${\mathcal J}_m(Q_{g,h})$ for every $m\ge 1$ via 
the combinatorial geometric model. 
Finally, in Section \ref{sec:cluster-cat} we show how we can use the above results 
to obtain the description of 
appropriate subcategories of the cluster category of type $\tilde{A}$. 
In the appendix, we include the geometric description of 
the relations imposed in Theorem~\ref{thm:hauptsatz}.

%%%%%%%%%%%%%%% NEW SECTION %%%%%%%%%%%%%%%%%%%%%
\section{The module category of type $\tilde{A}$}\label{sec:mod-cat}
%%%%%%%%%%%%%%%%%%%%%%%%%%%%%%%%%%%%

\subsection{The AR-quiver of the module category} \label{ssec:AR-quiver-mod}

As quiver of type $\tilde{A}_n$ we choose a quiver $Q_{g,h}$ on $n+1=g+h$ vertices, 
for $g\ge h>0$. 
Let $0,1,2,\dots, g-1, g, g+1, \dots, n$ be arranged clockwise on a circle. 
Then $Q_{g,h}$ has $g$ successive clockwise arrows $\beta_i:i\to i+1$ for 
$i=0\dots, g-1$ 
and $h$ successive anti-clockwise arrows, $\alpha_i: i+1\to i$ for $i=g,g+1,\dots, n$ 
(with $\alpha_n=0\to n$) 
with a source at $0$ and a sink at $g$. 

$$
\xymatrix@=2mm@R=4mm{ 
 & 1 \ar[rr]^{\beta_1} & & 2 \ar[r]^{\beta_2} & & \cdots && &\ar[r]& g-1\ar[rd]^{\beta_{g-1}} \\
 0\ar[ru]^{\beta_0}\ar[rd]_{\alpha_n} & & && && && &&  g\\ 
& n \ar[rrr]^{\alpha_{n-1}\ } &&& n-1\ar[rr] && \dots \ar[rr] && g+1\ar[rru]_{\alpha_g}
}
$$ 

Let $k$ be an algebraically closed field. We consider the module category $\mo\tilde{A}_n$ 
of left $kQ_{g,h}$-modules. 
Its AR-quiver describes the 
isomorphism classes of indecomposable modules and 
the irreducible maps between them. It does not describe the elements of the 
infinite radical 
$\Rad^{\infty}(\mo \tilde{A})=\cap_i \Rad^i(\mo \tilde{A}_n)$. 

In this article, we want to describe the whole module category $\mo\tilde{A}_n$ 
by giving a geometric combinatorial interpretation of its infinite radical. 
The indecomposable modules in $\Pp$, $\Ip$, $\T_g$ or $\T_h$ are described 
via oriented arcs in the annulus. The idea of using oriented or unoriented arcs 
has been used in various 
articles on (cluster) tubes, module/cluster categories of type $\tilde{A}_n$, 
cf. \cite{bm}, \cite{bbm}, \cite{w}, \cite{to}, \cite{bz}. We will recall the set-up necessary 
for our purpose. 
The focus of this article is the infinite radical of the module and the 
cluster categories of type $\tilde{A}_n$, restricting to the preprojective/preinjective components 
and the regular tubes $\T_g$, $\T_h$. We will introduce long moves as a new tool. 
Long moves correspond to elements in the infinite radical. 

\vskip 5pt

As a running example we take $n=4$ with $g=3$ and 
$h=2$. The quiver $Q_{3,2}$ looks as follows: 

$$
\xymatrix@=2mm@R=4mm{ 
 & \bullet^1\ar[rr] & & \bullet^2\ar[rd]    \\ 
\bullet^0\ar[ru]\ar[rrd] & &&& \bullet^3 \\ 
&& \bullet^4\ar[rru] 
}
$$ 

%The AR-quiver of $\mo\tilde{A}_n$ has several connected componets: 
The preprojective 
and preinjective components $\Pp$ and $\Ip$ are viewed as infinite horizontal tubes, 
with $\Pp$ bounded from the left by the slice containing the projective indecomposables 
and $\Ip$ bounded from the right by the slice containing the injective indecomposables. 
We are adopting the convention of \cite{br} to draw the tube $\T_h$ upside down. This will 
be conventient when describing relations involving the infinite radical. 

In addition, there are infinitely many tubes: two regular tubes $\T_g$ and $\T_h$ of rank 
$g$ and $h$ respectively and the homogeneous tubes of rank $1$, $T_{\lambda}$, 
$\lambda\in k\setminus\{0\}$. See Figure~\ref{fig:AR-3-2} for the AR-quiver of $\mo(kQ_{3,2})$. 
Note that in the figure,  the component $\T_2=\T_h$ is drawn upside down, as mentioned above. 
We write $M_i$ for the modules sitting at the mouth of $\T_g$ and $\overline{M}_i$ for the 
modules at the mouth of $\T_h$. These are the quasi-simple modules.
In order to keep the picture readable, we have omitted the dashed lines for 
the AR translate $\tau$. The effect of $\tau$ is to send a vertex horizontally 
to its neightbour to the left. For the homogeneous tubes $\T^{\lambda}$, this means that 
$\tau$ sends a vertex to itself.

\begin{figure}[h]
\begin{center}
\xymatrix@=1mm@R=2mm{ 
 & & & & & % 
_{\overline{M}_1}\ar[rd]\ar@{--}[ddd] && _{\overline{M}_2}\ar[rd] && _{\overline{M}_1}\ar@{--}[ddd] && % 
  & & &   % 
   &\\
 %%%
 & & & & & % 
& \bullet \ar[ru]\ar[rd]&& \bullet\ar[rd]\ar[ru] &&  &   % 
  & & & \vdots &% 
   &\\
%%%
&_{P_3} \ar[rd]  && \bullet \ar[rd] & \dots & % 
\bullet\ar[ru]\ar@{..>}[rd] && \bullet \ar[ru] \ar@{..>}[rd]&& \bullet  & &  % 
  & & &  \ar@{..>}[dr] & % 
    &  & \bullet\ar[rd] & &_{I_0}\\
%%%
& & _{P_4}\ar[ru]\ar[rd] & & \bullet &  % 
 & \ar@{..>}[ru] &   & \ar@{..>}[ru] &   & & % 
&&\bullet\ar[rd] \ar@{..>}[ru] && \bullet % 
 && \dots &_{I_4}\ar[ru] \\
%%%
& && _{P_0}\ar[rd]\ar[ru] & \dots &   % 
  & \ar@{..>}[rd] & & \ar@{..>}[rd] && \ar@{..>}[rd]  & % 
&&&_{N_4^{\lambda}}\ar[rd]\ar[ru] &  % 
&  & _{I_3}\ar[ru]\ar[rd] &&  \\
%%%
 && _{P_1}\ar[rd]\ar[ru] && \bullet &  % 
\bullet\ar[rd]\ar@{..>}[ru]  && \bullet\ar[rd]\ar@{..>}[ru]  && \bullet\ar[rd]\ar@{..>}[ru]  && \bullet % 
&&_{N_3^{\lambda}}\ar[ur] \ar[dr]\ar@{--}[dd]\ar@{--}[uuu]    & &_{N_3^{\lambda}}\ar@{--}[dd]\ar@{--}[uuu]% 
 &&\dots &_{I_2}\ar[rd] \\
%%%
 & _{P_2}\ar[rd]\ar[ru] && \bullet\ar[ru] & \dots & % 
& \bullet\ar[rd] \ar[ru] && \bullet\ar[rd]\ar[ru]  && \bullet\ar[rd]\ar[ru] &   % 
&&&_{N_2^{\lambda}}\ar[ur] \ar[dr]  & % 
 &  &\bullet\ar[rd]\ar[ru]  && _{I_1}\ar[rd] \\
%%%
_{P_3} \ar[ru]  && \bullet \ar[ru] & \dots && % 
_{M_1}\ar[ru]\ar@{--}[uuu] && _{M_2}\ar[ru] && _{M_3}\ar[ru] && _{M_1}\ar@{--}[uuu] % 
%&&_{N_1^{\lambda}}\ar^2[uu]  % 
&&_{N_1^{\lambda}}\ar[ur] &&_{N_1^{\lambda}} % 
 & & \dots &\bullet\ar[ru] && _{I_0} \\ 
  & \\ 
%%%
 & \Pp & & & & % 
 &  \T_3,  & & \T_2 &&& % 
 && & \T^{\lambda} & & % 
  && \Ip 
}
\end{center}
\caption{AR-quiver of $\mo kQ_{3,2}$: $\Pp$ and $\Ip$ and tubes}
\label{fig:AR-3-2}
\end{figure}

%%%%%%%%%%%%%%% NEW SECTION %%%%%%%%%%%%%%%%%%%%%

\section{A truncated translation quiver}\label{sec:bruestle}
%%%%%%%%%%%%%%%%%%%%%%%%%%%%%%%%%%%%

In \cite{br}, Br\"ustle 
defines the Auslander Reiten quiver of 
the module category $\mo kQ_{g,h}$ (the quiver $Q_{g,h}$ is called $K$ in 
\cite{br}). 
Br\"ustle then defines subquivers $Q^{\ast}$ of the AR-quiver 
formed by the 
vertices of the preprojective indecomposables (quiver $Q^P$), 
for the preinjective indecomposables (quiver $Q^I$)
and for the vertices in the tubes (quivers $Q^{\lambda}$) for 
$\ast$ appropriate (preprojective, preinjective, tubes) (Sections 2-4 of \cite{br}). 
We will now recall his notation for the vertices of the AR-quiver. The vertices 
all lie in lattices of cylindrical shape. These cylinders are cut open 
to draw the lattices in the plane, the components $Q^P$ and $Q^I$ 
are cylinders lying on their sides, bounded on the left by the projectives or 
on the right by the injectives. The tubes are oriented vertically, with 
$Q^{\infty}$ pointing downwards, all component are viewed in the plane. 
\begin{itemize}
\item 
The vertices of $Q^P$ are the $(r,i)_P$ with $r\ge 0$, $0\le i\le n$. 
The $(0,i)_P$ form the slice of projectives, 
in $Q^P$, the entry $r$ increases to the right. 
The one-dimensional projective indecomposable module $P_0$ gets the 
notation $(0,0)_P$, the largest projective indecomposable $P_g$ is $(0,g)_P$. \\
There are arrows $(r,\beta_i):(r,i+1)_P\to (r,i)_P$ and $(r,\beta_i'):(r,i)\to (r+1,i+1)$ 
for $i=0,\dots, g-1$, as well as arrows 
$(r,\alpha_i):(r,i)\to (r,i+1)$ and $(r,\alpha_i'):(r,i+1)\to (r+1,i)$ for 
$g\le i\le n$. For $0<i<g-1$ this looks as 
follows 
%Arrows between these vertices are of the form 
$$
\xymatrix@R=+0.8pc @C=+0.5pc{
 & (r,i-1)_P\ar[rd]^{(r,\beta_{i-1}')}   &  \\ 
(r,i)_P \ar[ru]^{(r,\beta_{i-1})} \ar[rd]_{(r,\beta_i')}\ar@{..}[rr]  & & (r+1,i)_P \\
 & (r,i+1)_P \ar[ru]_{(r+1,\beta_i)}
}
$$
where the second entry is taken modulo $n+1$. 
The dashed line indicates the AR translation. 
\item 
The underlying graphs of $Q^P$ and $Q^I$ can be mapped one into the other by a mirror 
symmetry along a vertical line (for the mirror axis $S$ see Figure~\ref{fig:quiver-again}). 
The vertices of $Q^I$ are the 
$(r,i)_I$ with $r\ge 0$, $0\le i\le n$ with $r$ increasing to the left. 
The $(0,i)_I$ form the slice of injective indecomposables. 
By the symmetry, for every arrow $\gamma_P:(r_1,s_1)_P\to (r_2,s_2)_P$ in $Q^P$, there is an arrow 
$\gamma_I:(r_2,s_2)_I\to (r_1,s_1)_I$ in $Q^I$. 
For $0<i<g-1$, we get subquivers 
$$
\xymatrix@R=+0.6pc @C=1pc{
 & (r,i-1)_I\ar[rd]^{(r,\beta_{i-1})}   &  \\ 
(r+1,i)_I \ar[ru]^{(r,\beta_{i-1}')} \ar[rd]_{(r,\beta_i)}\ar@{..}[rr]  & & (r,i)_I \\
 & (r,i+1)_I \ar[ru]_{(r,\beta_i')}
}
$$ 
with $k\ge 1$ and where the second entry is taken modulo $n+1$. 

\item 
The vertices of $Q^0$ are the 
$(r,s)_g$ with $r\ge 0$, $1\le s\le g$. The vertices at the mouth of $\T_g$ 
are $(0,s)_g$. The arrows in $Q^0$ are 
$\pi(r,s)=\pi(r,s)_0:(r+1,s+1)_g\to (r,s)_g$ and $\rho(r,s)=\rho_0(r,s):(r,s)_g\to (r+1,s)_g$ 
for $r\ge 0$. For $r>0$, this gives 
$$
\xymatrix@R=+.8pc{ 
 & (r+1,s)_g\ar[rd]^{\pi(r,s-1)}   &  \\ 
(r,s)_g \ar[ru]^{\rho(r,s)} \ar[rd]_{\pi(r-,s-1)}\ar@{..}[rr]  & & (r,s-1)_g &&  \\
 & (r-1,s-1)_g \ar[ru]_{\rho(r-1,s-1)}
}
$$ 
with the second entry in reduced modulo $g$. 
For $r=0$, this becomes \\
\xymatrix@R=+0.6pc @C=+0.3pc{
&&& (1,s+1)_g\ar[rd] \\ 
&& (0,s+1)_g\ar[ru] \ar@{..}[rr]  && (0,s)_g 
}
This is chosen such that the vertex $(0,g)_g$ is the $h+1$-dimensional quasi-simple indecomposable 
with a non-zero morphism from $P_g=(0,g)_P$ and with a non-zero morphism to $I_0=(0,h)_I$. 
\item 
The vertices of $Q^{\infty}$ are the 
$(r,s)_h$, $r\ge 0$, $0\le s\le h-1$. The vertices at the mouth of $\T_h$ 
are the $(0,s)_h$. 
The arrows in $Q^{\infty}$ are $\pi(r,s)=\pi_{\infty}(r,s):(r+1,s-1)_h\to (r,s)_h$ 
and $\rho(r,s)=\rho_{\infty}(r,s):(r,s)_h\to (r+1,s)_h$ for $r\ge 0$. For $r>0$, we have 
$$
\xymatrix@R=+0.8pc{ 
 & (r-1,s)_h\ar[rd]^{\rho(r-1,s)}  &  \\ 
(r,s-1)_h \ar[ru]^{\pi(r-1,s)}  \ar[rd]_{\rho(r,s-1)}\ar@{..}[rr]  & & (r,s)_h    \\
 & (r+1,s-1)_h \ar[ru]_{\pi(r,s)}  & 
}
$$ 
with $s+1=0$ if $s=h-1$. 
For $r=0$, this becomes \\
\xymatrix@R=+0.6pc @C=+0.3pc{
(0,s)_h\ar[rd] \ar@{..}[rr]  && (0,s+1)_h \\
& (1,s)_h\ar[ru]
}
This is chosen such that the vertex $(0,0)_h$ is the $g+1$-dimensional quasi-simple indecomposable 
with a non-zero morphism from $P_g=(0,g)_P$ and with a non-zero morphism to $I_0=(0,h)_I$. 
\item 
The vertices of $Q^{\lambda}$ ($\lambda\in k\setminus\{0\}$) are the 
$(\lambda,i)_{\lambda}$, $i\ge 1$. 
The vertex at the mouth of $Q^{\lambda}$ is $(\lambda,1)_{\lambda}$. 
The arrows in $Q^{\lambda}$ are as follows (for $s\ge 2$): 
$$
\xymatrix@R=+0.6pc @C=+0.3pc{
 & (\lambda,s+1)_{\lambda}\ar[rd]  &  &&& (\lambda,2)_{\lambda}\ar[rd] \\ 
(\lambda,s)_{\lambda} \ar[ru] \ar[rd]\ar@{..}[rr]  && (\lambda,s)_{\lambda} % 
 && (\lambda,1)_{\lambda}  \ar[ru] \ar@{..}[rr]  && (\lambda,1)_{\lambda}  \\
 & (\lambda,s-1)_{\lambda} \ar[ru] 
}
$$ 
\end{itemize}

The union of all these vertices gives all the isomorphism classes of indecomposable 
$kQ_{g,h}$-modules, cf. \cite[Proposition \S 1]{br} or \cite{gr}. 

In what follows, we will write $\Pp$ to denote the subcategory of $\mo kQ_{g,h}$ 
whose objects are the 
preprojective indecomposable modules, $\Ip$ for the subcategory of the preinjective 
indecomposable modules, $\T^0=\T_g$ and $\T^{\infty}=\T_h$ (the former versions 
are used in Br\"ustle, the latter are more convenient for our geometric model) 
for the subcategories formed by the 
objects in the tubes of rank $g$ and $h$ respectively and $\T^{\lambda}$ for the 
objects forming the homogenous tubes ($\lambda\in k\setminus\{0\}$). 

On the AR-quiver as described above, 
one imposes the mesh relations for any diamond (quadrilateral formed by 
four neighbours) in it and for all the triangles appearing at mouths of tubes, as indicated 
by the dotted lines. 

Figure~\ref{fig:quiver-again} below illustrates this AR-quiver. 

\medskip

Br\"ustle then truncates AR$\mo kQ_{g,h}$ to define quivers $Q_m$, for $m\ge 1$. 
This is obtained by cutting the preprojective component $Q^P$ at the 
vertices $(ghm,i)$, so that the resulting full subquiver $Q_m^P$ of $Q^P$ 
has $ghm+1$ slices, starting 
from the vertices of the projective indecomposables, 
(see Section 2 of~\cite{br}). 
Similarly, $Q^I$ is cut in order to get a full quiver $Q_m^I$ of $Q^I$, 
containing $ghm+1$ 
slices, ending at the injectives, 
(see Section 3 of~\cite{br}). 
The regular tubes are cut to form quivers that 
have cylindrical shapes with a cone on the top: 
$Q_{m}^0$ denotes the full subquiver with lowest row given by the 
$(0,i)_g$, $0\le i\le g$, and top row the single vertex $(gm+g,g)_g$. 
Similarly, 
$Q_m^{\infty}$ is the full subquiver with lowest row containing the 
$(0,i)_h$, $0\le i\le h$, and top vertex $(hm+h,h)$. 
The components $Q^{\lambda}$ are cut get the full subquivers 
$Q_m^{\lambda}$ on 
the $m+1$ vertices $(\lambda,1),\dots, (\lambda,m+2)$. 
For the tubes see Section 4 of~\cite{br}. 

Furthemore, let $\Pp_m$ be the full subcategory of $\Pp$ on the objects corresponding to 
the vertices of $Q_m^P$ 
define the subcategories $\Ip_m$, $\T^0_m$, $\T^{\infty}_m$ and $T^{\lambda}_m$ 
similarly. 
Br\"ustle proves the following statements (Sections 2,3,4 in \cite{br}): 
 
\begin{prop}
There are isomorphisms between the $k$-categories $\mathcal C(Q_m^{\ast})$, 
$\ast \in \{P,I,0,\infty\}$ or $\ast=\lambda\in k\setminus\{0\}$, 
defined by $Q_m^{\ast}$ and the corresponding categories  
$\Pp_m$, $\Ip_m$, $\T_m^0$, $\T_m^{\infty}$, $\T_m^{\lambda}$. 
\end{prop} 

To describe the AR-quiver of a truncated version of $\mo\tilde{A}_n$, we then have to define 
$Q_m$ from the truncated pieces of correct size (cf. Section 5 in \cite{br}): 

\begin{dfn}
For $m\ge 1$ let $Q_m$ be the quiver whose vertices are the vertices of $Q_m^P$, of $Q_m^I$ and 
of $Q_{2m(n+1)}^{\lambda}$ for $\lambda\in k\cup\{\infty\}$, the arrows are all the arrows 
of these quivers with additional ``connecting'' arrows $\iota_0(x)$, $\kappa_0(x)$ for 
$x=0,\dots, g$, $\iota_{\infty}(y)$, $\iota_{\infty}(y)$ for $y=0,g,g+1,\dots, n$ 
and $\iota_{\lambda}$, $\kappa_{\lambda}$ for $\lambda\in k\setminus\{0\}$. 
\end{dfn}

For illustration we include a picture of $Q_m$ for the running example $(g=3,h=1)$, 
see Figure~\ref{fig:quiver-again}. It is taken from \cite{br}. 

%\newpage
\begin{figure}[htp]
\begin{center}
\psfragscanon
\psfrag{QI}{$Q_m^I$}
\psfrag{QP}{$Q_m^P$}
\psfrag{Q0}{$Q_{2m(n+1)}^{0}$}
\psfrag{QL}{$Q_{2m(n+1)}^{\lambda}$}
\psfrag{Q-inf}{$Q_{2m(n+1)}^{\infty}$}
\psfrag{i00}{\tiny $_{\iota_0(0)}$}
\psfrag{i01}{\tiny $_{\iota_0(1)}$}
\psfrag{i02}{\tiny $_{\iota_0(2)}$}
\psfrag{i03}{\tiny $_{\iota_0(3)}$}
\psfrag{k00}{\tiny $_{\kappa_0(0)}$}
\psfrag{k01}{\tiny $_{\kappa_0(1)}$}
\psfrag{k02}{\tiny $_{\kappa_0(2)}$}
\psfrag{k03}{\tiny $_{\kappa_0(3)}$}
\psfrag{ki0}{\tiny $_{\kappa_{\infty}(0)}$}
\psfrag{ki3}{\tiny $_{\kappa_{\infty}(3)}$}
\psfrag{ki4}{\tiny $_{\kappa_{\infty}(4)}$}
\psfrag{ii0}{\tiny $_{\iota_{\infty}(0)}$}
\psfrag{ii3}{\tiny $_{\iota_{\infty}(3)}$}
\psfrag{ii4}{\tiny $_{\iota_{\infty}(4)}$}
\psfrag{00}{\tiny $_{(0,0)}$}
\psfrag{01}{\tiny $_{(0,1)}$}
\psfrag{02}{\tiny $_{(0,2)}$}
\psfrag{03}{\tiny $_{(0,3)}$}
\psfrag{04}{\tiny $_{(0,4)}$}
\psfrag{ri}{\tiny $_{\rho_{\infty}}$}
\psfrag{pi}{\tiny $_{\pi_{\infty}}$}
\psfrag{r0}{\tiny $_{\rho_{0}}$}
\psfrag{p0}{\tiny $_{\pi_{0}}$}
\psfrag{il}{\tiny $_{\iota_{\lambda}}$}
\psfrag{kl}{\tiny $_{\kappa_{\lambda}}$}
\psfrag{pl1}{\tiny $_{\pi_{\lambda}(1)}$}
\psfrag{rl1}{\tiny $_{\rho_{\lambda}(1)}$}
\psfrag{20m0}{\tiny $_{(20m,0)}$}
\psfrag{6m3}{\tiny $_{(6m,3)}$}
\psfrag{6m0}{\tiny $_{(6m,0)}$}
\psfrag{30m30}{\tiny $_{(30m,3)_0}$}
\psfrag{30m3I}{\tiny $_{(30m,3)_I}$}
\psfrag{6m3I}{\tiny $_{(6m,3)_I}$}
\psfrag{S}{$S$}
\includegraphics[width=11.5cm]{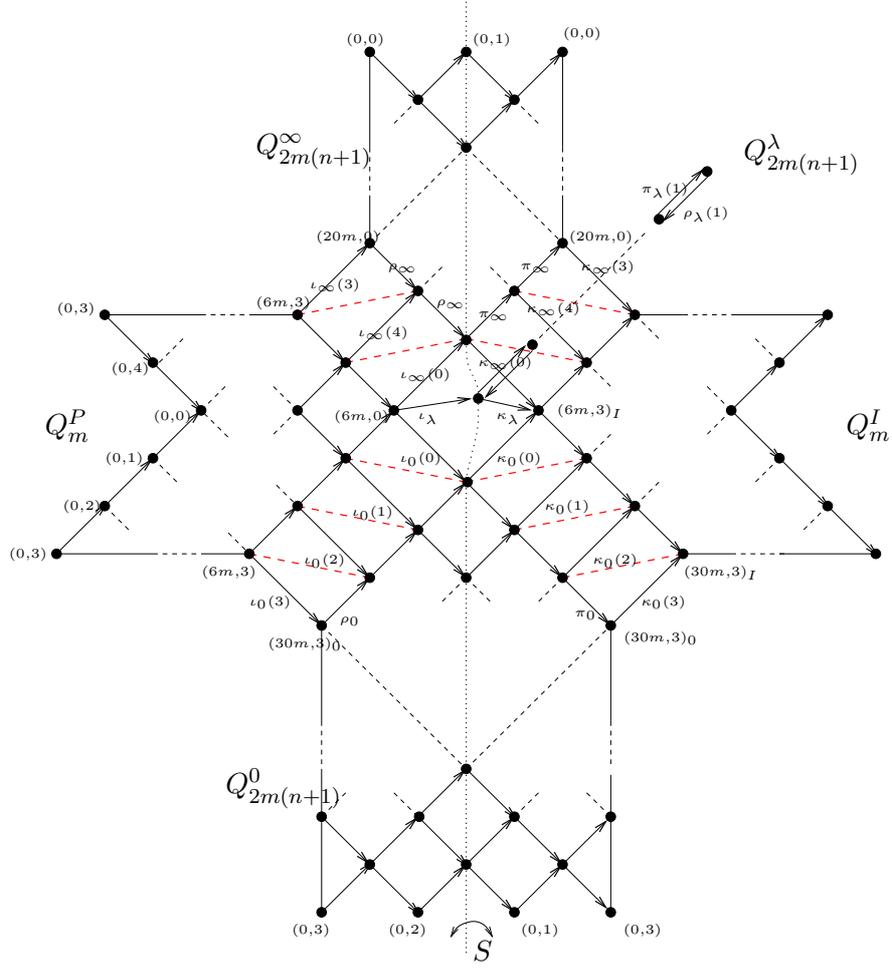}
\caption{$Q_m$ for $g=3$, $n=4$} % with $ghm=6m$, $g2m(n+1)=30m$, $h2m(n+1)=20m$}
\label{fig:quiver-again}
\end{center}
\end{figure}

\begin{dfn}
Let $m\ge 1$. 
By $\mathcal J_mQ_{g,h}$ we denote the full subcategory of $\mo\tilde{A}_n$ whose 
objects are the union of the objects of 
$\Pp_m$, $\Ip_m$, $\T_{2m(n+1)}^{\sigma}$, $\sigma\in k\cup\{\infty\}$. 
\end{dfn}

Note that for every object $M$ of $\mo\tilde{A}$ there exists $m\ge 1$ such that $M$ is isomorphic 
to a direct sum of modules from $\mathcal J_mQ_{g,h}$ 
and hence the categories $\mathcal J_m Q_{g,h}$ can be used to describe $\mo kQ_{g,h}$.

\begin{comment}
  \begin{figure}[htp]
  \begin{center}
\psfrag{QI}{$Q_m^I$}
\psfrag{QP}{$Q_m^P$}
\psfrag{Q0}{$Q_{2m(n+1)}^{0}$}
\psfrag{QL}{$Q_{2m(n+1)}^{\lambda}$}
\psfrag{Q-inf}{$Q_{2m(n+1)}^{\infty}$}
\psfrag{i00}{\tiny $_{\iota_0(0)}$}
\psfrag{i01}{\tiny $_{\iota_0(1)}$}
\psfrag{i02}{\tiny $_{\iota_0(2)}$}
\psfrag{i03}{\tiny $_{\iota_0(3)}$}
\psfrag{k00}{\tiny $_{\kappa_0(0)}$}
\psfrag{k01}{\tiny $_{\kappa_0(1)}$}
\psfrag{k02}{\tiny $_{\kappa_0(2)}$}
\psfrag{k03}{\tiny $_{\kappa_0(3)}$}
\psfrag{ki0}{\tiny $_{\kappa_{\infty}(0)}$}
\psfrag{ki3}{\tiny $_{\kappa_{\infty}(3)}$}
\psfrag{ki4}{\tiny $_{\kappa_{\infty}(4)}$}
\psfrag{ii0}{\tiny $_{\iota_{\infty}(0)}$}
\psfrag{ii3}{\tiny $_{\iota_{\infty}(3)}$}
\psfrag{ii4}{\tiny $_{\iota_{\infty}(4)}$}
\psfrag{00}{\tiny $_{(0,0)}$}
\psfrag{01}{\tiny $_{(0,1)}$}
\psfrag{02}{\tiny $_{(0,2)}$}
\psfrag{03}{\tiny $_{(0,3)}$}
\psfrag{04}{\tiny $_{(0,4)}$}
\psfrag{ri}{\tiny $_{\rho_{\infty}}$}
\psfrag{pi}{\tiny $_{\pi_{\infty}}$}
\psfrag{r0}{\tiny $_{\rho_{0}}$}
\psfrag{p0}{\tiny $_{\pi_{0}}$}
\psfrag{il}{\tiny $_{\iota_{\lambda}}$}
\psfrag{kl}{\tiny $_{\kappa_{\lambda}}$}
\psfrag{pl1}{\tiny $_{\pi_{\lambda}(1)}$}
\psfrag{rl1}{\tiny $_{\rho_{\lambda}(1)}$}
\psfrag{20m0}{\tiny $_{(20m,0)}$}
\psfrag{6m3}{\tiny $_{(6m,3)}$}
\psfrag{6m0}{\tiny $_{(6m,0)}$}
\psfrag{30m30}{\tiny $_{(30m,3)_0}$}
\psfrag{30m3I}{\tiny $_{(30m,3)_I}$}
\psfrag{6m3I}{\tiny $_{(6m,3)_I}$} 
     \includegraphics[width=11.5cm]{Qm-relations.eps}
  \end{center}\caption{\label{figcompzero2} mmmmmm}   
  \end{figure}
\end{comment}

\newpage

We need a few abbreviations to be able to write the relations in Theorem~\ref{thm:hauptsatz}: 
\begin{center}
\psfragscanon
\psfrag{QI}{$Q_m^I$}
\psfrag{QP}{$Q_m^P$}
\psfrag{Q0}{$Q_{2m(n+1)}^{0}$}
\psfrag{Q-inf}{$Q_{2m(n+1)}^{\infty}$}
\psfrag{ghm0P}{\tiny $_{(ghm,0)_P}$}
\psfrag{ghm0I}{\tiny $_{(ghm,0)_I}$}
\psfrag{i00}{\tiny $_{\iota_0(0)}$}
\psfrag{i0g}{\tiny $_{\iota_0(g)}$}
\psfrag{k00}{\tiny $_{\kappa_0(0)}$}
\psfrag{k0g}{\tiny $_{\kappa_0(g)}$}
\psfrag{ki0}{\tiny $_{\kappa_{\infty}(0)}$}
\psfrag{kig}{\tiny $_{\kappa_{\infty}(g)}$}
\psfrag{ii0}{\tiny $_{\iota_{\infty}(0)}$}
\psfrag{iig}{\tiny $_{\iota_{\infty}(g)}$}
\psfrag{ghmb}{\tiny $(ghm,\beta_{g-1})$}
\psfrag{ghmgP}{\tiny $(ghm,g)_P$}
\psfrag{h2m0}{\tiny $(2hm(n+1),0)_h$}
\psfrag{ghmgI}{\tiny $(ghm,g)_I$}
\psfrag{g2mg0}{\tiny $(2gm(n+1),g)_g$}
\psfrag{bgP}{\tiny $\beta_P^g$}
\psfrag{ahP}{\tiny $\alpha_P^h$}
\psfrag{rh}{\tiny $\rho_{\infty}^h$}
\psfrag{ph}{\tiny $\pi_{\infty}^h$}
\psfrag{bhI}{\tiny $\beta_I^h$}
\psfrag{agI}{\tiny $\alpha_I^g$}
\psfrag{pg}{\tiny $\pi_0^g$}
\psfrag{rg}{\tiny $\rho_0^g$}
\includegraphics[width=10cm]{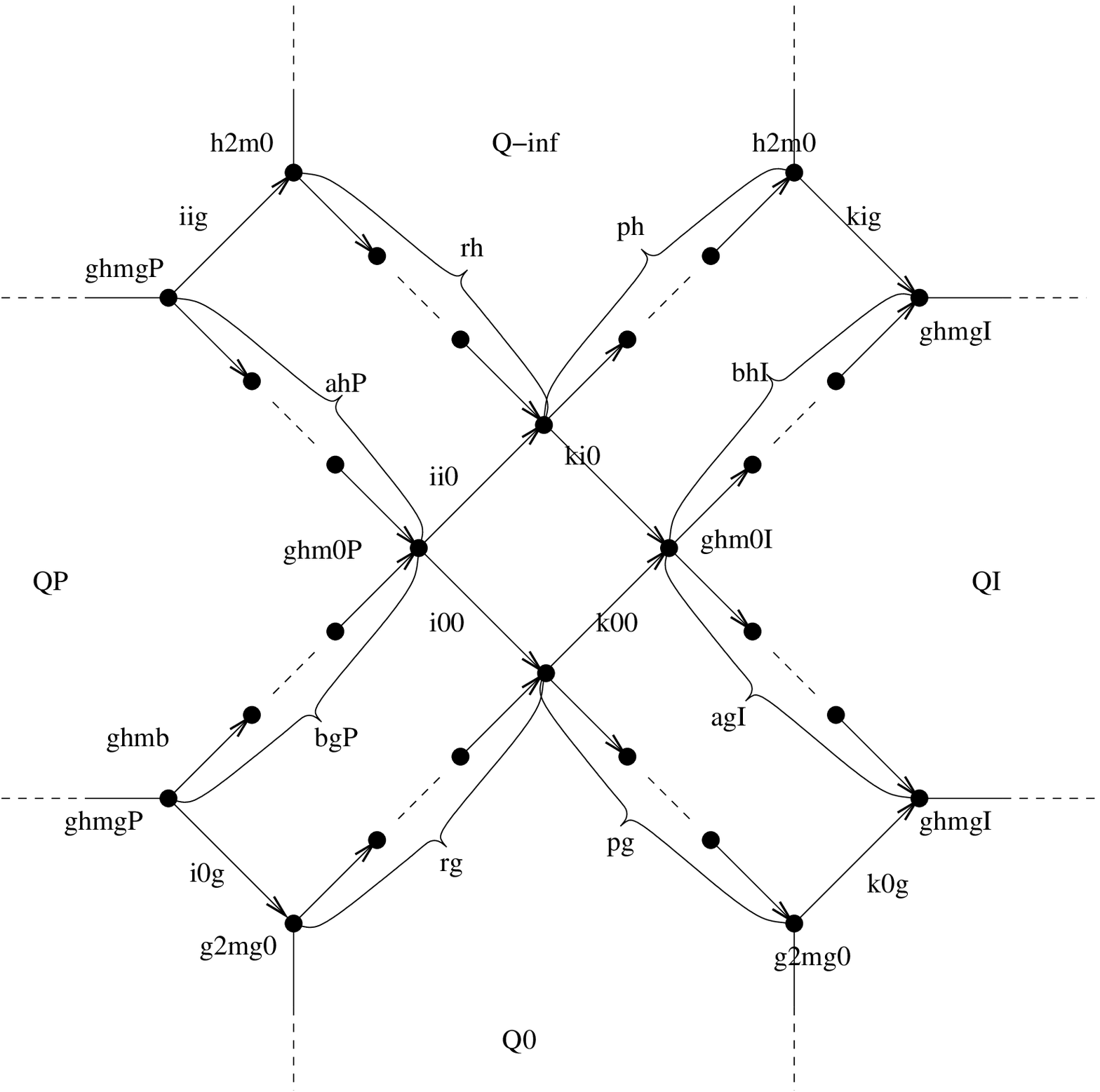}
\end{center}

\noindent
$\beta_P^g$ denotes the path of length $g$ from $(ghm,g)_P$ to $(ghm,0)_P$ 
composed by the arrows $(ghm,\beta_i)_P$, $i=g-1,\dots, 1,0$. 

\noindent
$\alpha_P^h$ denotes the path of length $h$ from $(ghm,g)_P$ to $(ghm,0)_P$ 
composed by the arrows $(ghm,\alpha_i)_P$, $i=g,g+1,\dots, n$.

\noindent
The definitions of $\alpha_I^g$, $\beta_I^g$, $\pi_0^g$, $\rho_0^g$, $\pi_{\infty}^h$ 
and $\rho_{\infty}^h$ can be understood from 
the picture above. 
Finally, we write \\
$\varepsilon_{\lambda}=\rho_{\lambda}(2m(n+1)+1)\pi_{\lambda}(2m(n+1)+1)$ for $\lambda\in k\setminus\{0\}$.

\begin{thm}[Hauptsatz]\label{thm:hauptsatz}
Let $\mathcal C(kQ_m)$ be the $k$-category  
category of $Q_m$ 
%$k$-category defined by $Q_m$ 
subject to the following relations 
\begin{itemize}
\item[(a)] 
$\gamma'\gamma=\delta'\delta$ for all arrows $\gamma$, $\gamma'$, $\delta$, $\delta'$ of $Q_m$  
in a diamond from $X$ to $Z$ with $X\ne (ghm,0)_P$. 

\vskip 5pt

\item[(b)] 
$\pi_{\lambda}\rho_{\lambda}=0 $ for all $\lambda\in k\cup \{\infty\}$ and all arrows  
$\pi_{\lambda}$, $\rho_{\lambda}$ of $Q_{2m(n+1)}^{\lambda}$ 
of the form 
$X^*\stackrel{\rho_{\lambda}}{\longrightarrow} Y\stackrel{\pi_{\lambda}}{\longrightarrow} Z^*$ 
for $X^*$ and $Z^*$ vertices at the mouth of a tube (possibly, $X^*=Z^*$). 

\vskip 5pt

\item[(c1)]
$\iota_0(g) = \pi_0^g\iota_0(0)\alpha_P^h$ \hskip 60pt 
$\iota_{\infty}(g) = \pi_{\infty}^h\iota_{\infty}(0)\beta_P^g$. 

\vskip 5pt

\item[(c2)] 
$\kappa_0(g) = \beta_I^h\kappa_0(0)\rho_0^g$ \hskip 57pt 
$\kappa_{\infty}(g) = \alpha_I^g \kappa_{\infty}(0)\rho_{\infty}^h$. 

\vskip 5pt

\item[(d)]
$\iota_{\lambda}\alpha_P^h = \lambda\iota_{\lambda}\beta_P^g 
 + \varepsilon_{\lambda}\iota_{\lambda}\beta_P^g$ \hskip 37pt
$\alpha_I^h\kappa_{\lambda} = \lambda\beta_I^g\kappa_{\lambda} 
 + \beta_I^g\kappa_{\lambda}\varepsilon_{\lambda}$. 

\vskip 5pt

\item[(e)]
$\kappa_0(0)\iota_0(0)(ghm,\alpha_n)_P=0$ \hskip 22pt $(ghm,\alpha_n)_I\kappa_0(0)\iota_0(0)=0$.  

\vskip 5pt

\item[(f)]
$\kappa_{\infty}(0)(\rho_{\infty}^h\pi_{\infty}^h)^j\iota_{\infty}(0) = 
\kappa_0(0)(\rho_0^g\pi_0^g)^{2m(n+1)+1-j}\iota_0(0)$. 

\vskip 5pt

\item[(g)] 
$\kappa_{\lambda}\varepsilon_{\lambda}^j\iota_{\lambda} = 
\sum_{i=0}^{j}{2m(n+1)+1-j+i \choose 2m(n+1)+1-j}\lambda^i
\kappa_0(0)(\rho_0^g\pi_0^g)^{j-i}\iota_0(0)$
\end{itemize}

\vskip 5pt

\noindent
with $\lambda\in k\setminus \{0\}$ and $j=0,\dots, 2m(n+1)+1$. \\
Then there is an isomorphism % $\Phi_m$ induces an isomorphism 
$$
\mathcal C(kQ_m)\to \mathcal{J}_m Q_{g,h}
$$
\end{thm}

%%%%%%%%%%%%%%% NEW SubSECTION %%%%%%%%%%%%%%%%%%%%%

\begin{rem}
We observe that the relations also imply 
$$ 
(e')\  \kappa_{\infty}(0)\iota_{\infty}(0)(ghm,\beta_0)_P=0 
\hskip 10pt 
(ghm,\beta_0)_I\kappa_{\infty}(0)\iota_{\infty}(0)=0
$$
To see this, one uses (e) and (f) (with $j=1$) and the diamond relations from (a) to push the path all the way down to 
the mouth of the tube, where it will include a triangle as in (b), hence the zero relation. 
\end{rem}

%%%%%%%%%%%%%%%%%%%%%%%%%%%%%%%
%
\section{A translation quiver on arcs in the annulus}\label{sec:arcs-quiver}
%The geometric interpretations of maps and relations}
%%%%%%%%%%%%%%%%%%%%%%%%%%%%%%%%%

We now describe the set-up of the geometric model we use to describe 
$\mo\tilde{A}_n$ and later for the cluster category $\mathcal C_{\tilde{A}_n}$. It 
is similar to the one appearing in \cite{w} and in \cite{bm}. Our main focus is 
on the interpretation of the infinite radical of the module category resp. the cluster category in 
type $\widetilde{A}_n$, in terms of the geometry. 

Let $P_{g,h}$ be an annulus with $g$ marked points on the outer boundary $\partial$ and 
$h$ marked points on the inner boundary $\partial'$, $gh\ne 0$, let $g\ge h$. .

We assume that the marked 
points are distributed in equidistance on the two boundaries. 
We will identify $P_{g,h}$ with a cylinder $\Cyl_{g,h}$ of height $1$ with $g$ marked 
points on the lower 
boundary and $h$ marked points on the upper boundary. We can view this cylinder 
as a rectangle of height $1$ and width $gh$ in $\mathbb{R}^2$, identifying its two vertical 
sides. We will always label the marked points of $\Cyl_{g,h}$ from left to right on the 
lower boundary and from right to left on the upper boundary. 

\begin{figure}
\psfragscanon
\psfrag{0}{\tiny$(0,0)$}%{\tiny$O_0$}
\psfrag{00}{\tiny$(0,1)$}%{\tiny$I_0$}
\psfrag{a}{\tiny$(h,0)$}%{\tiny$O_1$}
\psfrag{b}{\tiny$(2h,0)$}%{\tiny$O_2$}
\psfrag{g}{\tiny$(gh-h,0)$}%{\tiny$O_{g-1}$}
\psfrag{dd}{\tiny$\dots$}
%\psfrag{}{\tiny$(\frac{}{},0)$}
\psfrag{m}{\tiny$(g,1)$}%{\tiny$I_1$}
\psfrag{r}{\tiny$(gh-g,1)$}%{\tiny$I_{h-1}$}
\includegraphics[scale=.45]{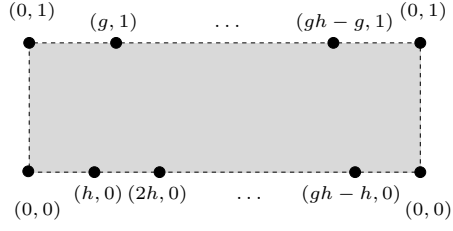}
\caption{Annulus  via rectangle $\Cyl_{g,h}$}\label{fig:arcsPgh}
\end{figure}

On the lower boundary (the outer boundary of the annulus), we choose the points 
$(0,0)$, $(h,0)$, $\dots$, $(gh-h,0)$ 
and $(gh,0)$ as marked points. 
On the upper boundary (the inner boundary of the 
annulus), we choose the points $(0,1)$, $(g,1)$, $\dots, (gh-g,1)$ and $(gh,1)$ 
to be marked points, as in Figure~\ref{fig:arcsPgh}. 

With the annulus in mind, we denote the points on the lower boundary by $i_{\partial}$, 
for $0\le i \le g-1$: 
$$
0_{\partial}:=(0,0)=(gh,0), 1_{\partial}=(h,1), \dots, (g-1)_{\partial}:=(hg-h,0)
$$ 
and the points on the upper boundary by $j_{\partial'}$, for $0\le j\le h-1$: 
$$
0_{\partial'}:=(gh,1)=(0,1), 1_{\partial'}:=(g,1), 
\dots, (h-1)_{\partial'}:=(g(h-1),1),
$$ 
We thus label marked points on both boundaries from the left to the right, by $\mathbb{Z}_b$ 
with $b\in \{\partial,\partial'\}$.

%In case $g=3$, $h=2$, this gives: 
%\begin{center}
%\psfragscanon
%\psfrag{0}{\small $0_{\partial}$}%{\tiny$O_0$}
%\psfrag{00}{\small $0_{\partial'}$}%{\tiny$I_0$}
%%
%\psfrag{a}{\small $2_{\partial}$}%{\tiny$O_2$}
%\psfrag{b}{\small $1_{\partial}$}%{\tiny$O_1$}
%\psfrag{c}{\small $0_{\partial}$}%{\tiny$O_0$}
%\psfrag{m}{\small $1_{\partial'}$}%{\tiny$I_1$}
%\psfrag{n}{\small $0_{\partial'}$}%{\tiny$I_0$}
%
%\includegraphics[scale=.4]{cylinder.eps}
%\end{center}

For our purposes, it will be most 
convenient to work in the universal cover $\U=(\U,\pi_{gh})$ of $\Cyl_{g,h}$ 
with $\U=\{(x,y)\in \mathbb{R}^2\mid 0 \le y\le 1\}$ an infinite strip in the plane. 
It inherits the orientation from its embedding in $\mathbb{R}^2$. The covering map 
$\pi_{gh}:\U\to \Cyl_{g,h}$ is induced from wrapping 
$\U$ around $\Cyl_{g,h}$, it takes the first entry of $(x,y)\in \U$ modulo $gh$: 
$$
\pi:=\pi_{gh}:\U\to \Cyl_{g,h}, \quad\quad (x,y)\mapsto (x\mod hg, y). 
$$

As we can identify the annulus $P_{g,h}$ with $\Cyl_{g,h}$, 
$\pi$ is also a covering map of $P_{g,h}$. 

We take as marked points on the lower boundary the points 
$\{(hx,0)\mid x\in \mathbb{Z}\}$ and as marked points on the upper boundary of 
$\U$ the points $\{(gx,1)\mid x\in \mathbb{Z}\}$, see Figure~\ref{fig:cover}.

When working in the universal cover, it is most convenient to 
use integers to denote marked points on the lower and upper boundary. 
We will write subscripts to indicate the boundary on which the points sit. 
So for $i,j\in\mathbb{Z}$, $i_{\partial}$ is a marked point on the lower boundary 
and $j_{\partial'}$ a marked point on the upper boundary. 
We do this in such a 
way that the copy of $\Cyl_{g,h}$ with vertices $(0,0)$ and $(gh,0)$ obtains 
the labels $0_{\partial}, 1_{\partial}, \dots, g_{\partial}$ on the lower boundary 
and the labels $0_{\partial'}, 1_{\partial'}, \dots, h_{\partial'}$ on the upper 
boundary, as shown for the case $g=3,h=2$ in Figure~\ref{fig:cover}. More 
generally, for fixed $g,h$, with $g\ge h\ge 1$, we endow $\U$ with the following set of marked points 
$$
\begin{array}{c}
\{i_{\partial} =  (ih,0) \mid i\in\mathbb{Z}\} \\ 
\{j_{\partial'} =  (jg,1)\mid j\in\mathbb{Z}\} 
\end{array}
$$
placing $0_{\partial'}$above $0_{\partial}$, $h_{\partial'}$ above 
$g_{\partial}$ etc.

\begin{figure}[h]
\psfragscanon
\psfrag{0}{\tiny$(6,0)$}
\psfrag{00}{\tiny$(6,1)$}
\psfrag{x}{$\cdots$}
\psfrag{a}{\tiny$(4,0)$}
\psfrag{b}{\tiny$(2,0)$}
\psfrag{c}{\tiny$(0,0)$}
\psfrag{d}{\tiny$(-2,0)$}
\psfrag{e}{\tiny$(-4,0)$}
\psfrag{f}{\tiny$(-6,0)$}
\psfrag{g}{\tiny$(8,0)$}
\psfrag{h}{\tiny$(10,0)$}
\psfrag{i}{\tiny$(12,0)$}
%\psfrag{}{\tiny$(\frac{}{},0)$}
\psfrag{m}{\tiny$(3,1)$}
\psfrag{n}{\tiny$(0,1)$}
\psfrag{o}{\tiny$(-3,1)$}
\psfrag{p}{\tiny$(-6,1)$}
\psfrag{r}{\tiny$(9,1)$}
\psfrag{s}{\tiny$(12,1)$}
\psfrag{d0}{\color{blue}\small $0_{\partial}$}
\psfrag{d1}{\color{blue}\small $1_{\partial}$}
\psfrag{d2}{\color{blue}\small $2_{\partial}$}
\psfrag{d3}{\color{blue}\small $3_{\partial}$}
\psfrag{d4}{\color{blue}\small $4_{\partial}$}
\psfrag{d5}{\color{blue}\small $5_{\partial}$}
\psfrag{d6}{\color{blue}\small $6_{\partial}$}
\psfrag{d-1}{\color{blue}\small $-1_{\partial}$}
\psfrag{d-2}{\color{blue}\small $-2_{\partial}$}
\psfrag{d-3}{\color{blue}\small $-3_{\partial}$}
\psfrag{f0}{\color{blue}\small $0_{\partial'}$}
\psfrag{f1}{\color{blue}\small $-1_{\partial'}$}
\psfrag{f2}{\color{blue}\small $-2_{\partial'}$}
\psfrag{f-1}{\color{blue}\small $1_{\partial'}$}
\psfrag{f-2}{\color{blue}\small $2_{\partial'}$}
\psfrag{f-3}{\color{blue}\small $3_{\partial'}$}
\psfrag{f-4}{\color{blue}\small $4_{\partial'}$}

\includegraphics[scale=.46]{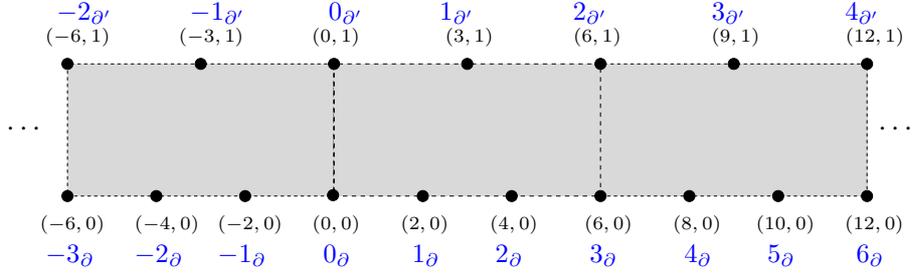}
\caption{Universal cover for $g=3$, $h=2$}\label{fig:cover}
\end{figure}

%%%%%%%%%%%%%%% NEW SUBSECTION %%%%%%%%%%%%%%%%%%%%%
%
\subsection{Arcs in $P_{g,h}$ and in the universal cover} 

We want to consider certain oriented arcs between marked points of the annulus $P_{g,h}$ 
to define a stable translation quiver $\Gamma$. 
We recall 
that our goal is to show that $\Gamma$ is isomorphic to the AR-quiver of the module 
category of type $\tilde{A}$, with underlying quiver $Q_{g,h}$ and to describe the 
infinite radical in terms of moves between arcs. 
Passing to unoriented 
arcs, we will then obtain a quiver which is isomorphic to 
the AR-quiver of the cluster category of 
type $\tilde{A}$ of the same underlying quiver 
$Q_{g,h}$.  

We will call 
the arrows in this quiver \textit{moves}. 
There are two types of moves, namely \textit{elementary moves} and \textit{long moves}. 
We describe these moves below. 
In the oriented case, elementary moves between arcs of an annulus 
have appeared in~\cite{w} and \cite{bm} independently 
to describe (components of) module categories of type $\tilde{A}$. 
In the unoriented case, elementary moves appeared in \cite{bz} for the cluster category 
of a marked surface, in \cite{gehrig} for tubes 
and in \cite{to} for the $m$-cluster category of an annulus. 
%The elementary moves were defined in \cite{bz} for the cluster category and \cite{to} 
%for the $m$-cluster category. 
To our knowledge, long moves of (oriented/unoriented) arcs have not been considered 
so far.

To define an oriented arc of $\U$, we need to give its start and 
end point as well as its winding number around the inner boundary. 
A convenient way to do define such arcs is to work in the universal cover instead. 

We need to recall a few facts before we can do this, cf. \cite[\S 2.2]{bm} in the case 
with marked points on one boundary. 
Let $\sigma:\U\to \U$ be the 
translation $(x,y)\mapsto (x+gh,y)$, with inverse $\sigma^{-1}(x,y)=(x-gh,y)$. The group 
$G=\langle \sigma\rangle$ acts naturally on $\U$. 
Note that in terms of $z_{\partial}$ and $z_{\partial'}$, 
$\sigma$ acts on vertices 
of the lower boundary as $z_{\partial}\mapsto (z+g)_{\partial}$ and on the 
upper boundary as $z_{\partial'}\mapsto (z+h)_{\partial'}$. 

An {\emph{oriented arc} in $\U$} is an isotopy class of arcs joining marked points of the boundary of $P_{g,h}$. 
If an arc starts at a marked point $x_{b_1}$ and ends at a marked point 
$y_{b_2}$ with $x\in\mathbb{Z}$, $b_i\in\{\partial,\partial'\}$, we write $[x_{b_1},y_{b_2}]$ for 
this arc (defined only up to isotopy fixing endpoints). 
%We have 
%\begin{eqnarray*}
%\sigma([x_{\partial},y_{\partial}]) & = & [(x+g)_{\partial}, (y+g)_{\partial}] \\
%\sigma([x_{\partial},y_{\partial'}]) & = & [(x+g)_{\partial}, (y+h)_{\partial'}] \\
%\sigma([x_{\partial'},y_{\partial}]) & = & [(x+h)_{\partial'}, (y+g)_{\partial}] \\
%\sigma([x_{\partial'},y_{\partial'}]) & = & [(x+h)_{\partial'}, (y+h)_{\partial'}] 
%\end{eqnarray*}

We will distinguish two main types of arcs in $\U$ and in $P_{g,h}$.

\begin{dfn}\label{defn:bridging}
Let $\alpha=[x_{b_1},y_{b_2}]$ be an arc in $\U$. If $b_1\ne b_2$, then we say 
that $\alpha$ is a \emph{bridging arc}. \\
If $b_1=b_2$ and $x\le y-2$, we say that $\alpha$ is a \emph{peripheral arc} and that 
$\alpha$ is \emph{based at $b_1$}. 
\end{dfn}

We do not consider arcs in the remaining cases. 
%Note that arcs of type 1 are of the form $[a_{\partial}, b_{\partial'}]$ or 
%$[c_{\partial'},d_{\partial}]$, for $a,b,c,d\in \mathbb{Z}$.
\begin{dfn}
A \emph{bridging} arc in $P_{g,h}$ is an arc $\pi_{g,h}(\alpha)$ where $\alpha$ is a bridging 
arc of $\U$. Similarly, a \emph{peripheral} arc of $P_{g,h}$ is an arc $\pi_{g,h}(\alpha)$ for 
$\alpha$ a peripheral arc of $\U$. 
\end{dfn} 

\begin{rem}
If $\alpha$ is an arc of $P_{g,h}$ starting at $\partial$ (at $\partial'$) it has a unique 
lift $\tilde{\alpha}$ in $\U$ with starting point in $\{0_{\partial}, 1_{\partial}, \dots, (g-1)_{\partial}\}$ 
(with starting point  $\{0_{\partial'}, 1_{\partial'}, \dots, (h-1)_{\partial'}\}$). 
%We will call this lift the {\em canonical lift of $\alpha$}. 

All other lifts of $\alpha$ (of $\beta$) can be obtained by applying $\sigma$ or $\sigma^{-1}$ 
repeatedly to $\tilde{\alpha}$, hence 
$\pi_{g,h}^{-1}(\alpha)=G\tilde{\alpha}$. 
\end{rem}

Bridging arcs will serve as models for the preprojective and preinjective components, 
peripheral arcs correspond to objects of the tubes $\T_g$, $\T_h$, as we will explain in 
Sections~\ref{sec:arcs-in-P} 
and \ref{sec:arcs-in-I}.

%%%%%%%%%%%%%%%%%%%%%%%%%%
%
\subsection{Elementary moves} \label{sec:elem-moves}
%
%%%%%%%%%%%%%%%%%%%%%%%%%%

The geometric idea is that a move rotates an arc clockwise around one of its endpoints. 
For elementary moves, this is a ``shortest'' clockwise rotation. 
We keep the idea of rotating in mind for the long moves. 

Let $g\ge h>0$ be fixed, let $\pi=\pi_{g,h}:\U\to P_{g,h}$ be the covering map. 
Let $\alpha$ be an arc of $P_{g,h}$ and let $[x_{b_1},y_{b_2}]$ be a lift of $\alpha$, 
with $b_1,b_2\in\{\partial, \partial'\}$. 
Since an elementary move gives rise to an irreducible map in the module category, 
it should be a minimal rotation: it fixes 
one endpoint of an arc and moves one endpoint by $\pm 1$ (depending on the boundary 
this sits on).  \\ 
So there are at most two elementary moves from $\alpha$, namely to the 
arcs $\pi([x_{b_1},(y\pm 1)_{b_2}])$ and 
$\pi([(x\pm 1)_{b_1},y_{b_2}])$, where $\pm 1$ is $-1$ on the boundary $\partial$ 
and $+1$ on $\partial'$. 

Note that when drawing vertices and moves between them, 
we will always use the convention that arrows (moves) go to the right. Furthermore, we will 
draw moves fixing the starting point, i.e. the first entry of an arc in $\U$ as pointing 
down and moves fixing the endpoint, i.e. the second entry of the arc in $\U$ 
will be pointing up. 

We will sometimes also need to work with moves between arcs in $\U$ itself, they 
are defined analogously. 

Observe that  if $b_1=b_2$, 
one of the images might 
be isotopic to a segment of the boundary. In this case, there is only one 
elementary move from $\alpha$. \\
On the other hand, if an arc corresponds to an injective indecomposable, 
it should only have an elementary move to another arc for an injective indecomposable 
or no elementary move, if $\alpha$ is the terminal object, i.e. if $\alpha$ corresponds 
to the 1-dimensional injective module. 

We use lifts to describe the elementary moves case-by-case. 

\begin{itemize}
\item Bridging arcs: 
Let $\alpha$ be a bridging arc of $P_{g,h}$. Then we define two elementary 
moves 
starting at $\alpha$. \\
(i) If the %canonical 
lift of $\alpha$ in $\U$ is 
$[i_{\partial},y_{\partial'}]$ for some $i,y\in \mathbb{Z}$, %$0\le i< g$, 
they are of the form 
%Then there are exactly two elementary moves starting at $\alpha$:
			$$\xymatrix@R=0.3pc{
			       &\pi([(i-1)_{\partial},y_{\partial'}])\\
			       \color{blue}{\alpha=\pi([i_{\partial},y_{\partial'}])}\ar[ur]\ar[dr]&\\
			       &\pi([i_{\partial},(y+1)_{\partial'}])
			}$$
The corresponding picture in $\U$ is: 
\begin{center}
\psfragscanon
\psfrag{1}{\tiny $(i-1)_{\partial}$}
\psfrag{2}{\tiny $i_{\partial}$}
\psfrag{3}{\tiny $(y+1)_{\partial'}$}
\psfrag{4}{\tiny $y_{\partial'}$}
\includegraphics[scale=.4]{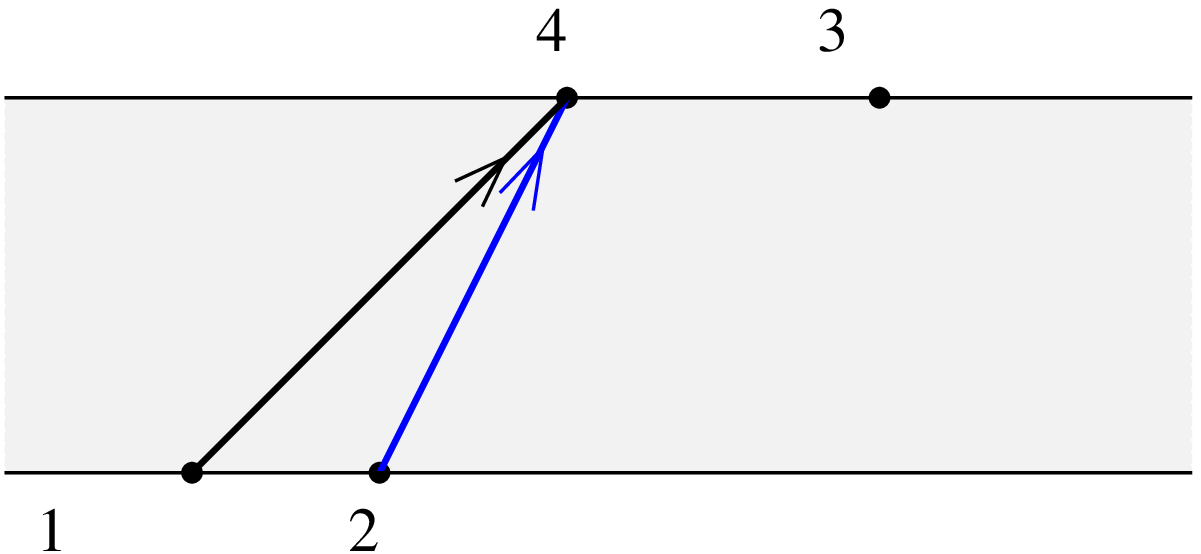}
\hskip 1cm
\includegraphics[scale=.4]{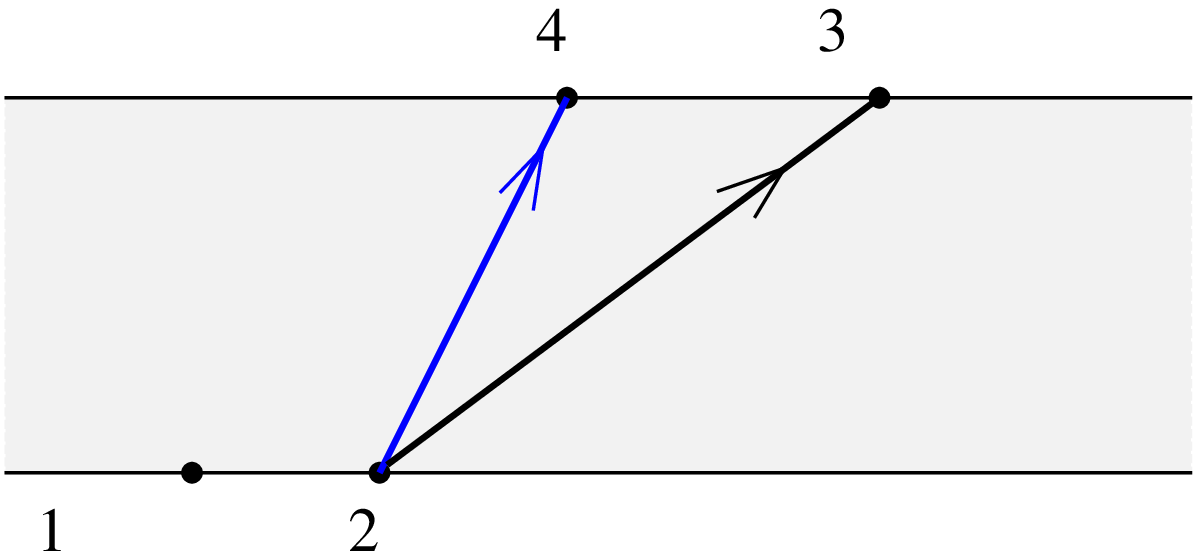}
\end{center}

\noindent
(ii) If the %canonical 
lift of $\alpha$ in $\U$ is $[j_{\partial'},x_{\partial}]$ with $j,x\in \mathbb{Z}$, 
%$0\le j< h$, 
they are of the form 
			$$\xymatrix@R=0.3pc{
			       &\pi([(j+1)_{\partial'},x_{\partial}])\\
			       \color{blue}{\alpha=\pi([j_{\partial'},x_{\partial}])}\ar[ur]\ar[dr]&\\
			       &\pi([j_{\partial'},(x-1)_{\partial}])
			}$$
In pictures in $\U$:  
\begin{center}
\psfragscanon
\psfrag{1}{\tiny $(x-1)_{\partial}$}
\psfrag{2}{\tiny $x_{\partial}$}
\psfrag{3}{\tiny $(j+1)_{\partial'}$}
\psfrag{4}{\tiny $j_{\partial'}$}
\includegraphics[scale=.4]{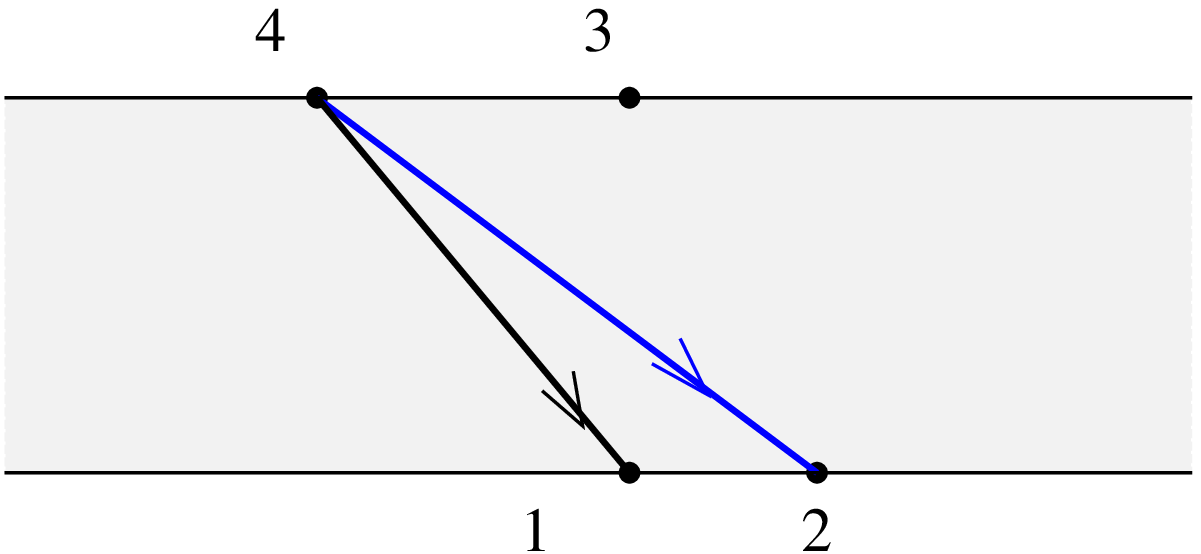}
\hskip 1cm
\includegraphics[scale=.4]{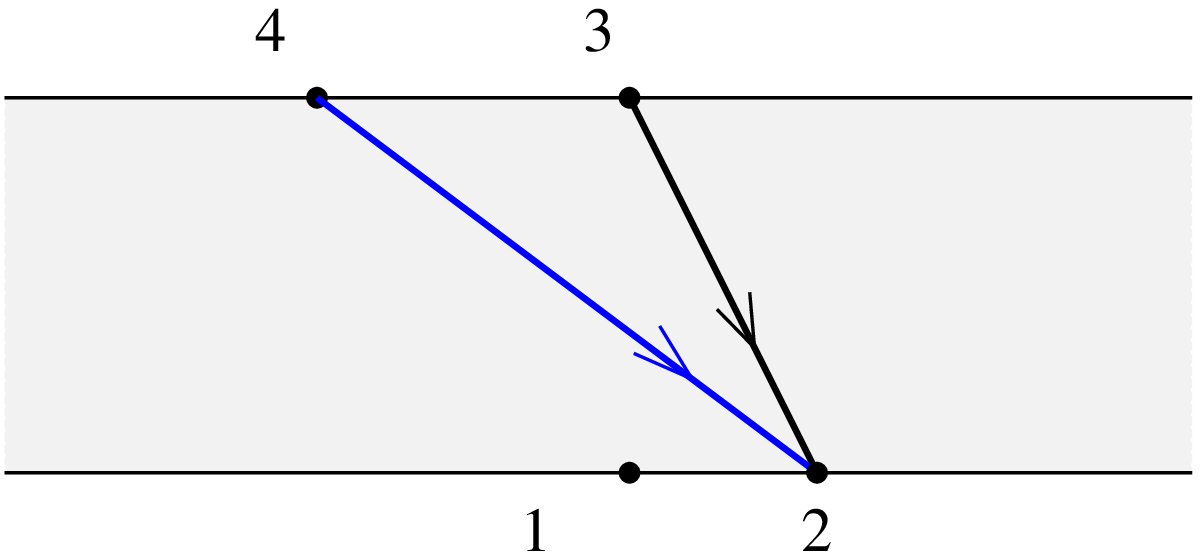}
\end{center}

\item 
Peripheral arcs: 
Let $\alpha$ be a peripheral arc of $P_{g,h}$. \\
(i) Assume that the %canonical 
lift of $\alpha$ is 
$[i_{\partial},a_{\partial}]$ with $i,a\in\mathbb{Z}$. %($0\le i<g$, $a\ge i+2$). 
	\begin{itemize}
		\item If $a=i+2$, there is only one elementary move from $\alpha$: 
 			$$\xymatrix@R=0.4pc{
			       &\pi([(i-1)_{\partial},(i+2)_{\partial}])\\
			       \textcolor{blue}{\alpha=\pi([i_{\partial},(i+2)_{\partial}])}\ar[ur] & 
			}$$
%		$$\alpha=\pi([i_{\partial},(i+2)_{\partial}])\rightarrow \pi([(i-1)_{\partial},(i+2)_{\partial}])$$
		\item If $a>i+2$, there are two elementary moves from $\alpha$: 
 			$$\xymatrix@R=0.3pc{
			       &\pi([(i-1)_{\partial},a_{\partial}])\\
			       \textcolor{blue}{\alpha=\pi([i_{\partial},a_{\partial}])}\ar[ur]\ar[dr]&\\
			       &\pi([i_{\partial},(a-1)_{\partial}])
			}$$
	\end{itemize}
In $\U$, the case $a>i+2$ looks as follows: 

\begin{center}
\psfragscanon
\psfrag{1}{\tiny $(i-1)_{\partial}$}
\psfrag{2}{\tiny $i_{\partial}$}
\psfrag{3}{\tiny $(a-1)_{\partial}$}
\psfrag{4}{\tiny $a_{\partial}$}
\includegraphics[scale=.38]{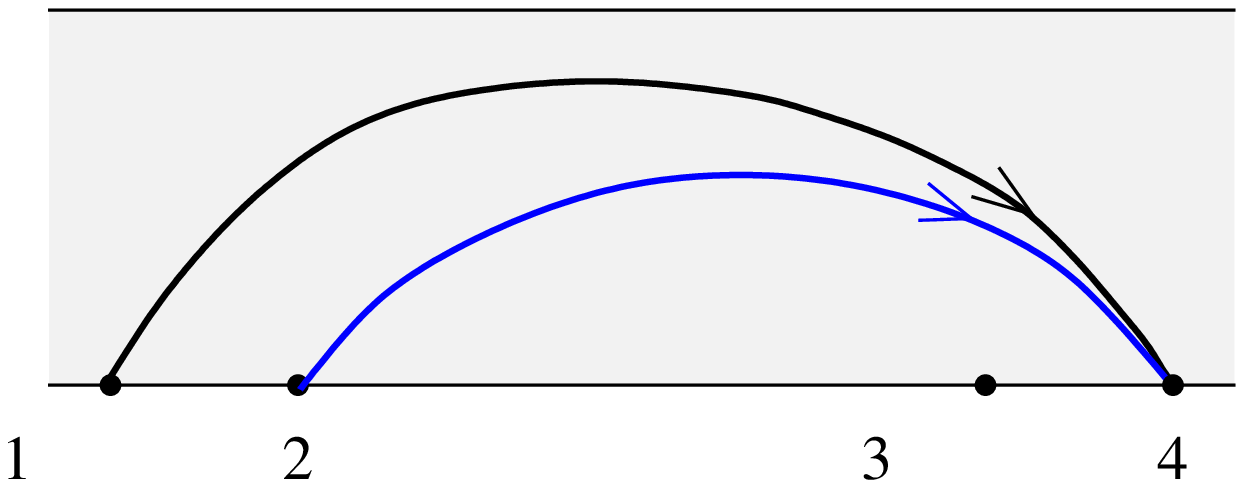}
\hskip 1cm
\includegraphics[scale=.38]{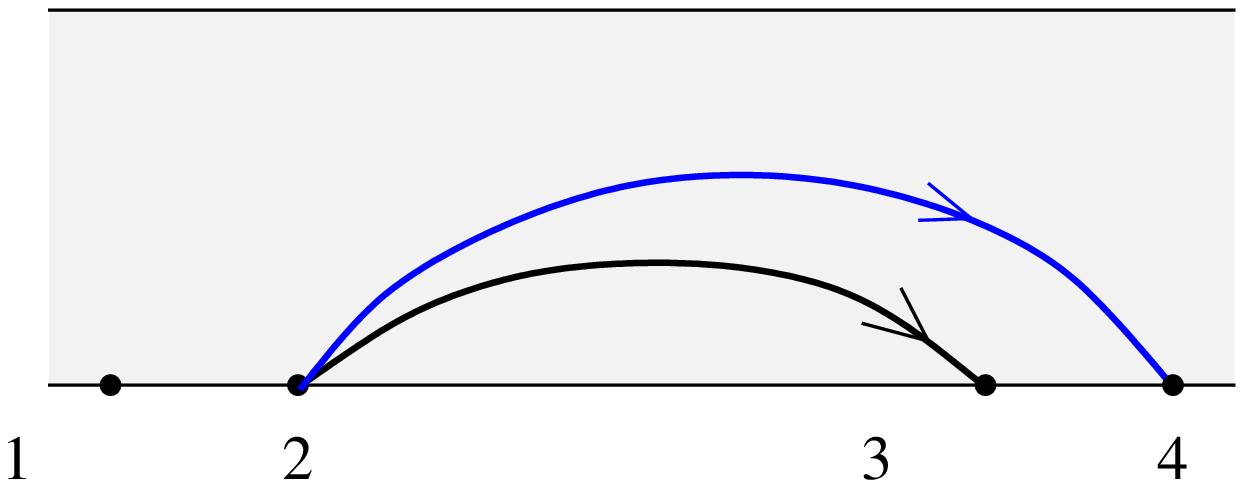}
\end{center}

\noindent
(ii) 
Let $[j_{\partial'},b_{\partial'}]$ be a %the canonical 
lift of $\alpha$, with $b,j\in\mathbb{Z}$. %(with $0\le j<h$ and $b\ge j+2$) . 
	\begin{itemize}
		\item If $b=j+2$, there is one elementary move from $\alpha$: 
			$$\xymatrix@R=0.4pc{
			       \textcolor{blue}{\alpha=\pi([j_{\partial},(j+2)_{\partial}])}\ar[dr] & \\
			       &\pi([j_{\partial},(j+3)_{\partial}])
			}$$
%		$$\pi([j_{\partial'},(j+2)_{\partial'}])\rightarrow \pi([j_{\partial'},(j+3)_{\partial'}])$$
		\item If $b>j+2$, there are two elementary moves from $\alpha$:
			$$\xymatrix@R=0.3pc{
			       &\pi([(j+1)_{\partial'},b_{\partial'}])\\
			       \textcolor{blue}{\beta=\pi([j_{\partial'},b_{\partial'}])}\ar[ur]\ar[dr]&\\
			       &\pi([j_{\partial'},(b+1)_{\partial'}])
			}$$
	\end{itemize}
For $b>j+2$, we have the following picture in $\U$: 
\begin{center}
\psfragscanon
\psfrag{1}{\tiny $(b+1)_{\partial'}$}
\psfrag{2}{\tiny $b_{\partial'}$}
\psfrag{3}{\tiny $j_{\partial'}$}
\psfrag{4}{\tiny $(j+1)_{\partial'}$}
\includegraphics[scale=.38]{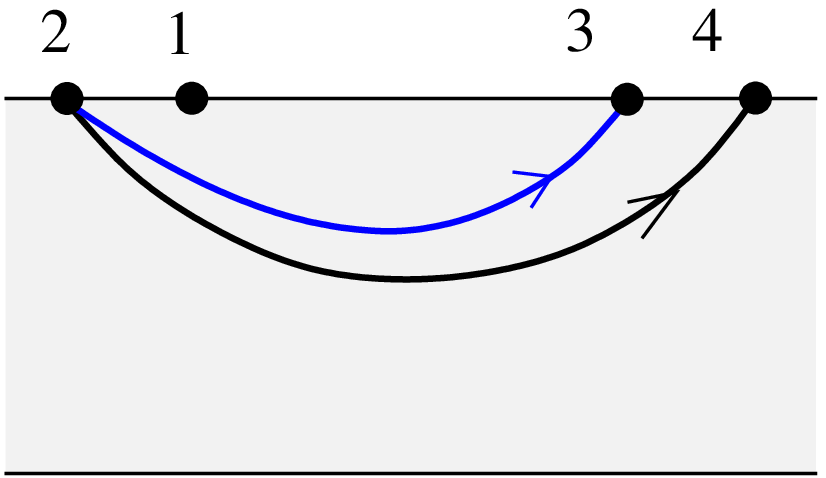}
\hskip 1cm
\includegraphics[scale=.38]{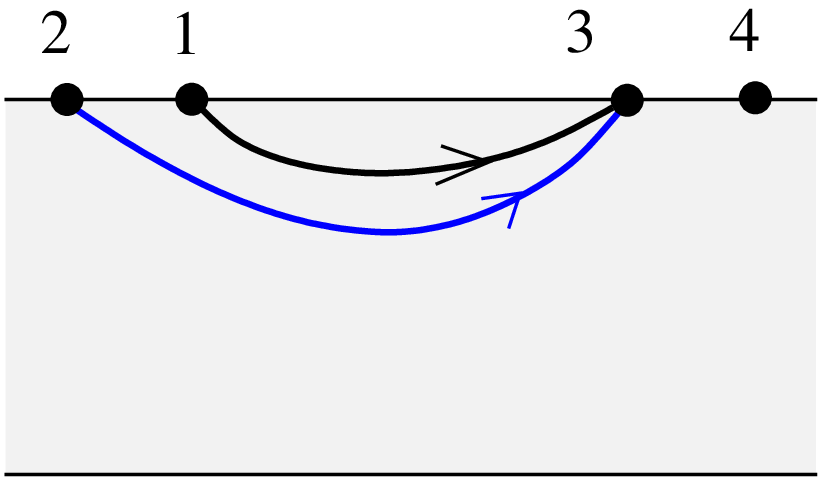}
\end{center}
\end{itemize}

%%%%%%%%%%%%%%%%%%%%%%%%%%
%
\subsection{Preprojective arcs} \label{sec:arcs-in-P}
%
%%%%%%%%%%%%%%%%%%%%%%%%%%

Recall that the quiver $Q_{g,h}$ has its source at vertex 0, its sink at vertex $g$ 
(cf. Section~\ref{ssec:AR-quiver-mod}). 
Hence, among the projective indecomposable modules, $P_g$ is one-dimensional, with irreducible morphisms 
to $P_{g-1}$ and to $P_{g+1}$. 
$P_0$ is the largest projective indecomposable module, with irreducible 
morphisms 
from $P_1$ and from $P_{n}$. The preprojective component of $\mo kQ_{g,h}$ starts with the 
projective indecomposable 
modules (cf. Figure~\ref{fig:AR-3-2} for $g=3$, $h=2$). 
%and left hand side of Figure~\ref{fig:AR-preprojective}). 
All other indecomposable modules in $\Pp$ are reached through compositions of irreducible 
maps from the $P_i$. 

We now describe a subset of the bridging arcs from $\partial$ to $\partial'$ that will play the 
role of projective indecomposable modules. 

Let 
$$
\beta_i:=\left\{\begin{array}{ll} 
			\pi([(i-g)_{\partial},0_{\partial'}]) & i=0,1,\dots, g-1,g  \\
%			\pi([(-g)_{\partial},0_{\partial'}]) & i=0 \\
			\pi([0_{\partial},(i-g)_{\partial'}]) & i=g+1,g+2,\dots, n 
		\end{array}
		\right.
$$

%\begin{figure}[h]
%\begin{center}
%\xymatrix@=1mm@R=2mm{ 
%&P_g\ar[rd]\ar@{.}[rr] & & \bullet\ar[rd] & & & & % 
%&& && &   & \bullet\mbox{\phantom{b}}\ar@{.}[rr]\ar[rd] & & I_0
%\\
%&&P_{g+1} \ar[ru] \ar[rd]  &&  &&& % 
%&& && &\cdots &  & I_n\ar[ru]\\
%& && \rotatebox{352}{$\ddots$}  \ar[rd] &  &&&% 
%&& && &&   \ \ \rotatebox{5}{$\iddots$} \ar[ru] & && 
%\\
%& &&& P_0  & \dots & & % 
%&& && &\mbox{\phantom{b}} I_g\ar[ru] \ar[rd] & & &&& 
%\\
% &&& P_1 \ar[ru] && &&% 
%&& && &&\   I_{g-1}\ar[rd] &&&  
%\\
%& & \ \rotatebox{5}{$\iddots$} \ar[ru] && && &% 
%&& && && & \rotatebox{348}{$\ddots$}\ar[rd] &  
%\\
%& P_{g-1} \ar[ru]\ar[rd]   && \ \rotatebox{5}{$\iddots$}  & &&& % 
%&& && &   & \cdots &  & \mbox{\phantom{b}}I_1 \ar[rd] 
%\\ 
%P_g\ar[ru]\ar@{.}[rr] &&\bullet\ar[ru] & \dots &&&& % 
%&& &&  && & \bullet\mbox{\phantom{b}}\ar@{.}[rr]\ar[ru] & & I_0
%}
%\end{center}
%\caption{Slices with the projectives/injectives} 
%\label{fig:AR-preprojective}
%\end{figure}

\begin{dfn}
The arcs $\beta_0,\dots, \beta_{n}$ of $P_{g,h}$ are called {\em projective arcs}. 
Figure~\ref{fig:projective-arcs} 
displays the $\beta_i$. 
We call an arc $\alpha=[x_{\partial},y_{\partial'}]$ of $P_{g,h}$ {\em preprojective} if it can be reached 
from $\beta_g$ through a finite (possibly empty) 
sequence of elementary moves. In particular, the $\beta_i$ ($i=0,\dots, n$) are 
preprojective arcs. 
\end{dfn}

%If $\alpha$ is a bridging arc from $\partial$ to $\partial'$, we say that $\alpha$ is a preprojective arc. 
Roughly speaking, the preprojective arcs of $P_{g,h}$ 
are winding positively from the outer to the inner boundary, with 
initial object $\beta_g$.

\begin{figure}
\psfragscanon
\psfrag{g-1}{\tiny$_{g-1}$}
\psfrag{h-1}{\tiny $_{h-1}$}
\psfrag{g}{\tiny$_{g}$}
\psfrag{h}{\tiny $_{h}$}
\psfrag{0}{\tiny$_0$}
\psfrag{1}{\tiny$_1$}
\psfrag{2}{\tiny$_2$}
\psfrag{a}{\small $\beta_g$}
\psfrag{b}{\small $\beta_{g-1}$}
\psfrag{c}{\small $\beta_{1}$}
\psfrag{d}{\small $\beta_{0}$}
\psfrag{e}{\small $\beta_n$}
\psfrag{f}{\small $\beta_{g+1}$}
\psfrag{rot0}{rotations around $0_{\partial}$}
\psfrag{roth}{rotations around $0_{\partial'}$}
\includegraphics[scale=.5]{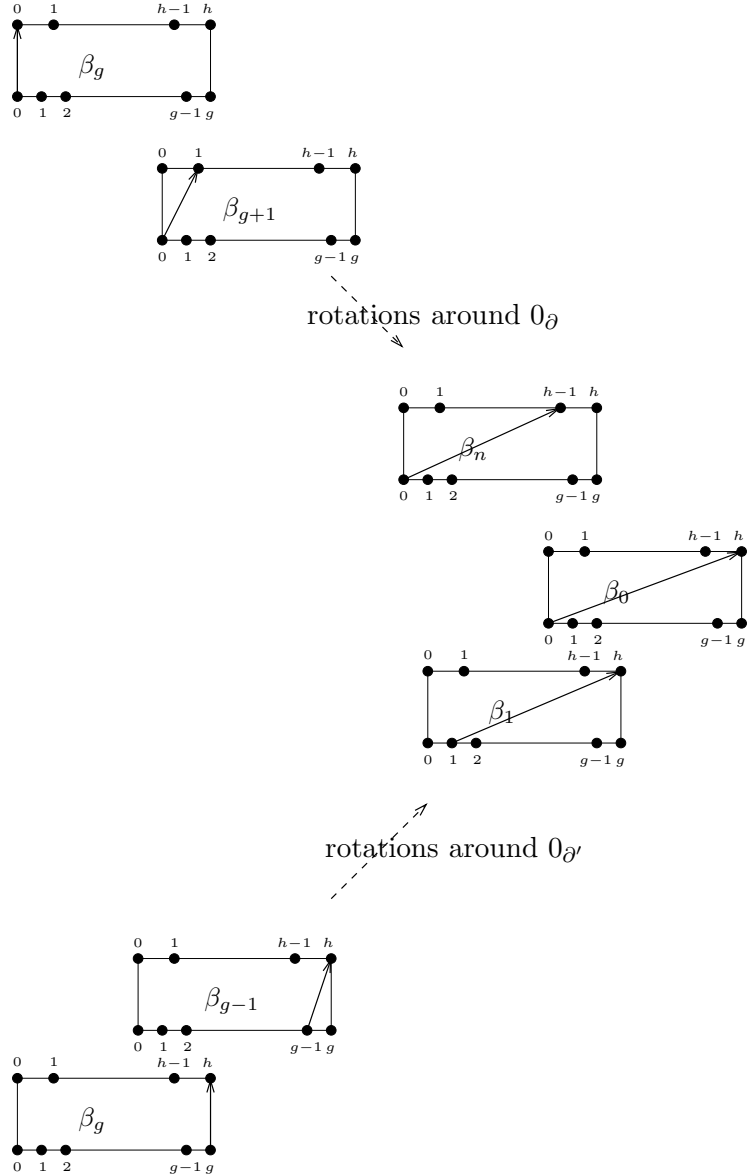}
\caption{The projective arcs}
\label{fig:projective-arcs}
\end{figure}

%%%%%%%%%%%%%%%%%%%%%%%%%%
%
\subsection{Preinjective arcs} \label{sec:arcs-in-I}
%
%%%%%%%%%%%%%%%%%%%%%%%%%%

Among the injective indecomposable modules of $\mo kQ_{g,h}$, $I_0$ is the one-dimensional one, with 
irreducible morphisms from $I_{n}$ and from $I_1$. $I_g$ is the largest one, with irreducible morphisms 
to $I_{g+1}$ and to $I_{g-1}$ (cf. Figure~\ref{fig:AR-3-2} for $g=3$, $h=2$). 
%of Figure~\ref{fig:AR-preprojective}).
All other indecomposable modules in $\Ip$ map to the injectives through compositions of 
irreducible maps. 

We now describe a subset of the bridging arcs from $\partial'$ to $\partial$ 
for the injective indecomposable modules. 
Let 
$$
\gamma_i:=\left\{\begin{array}{ll} 
			\pi([-2_{\partial'}, (i-g+2)_{\partial}]) & i=0,1,\dots, g \\
%			\pi([2_{\partial'}, (i-g-2)_{\partial}]) & i=0,1,\dots, g \\
			\pi([(i-g-2)_{\partial'},2_{\partial}]) & i=g+1,g+2,\dots, n
%			\pi([(i-g+2)_{\partial'},-2_{\partial}]) & i=g+1,g+2,\dots, n
		\end{array}
		\right.
$$
\begin{dfn}
The arcs $\gamma_0,\dots, \gamma_n$ of $P_{g,h}$ are 
called {\em injective arcs}. Figure~\ref{fig:injective-arcs} 
displays the $\gamma_i$. 
We call an arc $\alpha=[x_{\partial'},y_{\partial}]$ of $P_{g,h}$ {\em preinjective} if $\gamma_0$ 
is reached from $\alpha$ through a finite sequence of elementary moves. In particular, the $\gamma_i$ ($i=0,\dots, n$) are 
preinjective arcs. 
\end{dfn}

\begin{figure}
\psfragscanon
\psfrag{g-1}{\tiny$_{g-1}$}
\psfrag{h-1}{\tiny $_{h-1}$}
\psfrag{g}{\tiny$_{g}$}
\psfrag{h}{\tiny $_{h}$}
\psfrag{0}{\tiny$_0$}
\psfrag{1}{\tiny$_1$}
\psfrag{2}{\tiny$_2$}
\small\psfrag{rot2}{rotations around $2_{\partial}$}
\psfrag{roth-2}{rotations around $(h-2)_{\partial'}$}
\psfrag{a}{\small $\gamma_g$}
\psfrag{b}{\small $\gamma_{g+1}$}
\psfrag{c}{\small $\gamma_{n}$}
\psfrag{d}{\small $\gamma_{0}$}
\psfrag{e}{\small $\gamma_{g-1}$}
\psfrag{f}{\small $\gamma_1$}
\includegraphics[scale=.5]{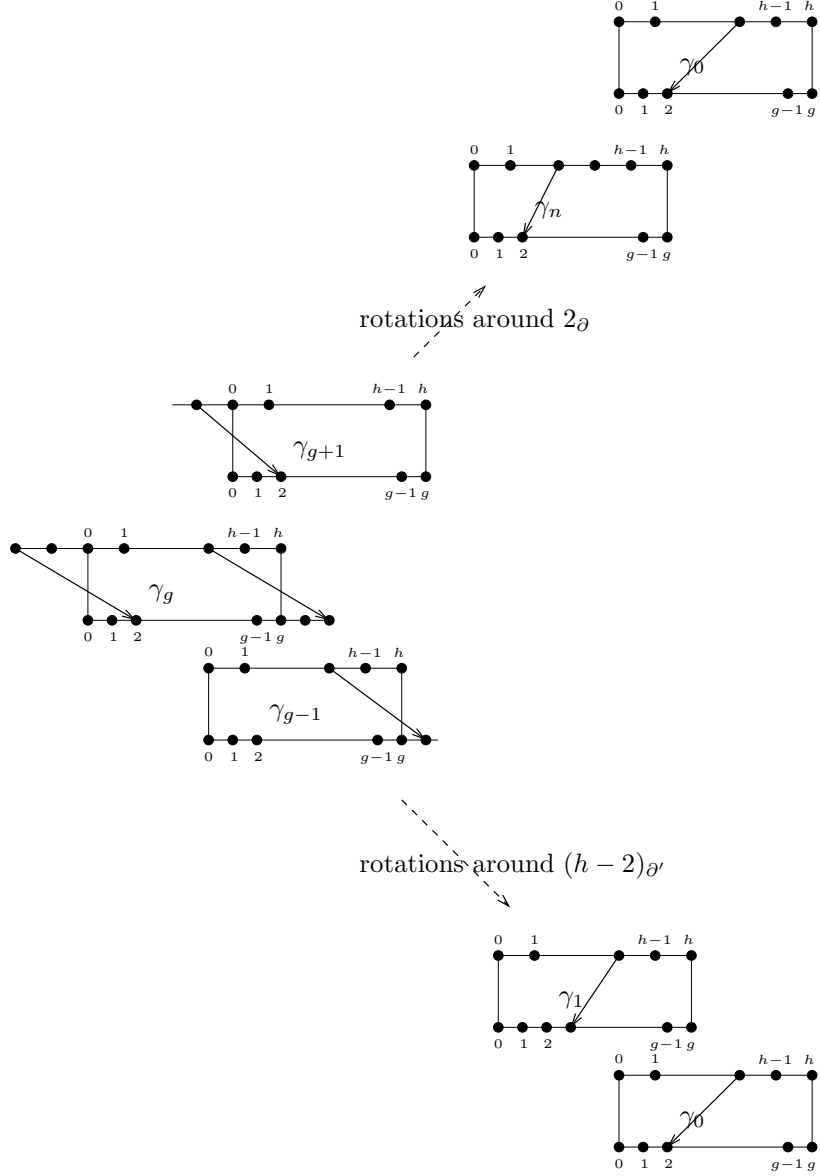}
\caption{The injective arcs}
\label{fig:injective-arcs}
\end{figure}

Up to rotating the inner boundary by $2\frac{2\pi}{h}$, the preinjective arcs are winding positively 
from the inner to the outer boundary towards the terminal object $\gamma_0$. 

\begin{dfn}
An oriented arc of $P_{g,h}$ is called {\em admissible}, if it is either preprojective, peripheral or preinjective.
\end{dfn}

%%%%%%%%%%%%%%%%%%%%%%%%%%
%
\subsection{A translation quiver}\label{sec:transl-quiver-tau}
%
%%%%%%%%%%%%%%%%%%%%%%%%%%

With the definitions from the last sections, we are ready to define a translation quiver on arcs in $P_{g,h}$. 
The vertices of $\Gamma=\Gamma(P_{g,h})$ are the admissible arcs of $P_{g,h}$. 
The arrows of $\Gamma$ are given by the elementary moves. Furthermore, let 
$\tau$ be induced by the map $i_{\partial}\mapsto (i+1)_{\partial}$ and 
$j_{\partial'}\mapsto (i-1)_{\partial'}$. 

In terms of $\U$, $\tau$ shifts endpoints on $\partial$ to the right and endpoints on $\partial'$ 
to the left. Peripheral arcs are thus shifted to the right/to the left, whereas the endpoints 
of bridging arcs of $\Gamma$ move in different directions under $\tau$.

The translation $\tau$ on $\Gamma$ makes $(\Gamma,\tau)$ a translation quiver. This means that 
whenever we have two vertices $\tau(\alpha)$, $\alpha\in \Gamma$, there is a mesh relation involving 
the subgraph on these two vertices and their common neighbours. In our case, these meshes 
have either four or three vertices: 
$$
\xymatrix@R=+0.6pc @C=+0.3pc{
 & \gamma_1\ar[rd]   &  && & \gamma\ar[rd]\\ 
\tau\alpha \ar[ru] \ar[rd]\ar@{..}[rr]  & & \alpha && \tau\alpha \ar[ru] \ar@{..}[rr]  & & \alpha\\
 & \gamma_2\ar[ru]
}
$$ 
The meshes on three vertices only occur at the mouth of tubes formed by 
peripheral arcs. As relations on the quiver $(\Gamma,\tau)$, we impose on 
meshes with four vertices that 
the two paths between $\tau(\alpha)$ and $\alpha$ are equal and on meshes with 
three vertices, that the composition of the two arrows is zero.

\medskip 

The following result is well-known, it appears in \cite{w} for the module category and 
in \cite{bm} for tubes. 

\begin{prop}\label{prop:translationquiver}
$(\Gamma,\tau)$ is a translation quiver, $(\Gamma,\tau)$ is 
isomorphic to the AR-quiver of $\mo kQ_{g,h}$. 
\end{prop}

By the above result, $(\Gamma,\tau)$ describes $\mo kQ_{g,h}$ up 
to maps in the infinite radical. 

\begin{rem} 
(1) Observe that by definition, $\tau(\beta_i)$, $0\le i\le n$ 
is not admissible and neither are the $\tau^{-1}(\gamma_i)$. 

(2) However, if we forget the orientations of arcs, there is a way 
to go from $\beta_i$ to $\gamma_i$ via $\tau^2$ or back 
via $\tau^{-2}$ as follows: Let $0\le i\le n$. Then $\gamma_i$ can be obtained 
from $\beta_i$ by applying $\tau^2$ to the arc (the result is not 
admissible, but bridging) and reversing the orientation 
of the image. I.e. if $\beta_i$ has lift $[a_{\partial},b_{\partial'}]$ in $\U$, 
then the arc $[\tau^2(b_{\partial'}),\tau^2(a_{\partial}]$ is a lift of 
$\gamma_i$. 
\end{rem}

%(ETC ETC, refer to relevant articles) \\

%%%%%%%%%%%%%%%%%%%%%%%%%%

\section{The infinite radical}\label{sec:inf-rad}

%%%%%%%%%%%%%%%%%%%%%%%%%%

In this section, we describe the elements of the infinite radical of $\mo kQ_{g,h}$. We will do this using 
the geometric 
interpretation of objects in terms of admissible arcs. With this, we will obtain a quiver 
$\GG=\GG(P_{g,h})$ from $\Gamma=\Gamma(P_{g,h})$. 
The vertices of $\GG$ are the vertices of $\Gamma$, the arrows of 
$\GG$ are the arrows of $\Gamma$ together with the arrows arising from elements 
of the infinite radical (corresponding to the ``long moves'', see below).

Elementary moves are combinatorial models for irreducible morphisms, i.e. elements 
of $\Rad(X,Y)/\Rad^2(X,Y)$ for $X,Y$ the two indecomposable objects corresponding 
to the starting and end point of $\alpha$. We have seen that 
an elementary move sends an arc $\alpha$ to 
another arc $\alpha'$ with a common endpoint by rotating the arc 
$\alpha$ clockwise around the common endpoint. 
The other endpoint is moved by $\pm 1$ along the corresponding boundary. 

Long moves are models for maps in the infinite radical of $\mo\tilde{A}$, 
i.e. for elements of 
$\Rad^{\infty}(\mo \tilde{A})=\cap_i \Rad^i(\mo \tilde{A})$. 
In terms of the AR-quiver, maps in the infinite radical arise from infinitely many 
compositions of irreducible morphisms, keeping the direction fixed, 
i.e. moving along a ray or coray in the quiver $\Gamma_{g,h}$. 

%%%%%%%%%%%%%%%%%%%%%%%%%%
%
\subsection{Long moves}\label{sec:long-moves}
%
%%%%%%%%%%%%%%%%%%%%%%%%%%
%

Long moves are viewed as limits of sequences of elementary moves, fixing 
one common vertex for all the elementary moves involved. 
This way, long moves follow a line in the quiver $\Gamma(P_{g,h})$. 

Geometrically, this is achieved by rotating an arc $\alpha$ clockwise to an arc $\alpha'$ 
which has one endpoint in common with $\alpha$ and where the other endpoint has 
changed boundary components. 
The endpoint switching boundary components 
is allowed to go to any marked point on the new boundary, as long as the image is an admissible arc 
(preprojective, peripheral or preinjective). 

In particular, there are infinitely many 
long moves from $\alpha$. This is now described case-by-case. 
We will indicate the long moves by drawing dashed arrows. 
Let $g\ge h>0$ be fixed, let $\pi=\pi_{g,h}:\U\to P_{g,h}$ be the covering map. 
\begin{itemize}
\item Bridging arcs: Let  $\alpha$ be a preprojective arc,  
let $[i_{\partial},y_{\partial'}]$ be a lift of $\alpha$. %its canonical lift, $0\le i<g$. 
	\begin{itemize}
	\item There are long moves from $\alpha$ to every vertex of the coray
	of $\pi([i_{\partial}, i+2_{\partial}])$ in the tube $\T_g$, i.e. % $\mathcal{T}_{\partial}$, i.e. 
	to every arc $\pi([i_{\partial},(i+k)_{\partial}])$ for $k\ge 2$. 
		\begin{equation}\label{bridging1}
		\xymatrix@R=3pc@C=3pc@!0{
			\ddots &&\ar[dr]\\
			\alpha=\pi([i_{\partial},y_{\partial'}])\ar@{-->}[rru]\ar@{-->}[rrr]\ar@{-->}[rrrrd]\ar@{-->}[rrrrrdd]& 
			&&\pi([i_{\partial},(i+4)_{\partial}])\ar[dr]\\
			&&&&\pi([i_{\partial},(i+3)_{\partial}])\ar[dr]\\
			&&&&&\pi([i_{\partial},(i+2)_{\partial}])
		}\end{equation} 
		Geometrically this means that from every preprojective arc starting at $i_{\partial}$, 
		there are long moves 
		to every peripheral arc starting at $i_{\partial}$.
	\item There are long moves from $\alpha$ to every vertex of the coray of 
	$\pi([(y-2)_{\partial'},y_{\partial'}]$ in the tube $\mathcal{T}_h$, i.e. to 
	every arc $[(y-k)_{\partial'},y_{\partial'}]$ with $k\ge 2$. (Recall that the quiver of 
	$\T_h$ is drawn upside down, the vertex $\pi([(y-2)_{\partial'},y_{\partial'}])$ sits at the mouth 
	of this component). 
		\begin{equation}\label{bridging2}
		    \xymatrix@R=3pc@C=3pc@!0{
			&&&&&\pi([(y-2)_{\partial'},y_{\partial'}]) \\
			&&&&\pi([(y-3)_{\partial'},y_{\partial'}])\ar[ur]\\
			\alpha=\pi([i_{\partial},y_{\partial'}])\ar@{-->}[rrd]\ar@{-->}[rrr]\ar@{-->}[rrrru]\ar@{-->}[rrrrruu]& 
			&&\pi([(y-4)_{\partial'},y_{\partial'}])\ar[ur]\\
			\iddots &&\ar[ur] 
		}\end{equation} 
		Geometrically this means that from every preprojective arc ending at 
		$y_{\partial'}$, there are long moves 
		to every peripheral arc ending at $y_{\partial}$.
	\end{itemize}
\item Peripheral arcs:  
        \begin{itemize}
        \item 
        Let $\alpha$ be peripheral and based at $\partial$, let $[a_{\partial},y_{\partial}]$ be 
        a lift of $\alpha$. Then we define a long move to every preinjective  
        arc $\pi([x_{\partial'},y_{\partial}])$ of $P_{g,h}$. 
		\begin{equation}\label{peripheral1}\xymatrix@R=3pc@C=3pc@!0{
			&&&&&& \gamma \\
			&&&&& \rotatebox{10}{$\ \iddots$}\ar[ur]\\
			& & & \rotatebox{15}{$\iddots$}  &\pi([x_{\partial'},y_{\partial}])\ar[ru]\\
			\alpha=\pi([a_{\partial},y_{\partial}])\ar@{-->}[rdd]\ar@{-->}[rrd]\ar@{-->}[rrr]
			  \ar@{-->}[rrrru]\ar@{-->}[rrrrrruuu]&&&\pi([(x-1)_{\partial'},y_{\partial}]\ar[ru])\\
			&&\pi([(x-2)_{\partial'},y_{\partial}])\ar[ru]\\
			\iddots &\ar[ru]
		}\end{equation}%\label{peripheral1}
        Note that the terminal object of this sequence is one of the injective 
        arcs $\gamma_0,\dots, \gamma_{g-1}$. More precisely, 
        if we assume that the lift of $\alpha$ 
        satisfies $-2\le y<g-2$, then the terminal object is 
        $\pi[2_{\partial'},y_{\partial}]=\gamma_0$ if $y=-2$ and 
        $\gamma=\gamma_{y+2}$ otherwise. 
        Geometrically the description above means that from every peripheral arc ending at $y_{\partial}$, 
        there are long moves 
        to every admissible bridging arc ending at $y_{\partial}$. 
        \item 
        Let $\alpha$ be a peripheral arc based at $\partial'$ and let $[x_{\partial'}, b_{\partial'}]$ 
        a lift of $\alpha$. Then we define a long move from $\alpha$ to every arc 
        $\pi([x_{\partial'},z_{\partial}])$. 
		\begin{equation}\label{peripheral2}\xymatrix@R=3pc@C=3pc@!0{
			\ddots &\ar[dr]\\
			&&\pi([x_{\partial'},(z+2)_{\partial}])\ar[dr]\\
			\alpha=\pi([x_{\partial'}, b_{\partial'}])\ar@{-->}[uur]\ar@{-->}[urr]\ar@{-->}[rrr]
			\ar@{-->}[drrrr]\ar@{-->}[dddrrrrrr]&&&\pi([x_{\partial'},(z+1)_{\partial}])\ar[dr]\\
			&&&  \rotatebox{345}{$\ddots$} &\pi([x_{\partial'},z_{\partial}])\ar[dr]\\
			&&&&& \rotatebox{350}{$\ \ddots$} \ar[rd] \\
			&&&&& & \gamma
		}\end{equation} 
	Again, the terminal object of this sequence is an injective arc, it is $\gamma_0$ or 
	one of the arcs $\gamma_{g+1},\dots, \gamma_n$. If we assume that the lift of $\alpha$ 
        satisfies $0\le x<-h-1$, then this injective arc is $\pi[x_{\partial'},-2]=\gamma_0$ if $x=2$ and 
        $\gamma=\gamma_{x+g-2}$ otherwise. 

	Geometrically the above description means that from every peripheral arc starting at 
	$x_{\partial'}$, there are long moves 
	to every admissible bridging arc starting at $x_{\partial'}$.
        \end{itemize}
\end{itemize}

\begin{rem}
Note that unlike $\Gamma$, the quiver $\GG$ is connected. 
\end{rem}

%%%%%%%%%%%%%%%%%%%%%%%%%%%%%%%%%%%
\subsection{Preimages under long moves} \label{ssec:preimages}
%%%%%%%%%%%%%%%%%%%%%%%%%%%%%%%%%%%

Now we make a few observations which follow from the definitions of the long moves above. 
We will need this when we define relations in the Section~\ref{sec:relations}. 

Consider a peripheral arc $\pi([i_{\partial},(i+k)_{\partial}])$ with $k\ge 2$. Then for every $y_{\partial'}$ 
such that $\pi([i_{\partial}, y_{\partial'}])$ is a preprojective 
arc there exists a long move 
$\pi([i_{\partial}, y_{\partial'}])\to\pi([i_{\partial},(i+k)_{\partial}])$. Combining this with elementary moves 
between preprojective arcs starting at $i_{\partial}$, we have the following 
picture: 
\begin{equation}\label{cons1}
\xymatrix@R=3pc@C=3pc@!0{
\ar@{..>}[rd]\\
&\pi([i_{\partial}, (j-2)_{\partial'}])\ar[rd]\ar@{-->}[rrrrrdd]\\
&&\pi([i_{\partial}, (j-1)_{\partial'}])\ar[rd]\ar@{-->}[rrrrd]\\
&&&\pi([i_{\partial}, j_{\partial'}])\ar[rd]\ar@{-->}[rrr]&&& \pi([i_{\partial}, (i+k)_{\partial}])\\
&&&&\pi([i_{\partial}, (j+1)_{\partial'}])\ar@{-->}[rru]\ar@{..>}[rd]\\
&&&&&
}\end{equation} 
Similarly, if we consider long moves from $\Pp$ to the tube of peripheral arcs at the inner boundary: 
Let $\pi([(y-k)_{\partial'},y_{\partial'}])$ be a peripheral arc, $k\ge 2$. Then there is a long move 
$\pi([j_{\partial}, y_{\partial'}])\rightarrow \pi((y-k)_{\partial'},y_{\partial'}])$ for every preprojective 
arc  of the form $\pi([j_{\partial}, y_{\partial'}])$. Combining this with elementary moves between 
preprojective arcs ending at $y_{\partial'}$, we get 
\begin{equation}\label{cons2}
\xymatrix@R=3pc@C=3pc@!0{
&&&&&\\
&&&&\pi([(j-2)_{\partial}, y_{\partial'}])\ar@{..>}[ru]\ar@{-->}[rrdd]\\
&&&\pi([(j-1)_{\partial}, y_{\partial'}])\ar[ru]\ar@{-->}[rrrd]\\
&&\pi([j_{\partial}, y_{\partial'}])\ar[ru]\ar@{-->}[rrrr]&&&& \pi([(y-k)_{\partial'}, y_{\partial'}])\\
&\pi([(j+1)_{\partial}, y_{\partial'}])\ar[ru]\ar@{-->}[rrrrru]\\
\ar@{..>}[ru]
}\end{equation}
Now we consider long moves from the tubes to $\Ip$. Let  $\pi([x_{\partial'},y_{\partial}])$
be an admissible bridging arc in $\Ip$. Then there is a long move 
$\pi([(j-k)_{\partial}, y_{\partial}]) \rightarrow \pi([x_{\partial'},y_{\partial}])$ for every
$k\geq 2$. In other words, there are long moves from every arc in the ray of $\pi([(j-2)_{\partial}, y_{\partial}])$
to $\pi([x_{\partial'},y_{\partial}])$.
\begin{equation}\label{cons3}
\xymatrix@R=3pc@C=3pc@!0{
&&&&\\
&&&\pi([(j-4)_{\partial}, y_{\partial}])\ar@{-->}[rrrd]\ar@{..>}[ru]\\
&&\pi([(j-3)_{\partial}, y_{\partial}])\ar[ru]\ar@{-->}[rrrr]&&&& \pi([x_{\partial'},y_{\partial}])\\
&\pi([(j-2)_{\partial}, y_{\partial}])\ar[ru]\ar@{-->}[rrrrru]\\
\ar@{..>}[ru]
}\end{equation}
Similarly we get long moves from every arc in the ray of $\pi([y_{\partial'}, (y+2)_{\partial'}])$
to $\pi([x_{\partial'},y_{\partial})]$. 
\begin{equation}\label{cons4}
\xymatrix@R=3pc@C=3pc@!0{
&\pi([y_{\partial'}, (y+2)_{\partial'}])\ar[rd]\ar@{-->}[rrrrrd]\\
&&\pi([y_{\partial'}, (y+3)_{\partial'}])\ar[rd]\ar@{-->}[rrrr]&&&& \pi([x_{\partial'},y_{\partial}])\\
&&&\pi([y_{\partial'}, (y+4)_{\partial'}])\ar@{-->}[rrru]\ar@{..>}[rd]\\
&&&&
}\end{equation}
\begin{rem}
The connected quiver $\GG$ 
is {\bf not} locally finite. 
\end{rem}

%%%%%%%%%%%%%%%%%%%%%%%%%%%%%%%%%%%
\subsection{Mesh relations in $\GG$} 
%%%%%%%%%%%%%%%%%%%%%%%%%%%%%%%%%%%

In what follows, when we 
refer to connected components, we will usually refer to the 
connected components $\Pp$, $\Ip$, $\T_g$, $\T_h$ of $\Gamma$.

For $\GG$, we keep the mesh relations involving the arrows of $\Gamma$. Within $\GG$, 
there are new diamonds of arrows arising, formed by moves of different kinds. 
If such a diamond involves vertices from two connected components of $\Gamma$, 
we define a new relation from it. 
We do not impose any relations on diamonds involving four long moves. 

The diagrams giving rise to new relations are of the form 
$$
\xymatrix@R=+0.6pc @C=+0.3pc{
 & Y_1\ar[rd]^{g_1} & \\ 
X\ar[ru]^{f_1}\ar[rd]_{f_2} & & Z \\
 & Y_2\ar[ru]_{g_2}
}
$$ 
where one pair of opposite arrows come from elementary moves 
(i.e. $f_1$ and $g_2$ or $g_1$ and $f_2$) 
and the two other arrows from 
long moves.

%%%%%%%%%%%%%%%%%%%%%%%%%%%%%%%%%%%
\subsection{Diamonds from preprojective arcs to peripheral arcs}\label{ssec:preproj-tube}
%%%%%%%%%%%%%%%%%%%%%%%%%%%%%%%%%%%

Recall that elementary moves in $\Pp$ are either of the form $\pi([i_{\partial}, y_{\partial'}]) \rightarrow \pi([(i-1)_{\partial}, y_{\partial'}])$ or
$\pi([i_{\partial}, y_{\partial'}]) \rightarrow \pi([i_{\partial}, (y+1)_{\partial'}])$. 
As before, we write dashed arrows to indicate long moves. 
Combining long moves with elementary moves, %(\ref{bridging1}) 
there exists diamonds
$$
\xymatrix@R=+0.6pc @C=+2pc{
 &  \pi([(i-1)_{\partial}, y_{\partial'}])\ar@{-->}[rd] & \\ 
\pi([i_{\partial}, y_{\partial'}])\ar[ru]\ar@{-->}[rd] & & \pi([(i-1)_{\partial}, (i+k)_{\partial}]) \\
 & \pi([i_{\partial}, (i+k)_{\partial}])\ar[ru]}
$$
for all $k \geq 2$. The vertices to the left and above belong to $\Pp$, the other two vertices 
are peripheral arcs based at the outer boundary. 

And %by (\ref{bridging2}), there are 
diamonds
$$
\xymatrix@R=+0.6pc @C=+2pc{
 &  \pi([(y-k)_{\partial'}, y_{\partial'}])\ar[rd] & \\ 
\pi([i_{\partial}, y_{\partial'}])\ar@{-->}[ru]\ar[rd] & & \pi([(y-k)_{\partial'}, (y+1)_{\partial'}]) \\
 & \pi([i_{\partial}, (y+1)_{\partial'}])\ar@{-->}[ru]
 }
$$
for all $k\geq 2$. Here the left vertex and the one below are preprojective, the other two peripheral 
and based at the inner boundary. 

In $\GG$ we require that all such diamonds commute. The geometric descriptions of these 
relations is in Cases A) and B) in Section~\ref{ssec:diamonds-long}. 

%These two diamonds are maps from the preprojective arcs to the two tubes.

%%%%%%%%%%%%%%%%%%%%%%%%%%%%%%%%%%%
\subsection{Diamonds from peripheral arcs to preinjective arcs}\label{ssec:tube-preinj}
%%%%%%%%%%%%%%%%%%%%%%%%%%%%%%%%%%%

We also have diamonds involving peripheral and preinjective arcs. 
$$
\xymatrix@R=+0.6pc @C=+2pc{
 & \pi([y_{\partial'}, i_{\partial}])\ar[rd]\\ 
 \pi([(i-k)_{\partial}, i_{\partial}])\ar[rd]\ar@{-->}[ru] & & \pi([y_{\partial'}, (i-1)_{\partial}]) \\
 &  \pi([(i-k)_{\partial}, (i-1)_{\partial}])\ar@{-->}[ru] & 
} 
$$
for all $k \geq 3$. The first vertex and the one below are peripheral arcs of $\partial$, 
two are preinjective arcs. 

Also, we have diamonds %by (\ref{bridging2}), there are diamonds
$$
\xymatrix@R=+0.6pc @C=+2pc{
 & \pi([(i+1)_{\partial'}, (i+k)_{\partial'}])\ar@{-->}[rd] \\ 
\pi([i_{\partial'}, (i+k)_{\partial'}])\ar[ru]\ar@{-->}[rd] & & \pi([(i+1)_{\partial'}, y_{\partial}]) \\
 &  \pi([i_{\partial'}, y_{\partial}])\ar[ru] &
}
$$
for $k\ge 3$. The first vertex and the one on top are peripheral arcs based at $\partial'$, 
the other two are preinjective arcs. 

In $\GG$ we require that all such diamonds commute. The geometric descriptions of these 
relations is in Cases C) and D) in Section~\ref{ssec:diamonds-long}.

%%%%%%%%%%%%%%%%%%%%%%%%%%
%
\subsection{Relations in $\GG$}\label{sec:relations-GG}
%
%%%%%%%%%%%%%%%%%%%%%%%%%%

Recall that $\tau$ is defined as moving endpoints on $\partial$ by $+1$ and endpoints on $\partial'$ by $-1$, 
inducing meshes and thus relations on $\Gamma$. There are additional relations we need to impose 
on $\GG$. In section \ref{sec:long-moves} we defined long moves (cf. (1)-(4)) and and 
then in Subsection~\ref{ssec:preimages} we described 
diagrams coming from the long moves (see (5)-(8)). 
In these diagrams we 
that all triangles commute.

From these relations we get the following facts.
\begin{itemize}
\item From (\ref{bridging1}) and (\ref{bridging2}) we get that any long map from $\Pp$ to a tube factors through infinitely many
arcs in that tube. In particular, from (1), any map of the form $\pi([i_{\partial}, y_{\partial'}]) 
\rightarrow \pi([i_{\partial},(i+k)_{\partial}])$, with $k\geq 2$, factors through all arcs 
$\pi([i_{\partial},j_{\partial}])$ where $j>i+k$.
Similarly, from (2), long maps from $\Pp$ to $\mathcal{T}_h$ factors through infinitely many arcs
in $\mathcal{T}_h$.
\item From (\ref{peripheral1}) and (\ref{peripheral2}) we get that any long map from a tube to $\Ip$ factors through infinitely many
preinjective arcs. In particular, from (\ref{peripheral1}), any map of the form $\pi([i_{\partial}, y_{\partial}])
\rightarrow \pi([x_{\partial'}, y_{\partial}])$ factors through all arcs $\pi([(x-k)_{\partial'}, y_{\partial}])$
where $k>1$. Similarly, from (\ref{peripheral2}), with long maps from $\mathcal{T}_h$ to $\Ip$.
\item From (\ref{cons1}) and (\ref{cons2}) 
any long map from $\Pp$ to a tube factors through infinitely many preprojective arcs.
\item From (\ref{cons3}) and (\ref{cons4}) we have that any long map from a tube to $\Ip$ factors through infinitely many
arcs in that tube. 
\end{itemize}

In addition, we have relations coming from the diamonds involving two different kind of 
arcs, see Subsections~\ref{ssec:preproj-tube} and \ref{ssec:tube-preinj}.

%%%%%%%%%%%%%%% NEW SECTION %%%%%%%%%%%%%%%%%%%%%
\section{An isomorphism of AR-quivers}\label{sec:arquivers}

%%%%%%%%%%%%%%%%%%%%%%%%%%%%%%%%%%%%

%%%%%%%%%%%%%%%%%%%%%%%%%
\subsection{Definition of $\GG_m$}\label{ssec:defn-GGm}

%%%%%%%%%%%%%%%%%%%%%%%%%

We now want to define a subquiver $\GG_m$ of $\GG$ and prove that it 
is isomorphic to the AR-quiver $Q_m$ defined by Br\"ustle in \cite{br}, 
except from the homogeneous tubes. See Section \ref{sec:bruestle}.

By Proposition \ref{prop:translationquiver}, the quiver $(\Gamma,\tau)$ is isomorphic to 
the AR-quiver of $\mod kQ_{g,h}$. Let $\GG^P$, $\GG^I$, $\GG^0$ and $\GG^\infty$ be the 
subquivers of $\GG$
consisting of the preprojective arcs, preinjective arcs, 
the arcs in the tube of rank $g$ and the tube of rank $h$ respectively.
Then we have that $\GG^P$ is isomorphic to $Q^P$, $\GG^I$ is isomorphic to 
$Q^I$, $\GG^0$ is isomorphic to $Q^0$ and
$\GG^\infty$ is isomorphic to $Q^\infty$.

Following \cite{br}, we define a subquiver $\GG^P_m$ of $\GG^P$ which is isomorphic to 
$Q^P_m$. We take the arcs
in $\GG^P_m$ to be all arcs that can be reached from a projective arc by applying 
$\tau^{-1}$ $ghm$ times, for some integer $m\geq 0$. In other words, 
the arcs in $\GG^P_m$ are all arcs of the form $\tau^{-r}P$,  with $P$ a projective 
arc and $0 \leq r \leq ghm$. It follows
directly from the definition of $Q^P_m$ that $\GG^P_m$ is isomorphic to $Q^P_m$. 
Similarly we define $\GG^I_m$ to be all arcs
in $\GG^I$ of the form $\tau^r I$, with $I$ an injective arc and $0 \leq r \leq ghm$. 
Clearly $\GG^I_m$ is isomorphic to $Q^I_m$.

When cut open to lie in the plane, the components $\GG_m^0$ and $\GG_m^{\infty}$ look like rectangles 
with an equilateral triangle on top. We start from this top vertex to describe $\GG_m^0$. 
We take the peripheral arcs at $\partial$, the longest one we want in $\GG_m^0$ 
is the arc $\pi[0_{\partial},(gm+g+2)_{\partial}]$. 
We take $g-1$ predecessors along the ray of this top vertex, namely 
$\pi[0_{\partial},(gm+g+2-i)_{\partial}]$ for $i=1,\dots, g-1$. Then $\GG_m^0$ consists of the 
full subgraph on the vertices in the corays from the mouth up to these $g$ vertices. 

For $\GG_m^{\infty}$, we start with 
$\pi[-2_{\partial'},(hm+h)_{\partial'}]$ and take $h-1$ predecessors of this vertex along its 
ray, $\pi[-2_{\partial'},(hm+h-i)_{\partial'}]$ for $i=1,\dots, h-1$. Then 
$\GG_m^{\infty}$ consists of the full subgraph on the vertices in the 
corays from the mouth down to these $h$ vertices.

\begin{dfn}
We define $\GG_m$ to be the full subquiver of $\GG$ consisting of the objects in $\GG^P_m$, 
$\GG^I_m$, $\GG^0_{2m(n+1)}$ and $\GG^\infty_{2m(n+1)}m$. 
\end{dfn}

Let $Q_m'$ be the full subquiver of $Q_m$ on all vertices apart from the homogeneous tubes. 
It is clear that the quivers $\GG_m$ and $Q_m'$ have the same vertices. Within the components, 
they also have the same arrows. Since we have arrows for long 
moves in $\GG_m$, the quiver $\GG_m$ has many more arrows than $Q_m'$, the latter only has 
very few arrows linking the components,namely the $\iota_0(i)$ and $\kappa_0(i)$ with $0\le i\le g$, 
the $\iota_{\infty}(i)$ and the $\kappa_{\infty}(i)$ for $i=g,g+1,\dots, n$. 
The translation maps on both quivers work the same, so we will denote both by $\tau$. 

To be able to go between $\GG_m$ and $Q_m'$, we need to know exactly which vertex 
of $\GG_m$ corresponds to which vertex of $Q_m'$. We do this now by giving a function 
$F:\GG_m \rightarrow Q_m'$

Recall from Section \ref{sec:inf-rad} the bridging arcs that play the role of
projective and injective indecomposable modules. 
This gives the function from the projective arcs to the projective objects in $Q_m$, i.e.
$F(\beta_i)=F(\pi[(i-g)_\partial, 0_{\partial'}])=(0,i)_P$ for $i=0,1, ..., g-1, g$ and 
$F(\beta_i)=F(\pi[0_{\partial}, (i-g)_{\partial'}])=(0,i)_P$ for 
$i=g+1, g+2, ..., n$.\\
This describes how $F$ works on $\GG_m^P$ since every preprojective arc 
is in the $\tau$-orbit of a projective arc, as every vertex of $Q_m^P$ belongs to the $\tau$-orbit 
of some projective $(0,j)_P$. 
E.g. if $0\le i\le g$, we have 
\begin{align} 
F(\tau^j \beta_i)  = & F(\tau^j \pi[(i-g)_\partial, 0_{\partial'}])= 
F(\pi[(i-g+j)_\partial, (-j)_{\partial'}])\label{eq:tau-orbit-beta} \\ 
 = & \tau^j (0,i)_P=(-j,i)_P \notag
\end{align} 
for $-ghm \leq j \leq 0$
%
%
%\medskip 
%
%For $i=g+1, g+2, ..., n$ we define 
%%
%\begin{eqnarray*} 
%F(\tau^j \beta_i) & = & F(\tau^j \pi[0_\partial, (i-g)_{\partial'}])=F(\pi[(j)_\partial, (i-g-j)_{\partial'}])\\ 
% & =\tau^j (0,i)_P &=  (-j,i)_P
%\end{eqnarray*}
%for $-ghm \leq j \leq 0$.\\

Now we consider the preinjective component. 

We define $F(\gamma_i)=F(\pi[-2_{\partial'}, (i-g+2)_\partial])=(0,i)_I$ for $i=0,1, ...,g$ and 
$F(\gamma_i)=F(\pi[(i-g-2)_{\partial'},2_\partial])=(0,i)_I$ for $i=g+1, g+2, ..., n$. 

As above, we can apply $\tau$ to get the other objects in the preinjective component. 
%We define, for $i=0,1, ..., g-1$ and $g$, 
%%
%\begin{eqnarray*} 
%F(\tau^j \gamma_i) & = & F(\tau^j\pi[-2_{\partial'},(i-g+2)_\partial])=F(\pi[(-2-j)_{\partial'},(i-g+2+j)_\partial]) \\ 
%  & =  & \tau^j(0,i)_I=(j,i)_I
%\end{eqnarray*}
%for $0\geq j \geq ghm$.
%
%\medskip 
%
%And for $i=g+1, g+2, ..., n$ we define
%\begin{eqnarray*}
%F(\tau^j \gamma_i) & = & F(\tau^j\pi[(i-g-2)_{\partial'},+2_\partial])=F(\pi[(i-g-2-j)_{\partial'},(2+j)_\partial]) \\
%  & = & \tau^j(0,i)_I=(j,i)_I
%\end{eqnarray*}
%for $0\geq j \geq ghm$. 

For the tube $\GG^0_{2m(n+1)}$ we only describe the function on the arcs at the top in the tube, 
because we only need those for the proof. Note that
we can easily describe the function on all objects by considering the rays of the arcs at the top.
The arc at the top is $\pi([0_\partial, (2m(n+1)+g+2)_{\partial}])$, so we set 
$F(\pi([0_\partial, (2gm(n+1)+g+2)_{\partial}]))=(2gm(n+1)+g,g)_g$. 
Then $F(\pi([k_\partial, (2gm(n+1)+g+2)_{\partial}])) = (2gm(n+1)+g-k,g)_g$ for $0 \leq k \leq g$. Also, we have
$F(\pi([0_\partial, (2gm(n+1)+g+2-k)_{\partial}])) = (2gm(n+1)+g-k,g-k)_g$, $0\leq k\leq g$.

The arc at the bottom of $\GG^{\infty}_{2m(n+1)}$ is 
$\pi [-2_{\partial'}, (2m(n+1)+h)_{\partial'}]$ and from this, one can describe the map $F$ 
similarly as for $\GG^0_{2m(n+1)}$. 

\begin{rem}
Observe that in the quiver 
$\GG_m$, 
the relations (c1),(c2), (e), (f) of Theorem~\ref{thm:hauptsatz} hold. 

To see that relations (c1) and (c2) hold, we use similar arguments as in the proof of Theorem 
\ref{thm:iso-quivers}, using the relations from Subsection~\ref{sec:relations-GG}. 

To see that (f) holds in $\GG_m$, one looks at the corresponding compositions into 
elementary moves, one long move and elementary moves 
and observes that both paths have the same effect on 
the arc corresponding to $(ghm,0)_P$. 
The proof of (e) then uses (f) with $j=0$ or $2m(n+1)+1$ and the fact, that these 
compositions will pass through two successive vertices at the boundary of the 
corresponding tube, hence the zero relation. 
A detailed geometric interpretation of these four relations is given in the appendix. 
\end{rem}

%%%%%%%%%%%%%%%%%%%%%%%%%
\subsection{Isomorphism}\label{ssec:isomorphism}

%%%%%%%%%%%%%%%%%%%%%%%%%

Recall that $Q_m'$ is the full subquiver of $Q_m$ on the vertices not belonging to the 
homogeneous tubes. 
Let $R$ denote the relations (c1),(c2),(e) and (f) on $Q_m'$, write $Q_m'/R$ for the 
quiver $Q_m'$ up to the relations in $R$. Then we can formulate our result. 

\begin{thm}\label{thm:iso-quivers}
Let $Q'_m$ be the full subquiver of $Q_m$ consisting of all objects not in a homogeneous tube. 
Then the quivers $\GG_m$ and $Q'_m/R$ are isomorphic.
\end{thm}
\begin{proof}
We already have that the components $\GG^*_m$ is isomorphic to $Q^*_m$ ($*\in \{P,I,0,\infty\}$).
We need to show that 
every long move in $\GG_m$ factors through a long move corresponding to 
one of the appropriate $\iota_0$, $\iota_{\infty}$, $\kappa_0$ or $\kappa_{\infty}$. 

W.l.o.g. assume we have a long move $f:\alpha\to \beta$ with $\alpha\in\GG_m^P$ and 
$\beta\in \GG_m^0$ (all other cases follow completely analogously). 
Since there exists a long move between $\alpha$ and $\beta$, their starting points in 
$P_{g,h}$ must be the same. 

We first describe the long moves corresponding to the $\iota_0(i)$ for $i=g,g-1,\dots, 1,0$. 
These maps start from $g+1$ different vertices of the last slice in $\GG_m^P$, namely from 
the $\tau^{-ghm}\beta_i$ with $i=g,g-1,\dots, 0$. By (\ref{eq:tau-orbit-beta}), they are 
arcs of the form $\pi[(-ghm+i-g)_{\partial},ghm_{\partial'}]$ for $i=g,\dots, 0$. 
We can choose lifts for them as follows: 
$$
[i_{\partial},(hm(n+1)+h)_{\partial'}] \mbox{ for } i=g,\dots, 0. 
$$
The $\iota_0(i)$ map these vertices to the $g+1$ vertices sitting on the ray at the top level 
of $\GG_m^0$, namely the vertices 
$\pi[i_{\partial},(2gm(n+1)+g+2)_{\partial}]$ for $i=g,\dots, 0$. 
We also use $\iota_0(i)$ to denote the corresponding long move, 
$$
\iota_0(i):[i_{\partial},(hm(n+1)+h)_{\partial'}] \to \pi[i_{\partial},(2gm(n+1)+g+2)_{\partial}]
$$

Observe that the rays in $\GG_m^P$ passing through the $g+1$ vertices 
$\pi[i_{\partial},(hm(n+1)+h)_{\partial'}]$ cover the whole component. Each of these rays 
consists of a (finite) sequence of elementary moves. In terms of arcs, every such ray 
starts at $\beta_i$, $i=g,\dots, 0$ and consists of all vertices obtained through a 
sequence of rotations around the common starting point of the arcs, the 
vertex $i$ of $P_{g,h}$

Similarly, $\GG_m^0$ is formed by the vertices of all corays between the top vertices 
$\pi[i_{\partial},(2gm(n+1)+g+2)_{\partial}]$ and the mouth. The corays in $\GG_m^0$ are 
formed all arcs obtained from the top one 
by sequence of rotations around the common starting point $i_{\partial}$. 

Denote the sequence of elementary moves (fixing starting points) 
from $\alpha$ to $\pi[i_{\partial},(hm(n+1)+h)_{\partial'}]$ 
by $(\mu_r)_r$ and the sequence of elementary moves (fixing starting points) from 
$\pi[i_{\partial},(2gm(n+1)+g+2)_{\partial}]$ to $\beta$ by $(\nu_t)_t$. 

Then the composition $(\nu_t)_t\circ\iota_0(i)\circ(\mu_r)_r$ is a long move from 
$\alpha$ to $\beta$. By using both commutativity relations for triangles between 
$\GG_m^P$ and $\GG_m^0$ (Subsection~\ref{sec:relations-GG}, 
for a geometric interpretation cf. Subsection~\ref{sssec:P-Tg}), it is then 
straightforward to see that $f$ is equal to this composition. 

\end{proof}

Recall the equivalence of categories $\mathcal C(kQ_m)\to \mathcal{J}_m Q_{g,h}$ from \cite{br}, see 
Theorem~\ref{thm:hauptsatz}. The category $\mathcal C(kQ_m)$ is the $k$-category of $Q_m$ subject 
to the relations $(a)-(g)$. 
As we have mentioned before, we do not consider the
homogeneous tubes, so we will use the equivalence $\mathcal C(kQ'_m)\to \mathcal{J'}_m Q_{g,h}$, 
where we take the full subcategories without the
homogeneous tubes. 

Now we define $\mathcal C(\GG)$ to be the $k$-category determined by $\GG$ and $\mathcal C(\GG_m)$ 
the $k$-category determined by $\GG_m$. Now, 
clearly $\mathcal C(\GG_m)$ is a full 
subcategory of $\mathcal C(\GG)$. Since in $\GG_m$ the relations 
(not involving the homogeneous tubes) are satisfied,
we have the following theorem.

\begin{thm}\label{thm:equ-categories}
Let $Q'_m$ be the full subquiver of $Q_m$ consisting of all objects not in a homogeneous tube. 
Then the category $\mathcal C(\GG_m)$ is
equivalent to $\mathcal C(kQ'_m)$ and $\mathcal{J'}_m Q_{g,h}$.
\end{thm}

%%%%%%%%%%%%%%% NEW SECTION %%%%%%%%%%%%%%%%%%%%%
\section{Relations and their geometric interpretation}\label{sec:relations}
%%%%%%%%%%%%%%%%%%%%%%%%%%%%%%%%%%%%

In this section, we give the geometric interpretation of the relations we have 
defined for $\GG$, using lifts of arcs in $\U$. 

Recall that $\GG$ is the quiver obtained from $\Gamma=$AR$\mo kQ_{g,h}$ by adding arrows 
corresponding to long moves, subject to additional relations (Sections~\ref{ssec:preproj-tube}, 
\ref{ssec:tube-preinj} and \ref{sec:relations-GG}). 
We first consider commuting squares. Let 
$$
\xymatrix@R=+0.6pc @C=+0.3pc{
 & Y_1\ar[rd]^{g_1} & \\ 
X\ar[ru]^{f_1}\ar[rd]_{f_2} & & Z \\
 & Y_2\ar[ru]_{g_2}
}
$$ 
be any diamond in $\GG$. 
Such a diamond commutes in the following cases: 

a) All the arrows in this diamond are elementary moves and hence 
all four vertices belong to the same 
connected component of $\Gamma$, with $\tau(Z)=X$. 

b) Two opposite arrows are elementary moves 
and the other two are long moves. In the latter case, 
$X$ and $Y_1$ (or $Y_2$) are in the same component (preprojective or a tube) of $\Gamma$ 
and $Y_2$ ($Y_1$, respectively) and $Z$ 
are in a common component of $\Gamma$ (a tube or the preinjective component). 

In both cases, the relations $g_1\circ f_1= g_2\circ f_2$ are imposed on $\GG$. 
We will illustrate the meshes in subsection~\ref{ssec:mesh-geometric}. They are all 
instances of the so-called Ptolemy relation: the two end terms $X$ and $Z$ of the diamond 
are viewed as the two different diagonals of a quadrilateral, the AR translate 
is the exchange of one of those diagonals by the other. This is called a flip. 

Then we give the geometric interpretation of the 
new relations for diamonds in \ref{ssec:diamonds-long}. 
We show that the new relations can also be viewed as Ptolemy relations by 
describing the quadrilateral in which the flip takes place.

%%%%%%%%%%%%%%%%%%%%%%%%%%%%%%%%%%%%
%
\subsection{Geometric interpretation of mesh relations: diamonds} \label{ssec:mesh-geometric}
%
%%%%%%%%%%%%%%%%%%%%%%%%%%%%%%%%%%%%

The geometric interpretation of a mesh relation in 
type $A$ is well known, cf. \cite[Section 2]{ccs} (for cluster categories of type $A$). 
The commutativity is the so-called Ptolemy relation: 
The two arcs 
$Y_1$ and $Y_2$ can be viewed as opposite boundary edges of a quadrilateral, 
with $X$ and $Z$ the two diagonals of this quadrilateral. Then the diagonals $X$ and $Z$ 
are related by a flip inside this polygon. 

\begin{center}
\psfragscanon
\psfrag{a1}{\tiny $a-1$}
\psfrag{a}{\tiny $a$}
\psfrag{b1}{\tiny $b-1$}
\psfrag{b}{\tiny $b$}
\psfrag{Y1}{\tiny $Y_1$}
\psfrag{Y2}{\tiny $Y_2$}
\psfrag{X}{\tiny $X$}
\psfrag{Z}{\tiny $Z$}
\psfrag{=}{$=$}
\psfrag{alpha}{\tiny\color{red}{$f_1$}}
\psfrag{beta}{\tiny\color{red}{$f_2$}}
\psfrag{gamma}{\tiny\color{red}{$g_1$}}
\psfrag{delta}{\tiny\color{red}{$g_2$}}
\includegraphics[scale=.33]{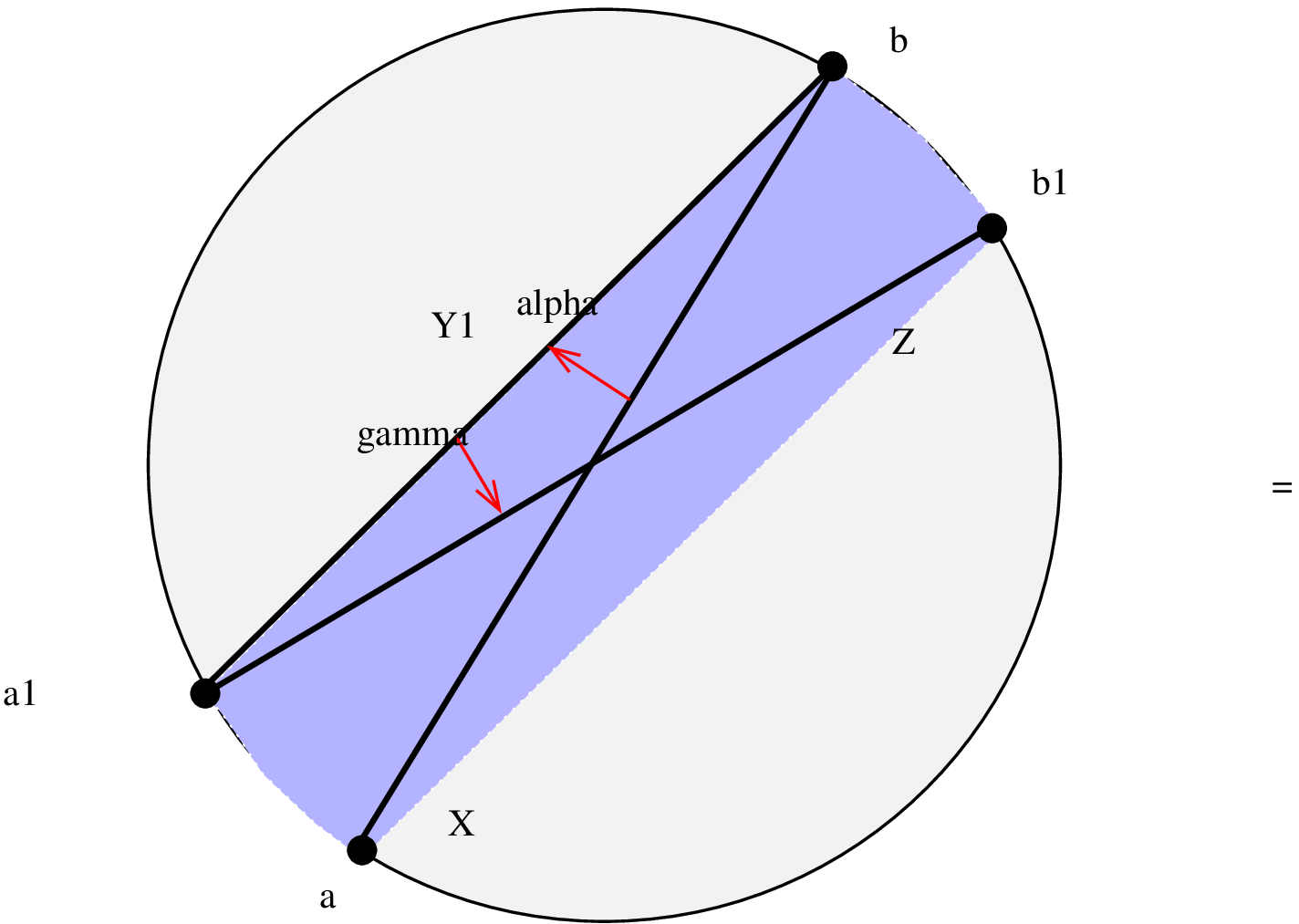}
\hskip 10pt 
\includegraphics[scale=.33]{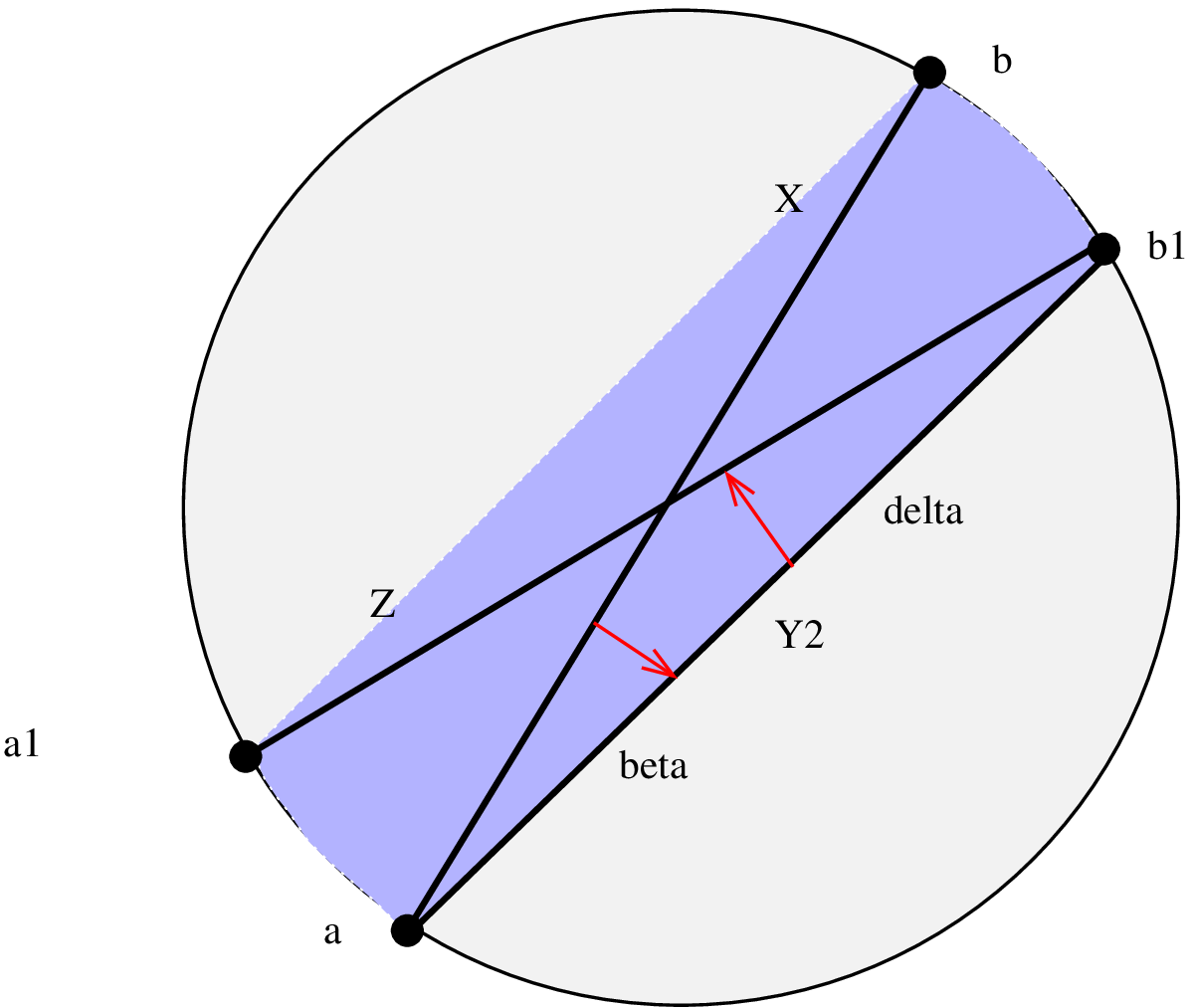}
\end{center}

The geometric reason for the commutativity is the Ptolemy relation: 
The two arcs 
$Y_1$ and $Y_2$ can be viewed as opposite boundary edges of a quadrilateral, 
with $X$ and $Z$ the two diagonals of this quadrilateral. Then the diagonals $X$ and $Z$ 
are related by a flip inside this polygon. 

\vskip 5pt

In $\Gamma$, we have four cases: 
meshes within $\Pp$ or $\Ip$, and meshes within the tubes. 
These meshes are all instances of the Ptolemy relation, we illustrate all 
cases briefly.

We consider the preprojective component, the preinjective case is completely 
analogous. 
Let  
$X= [i_{\partial},j_{\partial'}]$, $Y_1=[(i-1)_{\partial},j_{\partial'}]$, 
$Y_2=[i_{\partial},(j+1)_{\partial'}]$ 
and $Z=[(i-1)_{\partial},(j+1)_{\partial'}]$.

\begin{center}
\psfragscanon
\psfrag{4}{\tiny $j_{\partial'}$}
\psfrag{3}{\tiny $(j+1)_{\partial'}$}
\psfrag{2}{\tiny $i_{\partial}$}
\psfrag{1}{\tiny $(i-1)_{\partial}$}
\psfrag{X}{\tiny $X$}
\psfrag{Y1}{\tiny $Y_1$}
\psfrag{Y2}{\tiny $Y_2$}
\psfrag{Z}{\tiny $Z$}
\psfrag{X-Y1-Z}{\tiny $X\to Y_1\to Z$}
\psfrag{X-Y2-Z}{\tiny $X\to Y_2\to Z$}
\psfrag{=}{$=$}
\includegraphics[scale=.45]{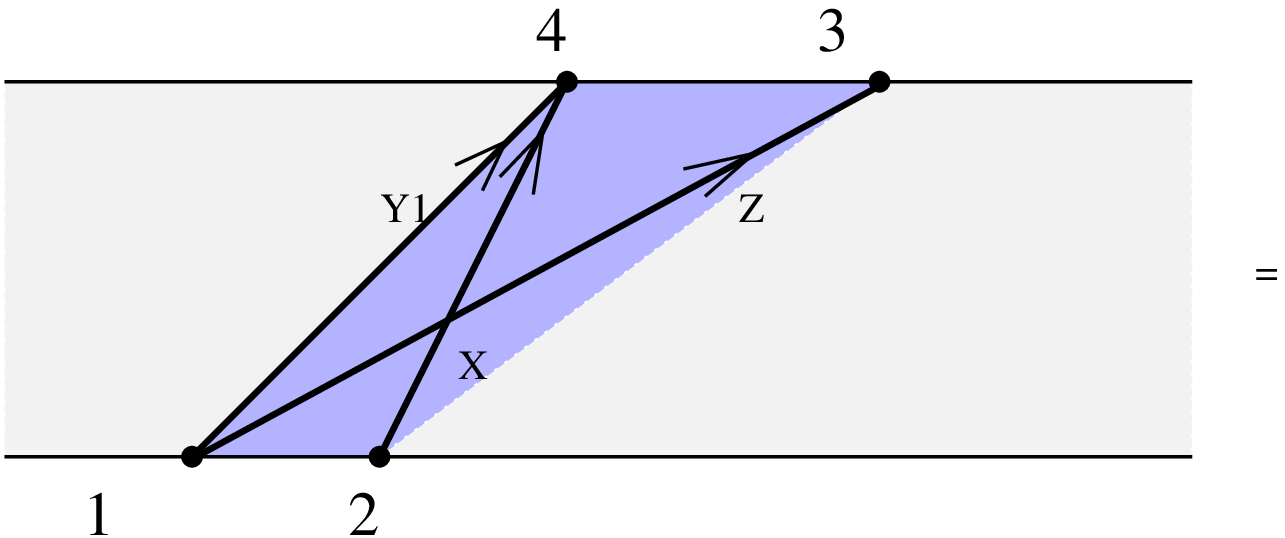}
\hskip .3cm
\includegraphics[scale=.45]{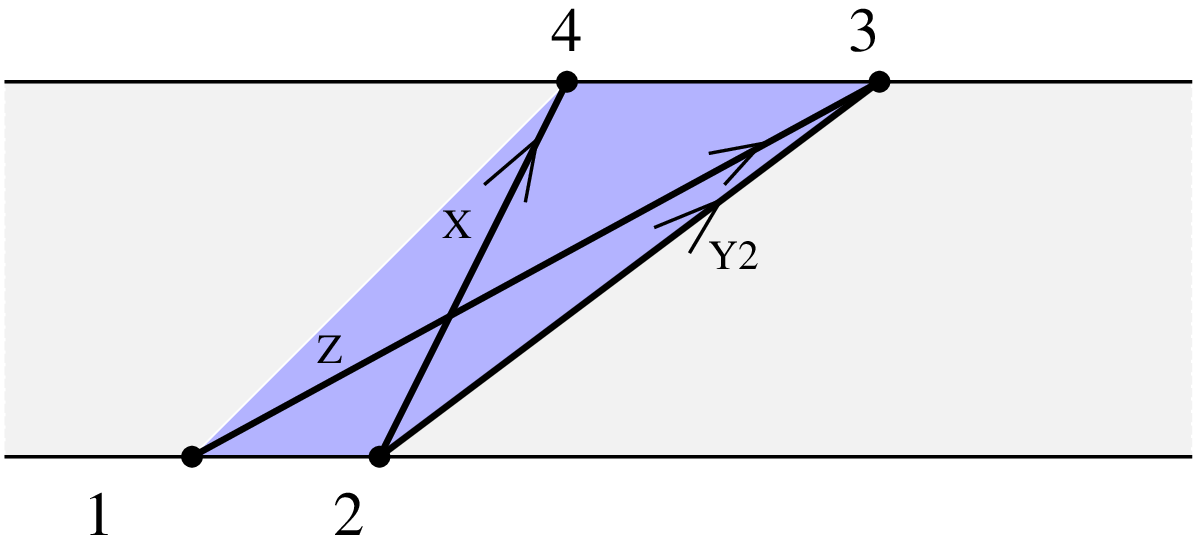}
\end{center}

Next we consider the regular tubes. 
W.l.o.g. let $X$ and $Y$ be indecomposable objects of the tube 
$\T_g$. 
Let $X=[i_{\partial},j_{\partial}]$, with $j\ge i+3$
$Y_1=[(i-1)_{\partial},j_{\partial}]$, $Y_2=[i_{\partial},(j-1)_{\partial}]$ and 
$Z=[(i-1)_{\partial},(j-1)_{\partial}]$.

\begin{center}
\psfragscanon
\psfrag{4}{\tiny $j$}
\psfrag{3}{\tiny $j-1$}
\psfrag{2}{\tiny $i$}
\psfrag{1}{\tiny $i-1$}
\psfrag{X}{\tiny $X$}
\psfrag{Y1}{\tiny $Y_1$}
\psfrag{Y2}{\tiny $Y_2$}
\psfrag{Z}{\tiny $Z$}
\psfrag{X-Y1-Z}{\tiny $X\to Y_1\to Z$}
\psfrag{X-Y2-Z}{\tiny $X\to Y_2\to Z$}
\psfrag{=}{$=$}
\includegraphics[height=1.8cm]{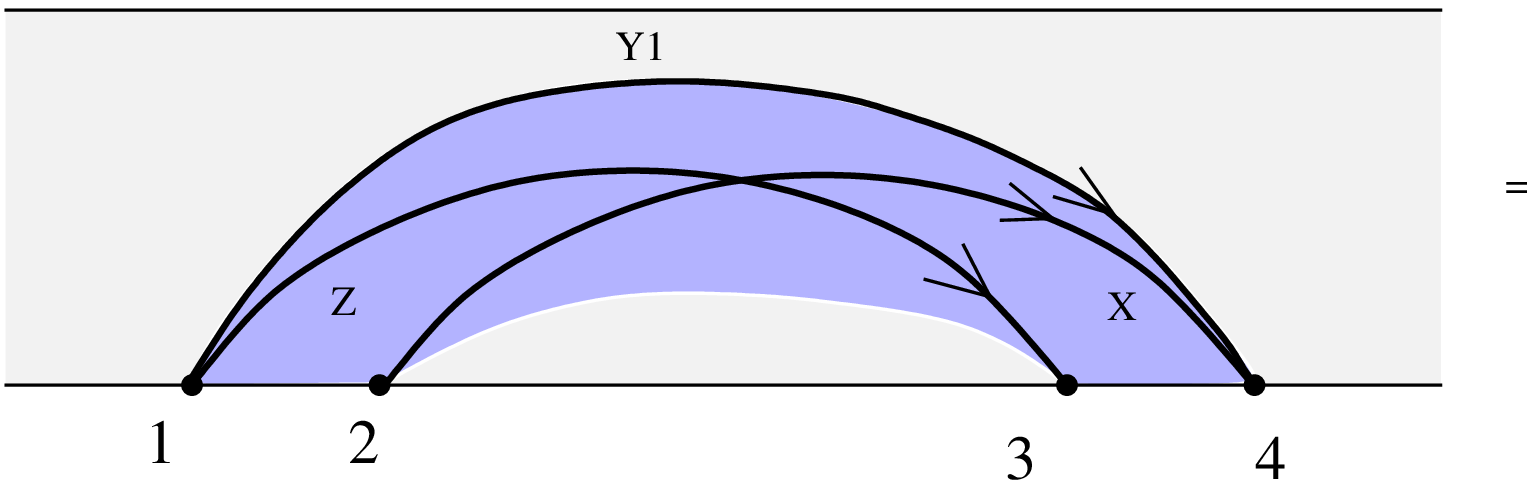}
\hskip .35cm
\includegraphics[height=1.8cm]{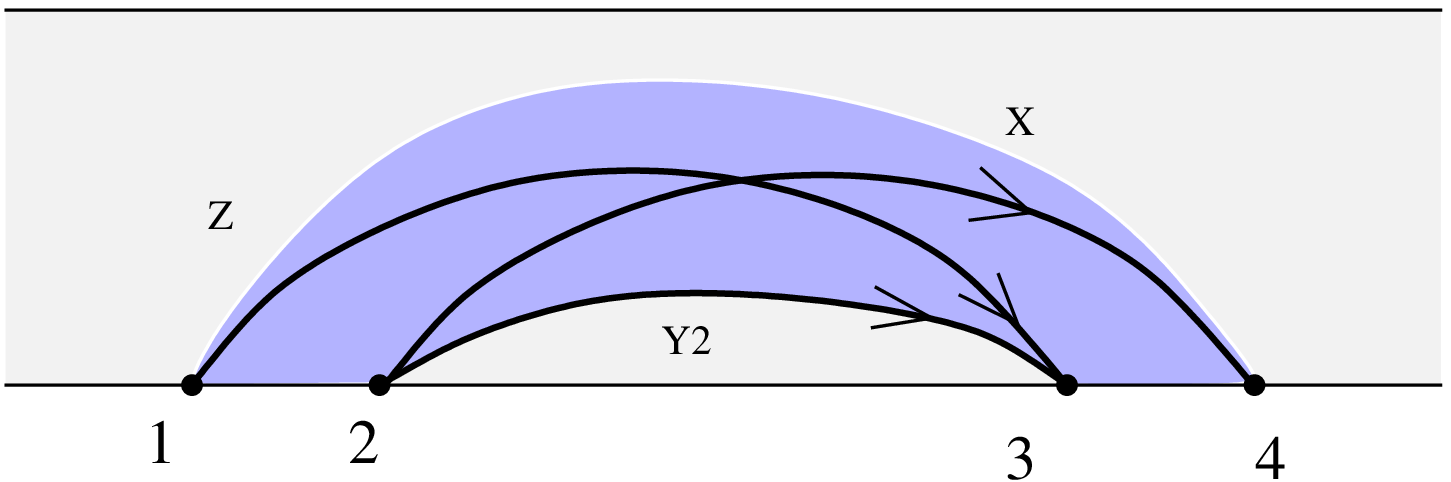}
\end{center}

%%%%%%%%%%%%%%%%%%%%%%%%%%%%%%%%%%%%
%
\subsection{Diamonds involving long moves} \label{ssec:diamonds-long}
%
%%%%%%%%%%%%%%%%%%%%%%%%%%%%%%%%%%%%

As before, long moves are indicated by dashed arrows. 

We have four cases: \\
A) $X$ and $Y_1$ belong to $\Pp$, 
$Y_2$ and $Z$ to $\T_g$, \\
B) $X$ and $Y_2$ belong to $\Pp$, $Y_1$ and $Z$ to $\T_h$, \\
C) $X$ and $Y_2$ belong to $\T_g$, $Y_1$ and $Z$ to $\Ip$, \\
D) $X$ and $Y_1$ belong to $\T_g$, $Y_2$ and $Z$ to $\Ip$. 

\medskip 

We will illustrate each of them by giving the quadrilateral of the Ptolemy relation.

%%%%%%%%%%%%%%%%%%%%%%%%%%%%%%%%%%%%
\subsubsection*{Case $A)$:} $\Pp$ to $\T_g$. 

$$
\xymatrix@R=+0.6pc @C=+0.3pc{
 & [(i-1)_{\partial},j_{\partial'}]\ar@{-->}[rrd] && \\ 
 [i_{\partial},j_{\partial'}]\ar[ru]\ar@{-->}[rrd]  && & [(i-1)_{\partial},k_{\partial}]  \\
 && [i_{\partial}, k_{\partial}]\ar[ru]  & 
}
$$ 
with $k\ge i+2$. 

\begin{center}
\psfragscanon
\psfrag{4}{\tiny $j_{\partial'}$}
\psfrag{3}{\tiny $k_{\partial}$}
\psfrag{2}{\tiny $i_{\partial}$}
\psfrag{1}{\tiny $(i-1)_{\partial}$}
\psfrag{X}{\tiny $X$}
\psfrag{Y1}{\tiny $Y_1$}
\psfrag{Y2}{\tiny $Y_2$}
\psfrag{Z}{\tiny $Z$}
\psfrag{X-Y1-Z}{\tiny $X\to Y_1\to Z$}
\psfrag{X-Y2-Z}{\tiny $X\to Y_2\to Z$}
\psfrag{=}{$=$}
\psfrag{alpha}{\tiny\color{red}{$f_1$}}
\psfrag{beta}{\tiny\color{red}{$f_2$}}
\psfrag{gamma}{\tiny\color{red}{$g_1$}}
\psfrag{delta}{\tiny\color{red}{$g_2$}}
\includegraphics[scale=.4]{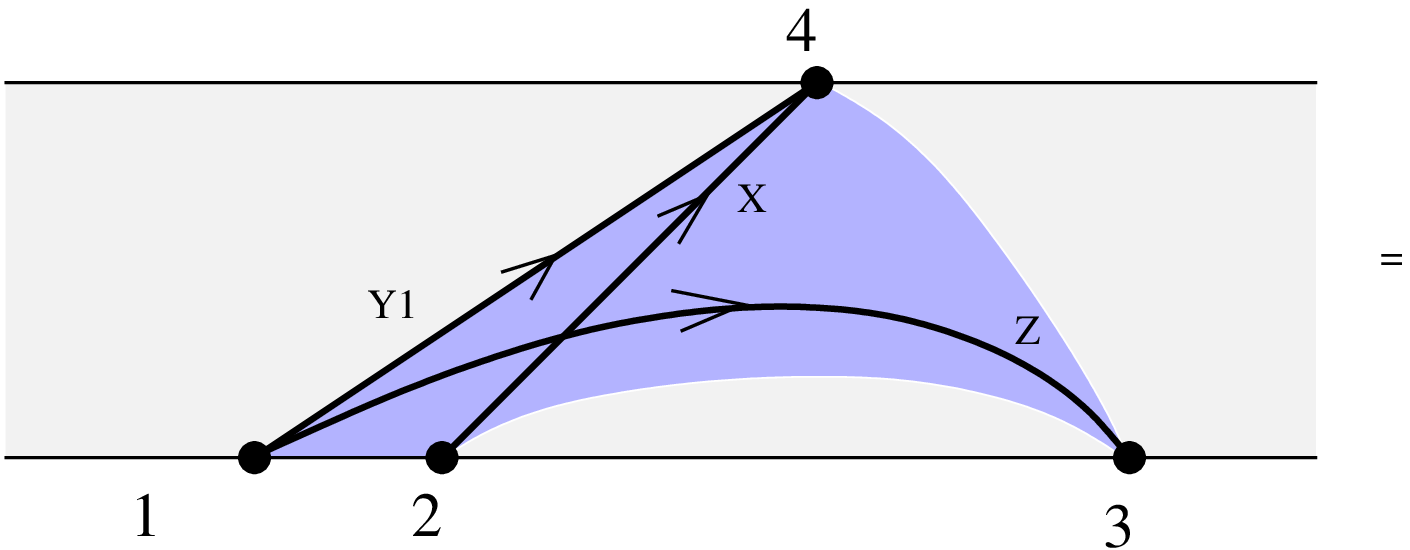}
\hskip .3cm
\includegraphics[scale=.4]{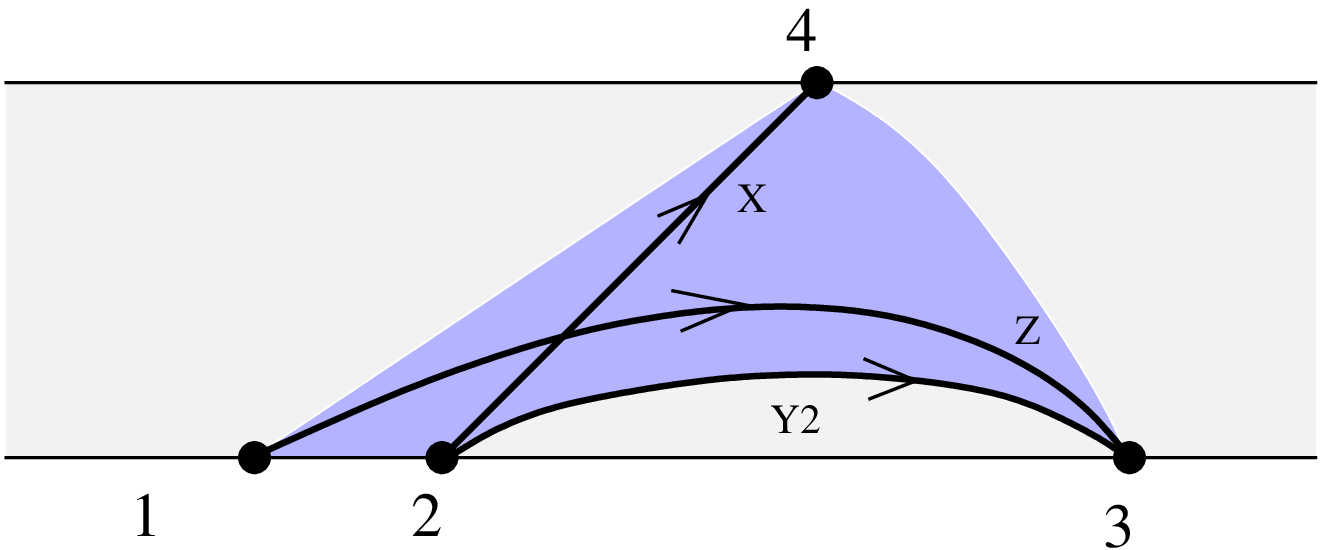}
\end{center}

%%%%%%%%%%%%%%%%%%%%%%%%%%%%%%%%%%%%
\subsubsection*{Case $B)$:} $\Pp$ to $\T_h$. 

$$
\xymatrix@R=+0.6pc @C=+0.3pc{
 &&&  &[m_{\partial'},j_{\partial'}]\ar[rd] \\ 
 &&  [i_{\partial},j_{\partial'}]\ar@{-->}[rru]\ar[rd] &&& [m_{\partial'},(j+1)_{\partial'}] \\
 &&&  [i_{\partial}, (j+1)_{\partial'}]\ar@{-->}[rru] 
}
$$ 
with $m\le j-2$.  

\begin{center}
\psfragscanon
\psfrag{4}{\tiny $(j+1)_{\partial'}$}
\psfrag{3}{\tiny $j_{\partial'}$}
\psfrag{2}{\tiny $m_{\partial'}$}
\psfrag{1}{\tiny $i_{\partial}$}
\psfrag{X}{\tiny $X$}
\psfrag{Y1}{\tiny $Y_1$}
\psfrag{Y2}{\tiny $Y_2$}
\psfrag{Z}{\tiny $Z$}
\psfrag{X-Y1-Z}{\tiny $X\to Y_1\to Z$}
\psfrag{X-Y2-Z}{\tiny $X\to Y_2\to Z$}
\psfrag{=}{$=$}
\includegraphics[height=3cm]{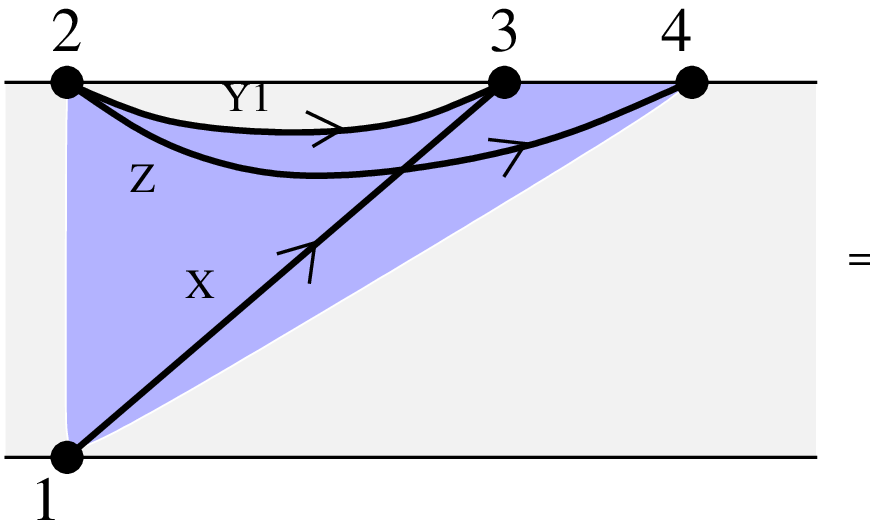}
\hskip .3cm
\includegraphics[height=3cm]{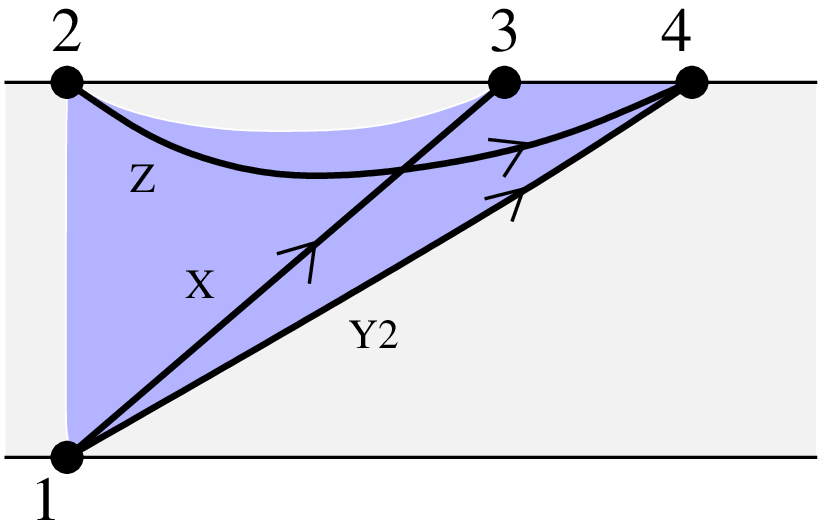}
\end{center}

%%%%%%%%%%%%%%%%%%%%%%%%%%%%%%%%%%%%
\subsubsection*{Case $C)$:} $\T_g$ to $\Ip$.

$$
\xymatrix@R=+0.6pc @C=+0.3pc{
 &&&  &[k_{\partial'},j_{\partial}]\ar[rd] \\ 
 &&  [i_{\partial},j_{\partial}]\ar@{-->}[rru]\ar[rd] &&& [k_{\partial'},(j-1)_{\partial}] \\
 &&&  [i_{\partial}, (j-1)_{\partial}]\ar@{-->}[rru] 
}
$$ 
with $j\ge i+3$. 
\begin{center}
\psfragscanon
\psfrag{4}{\tiny $k_{\partial'}$}
\psfrag{3}{\tiny $j_{\partial}$}
\psfrag{2}{\tiny $(j-1)_{\partial}$}
\psfrag{1}{\tiny $i_{\partial}$}
\psfrag{X}{\tiny $X$}
\psfrag{Y1}{\tiny $Y_1$}
\psfrag{Y2}{\tiny $Y_2$}
\psfrag{Z}{\tiny $Z$}
\psfrag{X-Y1-Z}{\tiny $X\to Y_1\to Z$}
\psfrag{X-Y2-Z}{\tiny $X\to Y_2\to Z$}
\psfrag{=}{$=$}
\includegraphics[height=2.3cm]{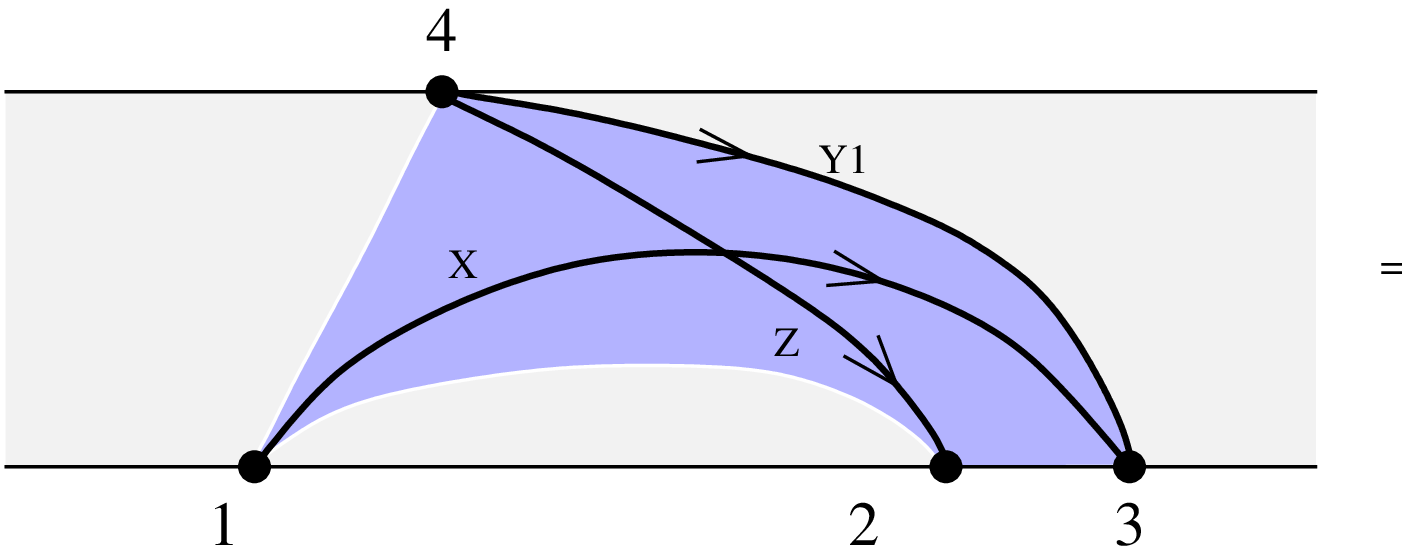}
\hskip .3cm
\includegraphics[height=2.3cm]{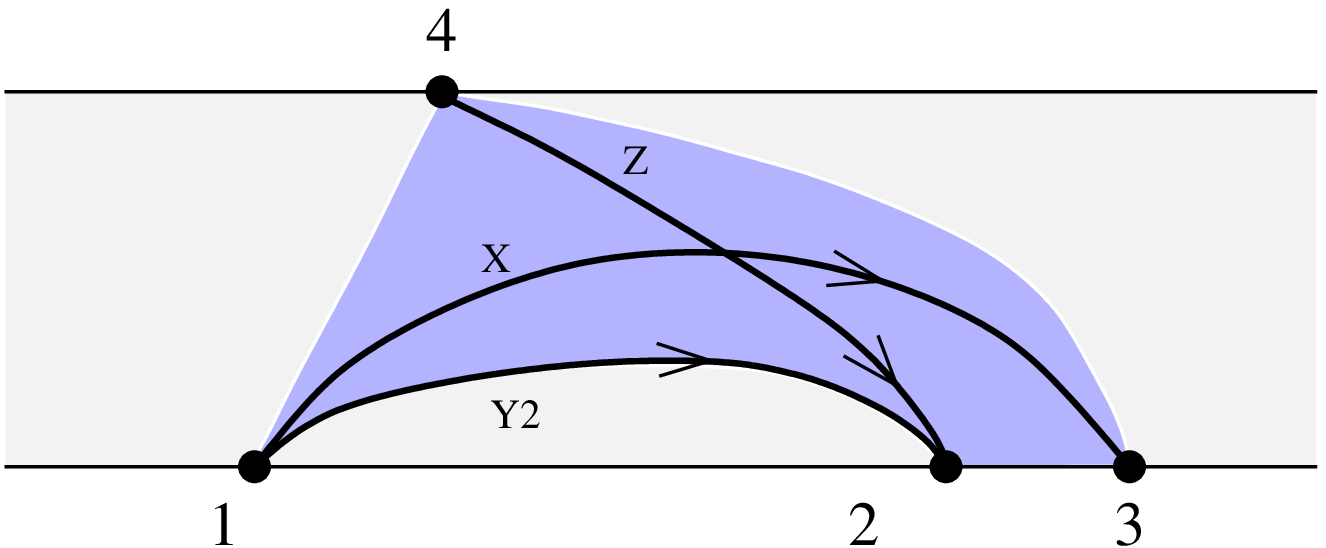}
\end{center}

%%%%%%%%%%%%%%%%%%%%%%%%%%%%%%%%%%%%
\subsubsection*{Case $D)$:} $\T_h$ to $\Ip$. 

$$
\xymatrix@R=+0.6pc @C=+0.3pc{
 & [(i+1)_{\partial'},j_{\partial'}]\ar@{-->}[rrd] && \\ 
 [i_{\partial'},j_{\partial'}]\ar[ru]\ar@{-->}[rrd]  && & [(i+1)_{\partial'},k_{\partial}]  \\
 && [i_{\partial'}, k_{\partial}]\ar[ru]  & 
}
$$ 
with $j\ge i+3$.

\begin{center}
\psfragscanon
\psfrag{4}{\tiny $i_{\partial'}$}
\psfrag{3}{\tiny $(i+1)_{\partial'}$}
\psfrag{2}{\tiny $j_{\partial'}$}
\psfrag{1}{\tiny $k_{\partial}$}
\psfrag{X}{\tiny $X$}
\psfrag{Y1}{\tiny $Y_1$}
\psfrag{Y2}{\tiny $Y_2$}
\psfrag{Z}{\tiny $Z$}
\psfrag{X-Y1-Z}{\tiny $X\to Y_1\to Z$}
\psfrag{X-Y2-Z}{\tiny $X\to Y_2\to Z$}
\psfrag{=}{$=$}
\includegraphics[height=3cm]{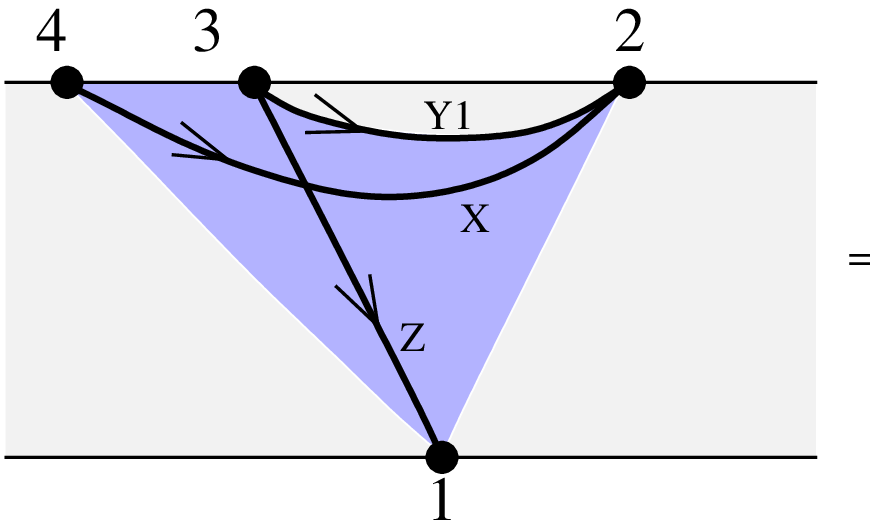}
\hskip .3cm
\includegraphics[height=3cm]{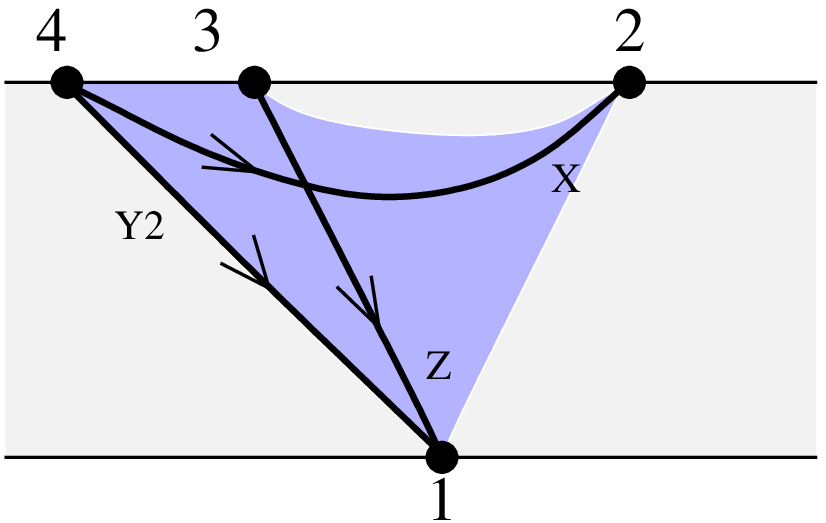}
\end{center}

%%%%%%%%%%%%%%% NEW SUBSECTION %%%%%%%%%%%%%%%%%%%%%
%
\subsection{Commuting triangles involving long moves, geometrically} \label{sec:triangles}

In $\GG$, there are also commuting triangles, involving two long moves and an elementary 
move. 
$$
\xymatrix@R=+0.6pc @C=+0.6pc{
X\ar[rd]_{f}\ar@/^/@{-->}[rrrdd]^h & & \\
 & Y\ar@{-->}[rrd]_{g} & \\ 
 & & & Z\\
}
\hskip .5cm
\xymatrix@R=+0.6pc @C=+0.6pc{
X\ar@{-->}[rrd]^{f}\ar@/_/@{-->}[rrrdd]_h & & \\
 & &Y\ar[rd]^{g} & \\ 
 & & & Z\\
}
\hskip .7cm
\xymatrix@R=+0.6pc @C=+0.6pc{
 & && Z \\
 & Y\ar@{-->}[rru]^{g} & \\ 
X\ar[ru]^{f}\ar@/_/@{-->}[rrruu]_h & & \\
}
\hskip .5cm
\xymatrix@R=+0.6pc @C=+0.6pc{
 & && Z \\
 && Y\ar[ru]_{g} & \\ 
X\ar@{-->}[rru]_{f}\ar@/^/@{-->}[rrruu]^h & & \\
}
$$ 
Recall that arrows for moves fixing the endpoint of an arc go 
up, arrows for moves fixing the starting point go down. There are twice four cases: 
two from $\Pp$ to 
$\T_g$ or to $\T_h$, two from $\T_g$ or from $\T_h$ to $\Ip$.

%%%%%%%%%%%%%%%%%%%%%%%%%%%
%
\subsubsection{Commuting triangles with long maps from $\Pp$ to $\T_g$}\label{sssec:P-Tg}
%%%%%%%%%%%%%%%%%%%%%%%%%%%
Let $[i_{\partial},j_{\partial'}]$ 
be a preprojective arc. Then for every $m \ge i+ 2$ there are 
two cases of commuting triangles 
$g\circ f=h$ involving two long {\em and one elementary} move.

$$
\xymatrix@R=1.5pc@C=2pc@!0{ 
 [i_{\partial}, j_{\partial'}] \ar[rdd]_f \ar@{-->}@/^1pc/[rrrrdddd]^h&&  \\
\\ 
& [i_{\partial}, (j+1)_{\partial'}] \ar@{-->}[rrrdd]_g &&&  &&  \\ 
\\
 &&& & [i_{\partial}, m_{\partial}] 
}
\hskip 1cm
\xymatrix@R=1.5pc@C=2pc@!0{ 
 [i_{\partial}, j_{\partial'}] \ar@{-->}[rrrdd]^f \ar@{-->}@/_1pc/[rrrrdddd]_h&&  \\
\\ 
&&& [i_{\partial}, (m+1)_{\partial}] \ar[rrdd]^g &&&  &&  \\ 
\\
 &&& && [i_{\partial}, m_{\partial}] 
}
$$

\begin{center}
\psfragscanon
\psfrag{1}{\tiny $i$}
\psfrag{2}{\tiny $m_{\partial}$}
\psfrag{6}{\tiny $m_{\partial}$}
\psfrag{3}{\tiny $j+1_{\partial'}$}
\psfrag{4}{\tiny $j_{\partial'}$}
\psfrag{5}{\tiny $m+1_{\partial}$}
\psfrag{A1}{$(I)$}
\psfrag{A2}{$(II)$}
\psfrag{alpha}{\tiny\color{red}{$f$}}
\psfrag{beta}{\tiny\color{red}{$g$}}
\psfrag{gamma}{\tiny\color{red}{$h$}}
\includegraphics[height=3cm]{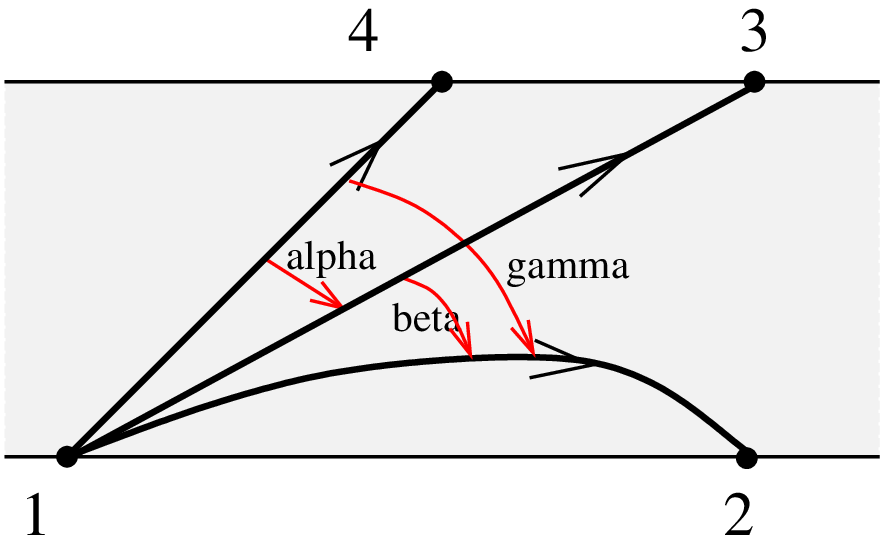}
\hskip 1cm 
\includegraphics[height=3cm]{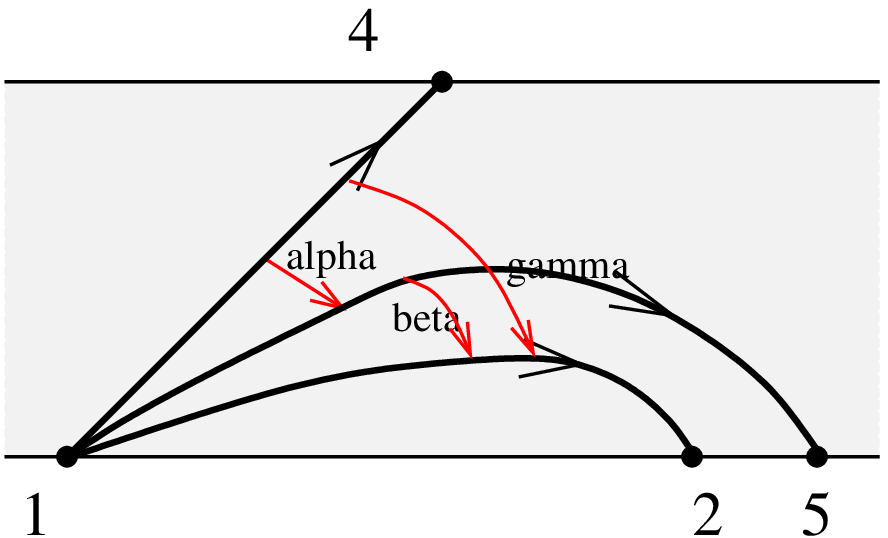}
\end{center}

%%%%%%%%%%%%%%%%%%%%%%%%%%%
%
\subsubsection{Commuting triangles from $\Pp$ to $\T_h$}\label{sssec:P-Th}
Let $[i_{\partial},j_{\partial'}]$ be a preprojective arc. Then for every $k\le j-2$, 
there are two kind of commuting triangles $g\circ f=h$ involving one elementary and two long 
moves.

$$
\xymatrix@R=1.5pc@C=2pc@!0{ 
 &&& & [k_{\partial}, j_{\partial'}] \\
\\ 
& [(i+1)_{\partial}, j_{\partial'}] \ar@{-->}[rrruu]^g &&&  &&  \\ 
\\
 [i_{\partial}, j_{\partial'}] \ar[ruu]^f \ar@{-->}@/_1pc/[rrrruuuu]_h&&  \\
}
\hskip 1cm
\xymatrix@R=1.5pc@C=2pc@!0{ 
  &&& && [k_{\partial}, j'_{\partial}]  \\
\\ 
&&& [(k-1)_{\partial}, j'_{\partial}] \ar[rruu]_g &&&  &&  \\ 
\\
 [i_{\partial}, j_{\partial'}] \ar@{-->}[rrruu]_f \ar@{-->}@/^1pc/[rrrruuuu]^h&& 
}
$$

\begin{center}
\psfragscanon
\psfrag{1}{\tiny $i_{\partial}$}
\psfrag{2}{\tiny $k-1_{\partial'}$}
\psfrag{3}{\tiny $j_{\partial'}$}
\psfrag{5}{\tiny $i-1_{\partial}$}
\psfrag{6}{\tiny $k_{\partial'}$}
\psfrag{alpha}{\tiny\color{red}{$f$}}
\psfrag{beta}{\tiny\color{red}{$g$}}
\psfrag{gamma}{\tiny\color{red}{$h$}}
\includegraphics[height=3cm]{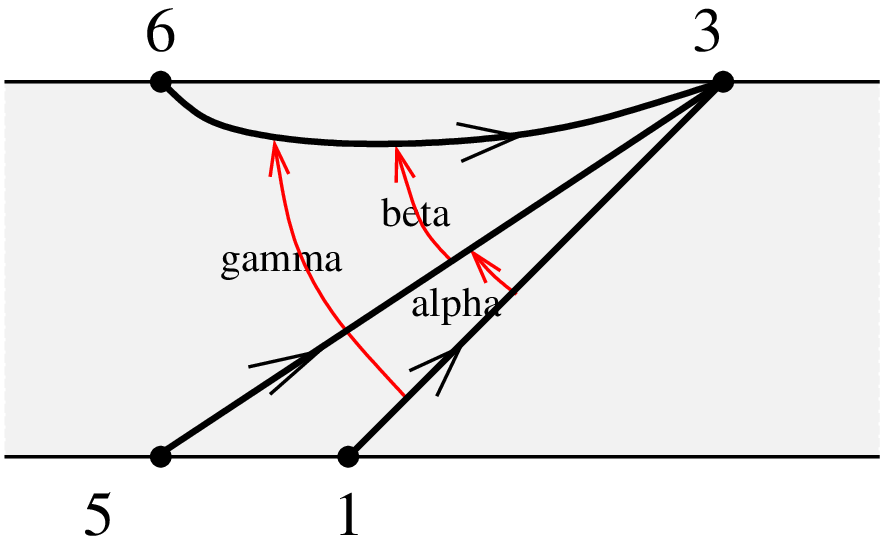}
\hskip 1cm
\includegraphics[height=3cm]{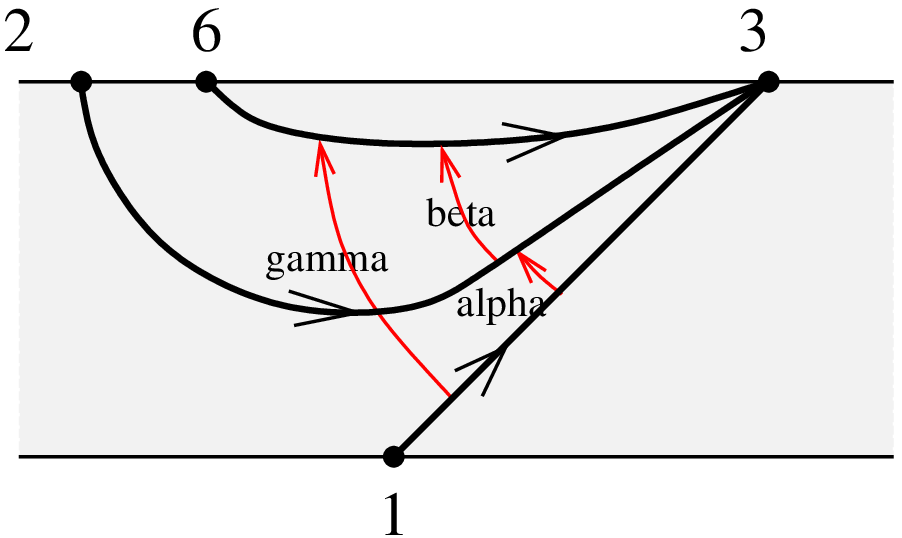}
\end{center}

%%%%%%%%%%%%%%%%%%%%%%%%%%%
%
\subsubsection{Commuting triangles from $\T_g$ to $\Ip$}\label{sssec:Tg-I}
%%%%%%%%%%%%%%%%%%%%%%%%%%%

Let $[j_{\partial},i_{\partial}]$ be a peripheral arc (i.e. $j\le i-2$). 
The two kind of commuting triangles $g\circ f=h$ involving elementary and long moves 
end at an arbitrary preinjective arc $[k_{\partial'},i_{\partial}]$: 
$$
\xymatrix@R=1.5pc@C=2pc@!0{ 
 &&& & [k_{\partial'}, j_{\partial}] \\
\\ 
& [(i-1)_{\partial}, j_{\partial}] \ar@{-->}[rrruu]^g &&&  &&  \\ 
\\
 [i_{\partial}, j_{\partial}] \ar[ruu]^f \ar@{-->}@/_1pc/[rrrruuuu]_h&&  \\
}
\hskip 1cm
\xymatrix@R=1.5pc@C=2pc@!0{ 
  &&& && [k_{\partial'}, j_{\partial}]  \\
\\ 
&&& [(k-1)_{\partial'}, j_{\partial}] \ar[rruu]_g &&&  &&  \\ 
\\
 [i_{\partial}, j_{\partial}] \ar@{-->}[rrruu]_f \ar@{-->}@/^1pc/[rrrruuuu]^h&& 
}
$$

\begin{center}
\psfragscanon
\psfrag{1}{\tiny $i_{\partial}$}
\psfrag{2}{\tiny $i-1_{\partial}$}
\psfrag{3}{\tiny $j_{\partial}$}
\psfrag{4}{\tiny $k_{\partial'}$}
\psfrag{5}{\tiny $k-1_{\partial'}$}
%\psfrag{6}{\tiny $j-k-1_{\partial}$}
\psfrag{alpha}{\tiny\color{red}{$f$}}
\psfrag{beta}{\tiny\color{red}{$g$}}
\psfrag{gamma}{\tiny\color{red}{$h$}}
\includegraphics[height=3.1cm]{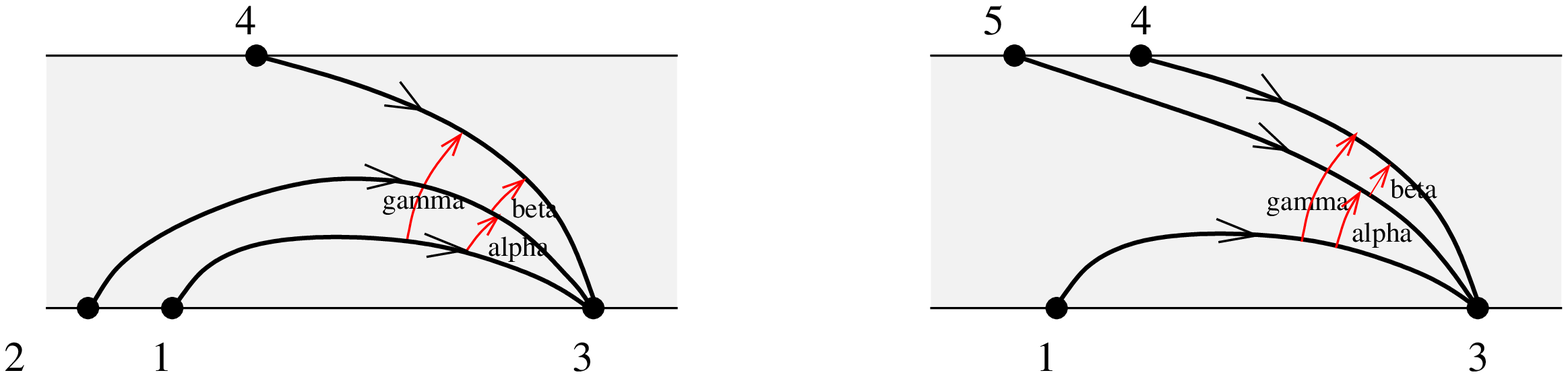}
\end{center}

%%%%%%%%%%%%%%%%%%%%%%%%%%%
\subsubsection{Commuting triangles from $\T_h$ to $\Ip$}\label{sssec:Th-I}
%%%%%%%%%%%%%%%%%%%%%%%%%%%

Let $[i_{\partial'},j_{\partial'}]$ be a peripheral arc (i.e. $i\le j-2$). 
The two kind of commuting triangles $g\circ f=h$ involving elementary and long moves 
end at an arbitrary preinjective arc $[i_{\partial'},m_{\partial'}]$: 

$$
\xymatrix@R=1.5pc@C=2pc@!0{ 
 [i_{\partial'}, j_{\partial'}] \ar[rdd]_f \ar@{-->}@/^1pc/[rrrrdddd]^h&&  \\
\\ 
& [i_{\partial'}, (j+1)_{\partial'}] \ar@{-->}[rrrdd]_g &&&  &&  \\ 
\\
 &&& & [i_{\partial'}, m_{\partial}] 
}
\hskip 1cm
\xymatrix@R=1.5pc@C=2pc@!0{ 
 [i_{\partial'}, j_{\partial'}] \ar@{-->}[rrrdd]^f \ar@{-->}@/_1pc/[rrrrdddd]_h&&  \\
\\ 
&&& [i_{\partial'}, (m+1)_{\partial}] \ar[rrdd]^g &&&  &&  \\ 
\\
 &&& && [i_{\partial'}, m_{\partial}] 
}
$$

\begin{center}
\psfragscanon
\psfrag{1}{\tiny $i_{\partial'}$}
\psfrag{2}{\tiny $m_{\partial}$}
\psfrag{3}{\tiny $j_{\partial'}$}
\psfrag{4}{\tiny $j$}
\psfrag{5}{\tiny $m+1_{\partial}$}
\psfrag{6}{\tiny $j+1_{\partial'}$}
\psfrag{alpha}{\tiny\color{red}{$f$}}
\psfrag{beta}{\tiny\color{red}{$g$}}
\psfrag{gamma}{\tiny\color{red}{$h$}}
\includegraphics[height=3.1cm]{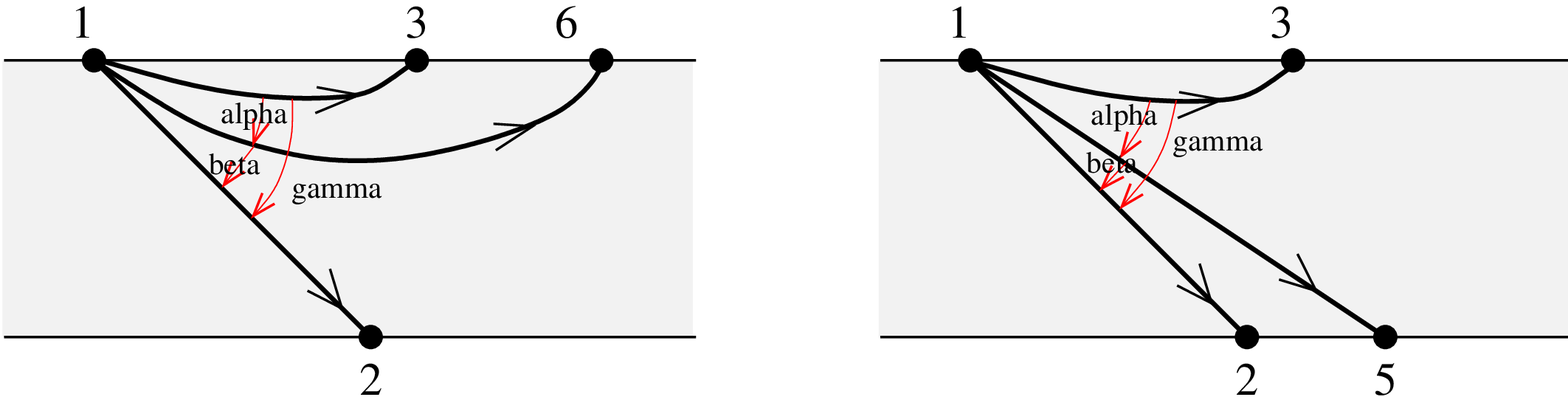}
\end{center}

%%%%%%%%%%%%%%% NEW SUBSECTION %%%%%%%%%%%%%%%%%%%%%
%
\subsection{Zero relations, geometrically} \label{ssec:zero-triangle}

It is well known that there are zero relations at the mouth of the tubes. Geometrically, 
they can be viewed as special cases of the diamond relations in the 
tubes, namely the case where $Y_2$ (or $Y_1$) becomes zero, i.e. where the vertex 
$X$ is of the form $[i_{\partial},(i+2)_{\partial}]$ or $[i_{\partial'},(i+2)_{\partial'}]$ 
see Subsection~\ref{ssec:mesh-geometric}

%%%%%%%%%%%%%%%%%%%%%%%%%%
%
\section{Application to the cluster category of affine type $A$}\label{sec:cluster-cat}
%
%%%%%%%%%%%%%%%%%%%%%%%%%%

The cluster category of type $\tilde{A}_n$ has been introduced by \cite{bmrrt}. It is 
by definition the orbit category 
$D^b(\mo\tilde{A})/\tau^{-1}\circ[1]$. The effect of taking such a quotient is most visibly 
on the AR quiver: the cluster category has a transjective component, arising from 
the components $\Pp$ and $\Ip$ of the module category. 

Using unoriented versions of our arcs, we can describe the AR-quiver of the cluster 
category. The only missing ingredient is the slice linking the preprojective component with 
the preinjective component. We also write $\tau$ for the translation map on unoriented 
arcs, induced by $i\mapsto i+1$ on $\partial$ and $i\mapsto i-1$ on $\partial'$. 

If $\alpha\in\GG$ is an oriented arc, let $\underline{\alpha}$ be the unoriented version of 
$\alpha$. Furthermore, let $\underline{\eta_0}$, $\underline{\eta_1}$, $\dots$, 
$\underline{\eta_n}$ be the arcs $\tau^{-1}(\underline{\gamma_i})=\tau(\underline{\beta_i})$. 

Let $\GGG$ be the quiver whose vertices are the 
$\underline{\alpha}$, $\alpha\in \GG$, together with the $\underline{\eta_i}$. Its arrows are 
the arrows of $\GG$ together with the obvious arrows from the $\underline{\gamma_i}$ 
to the $\underline{\eta_i}$ and from the $\underline{\eta_i}$ to the $\underline{\beta_i}$, subject to the 
relations from $\GG$ and the additional mesh relations around the new slice. 
The new quiver $\GGG$ is a stable translation quiver. 
Up to the additional slice of vertices, the quiver $\GGG$ looks like $\GG$, the latter is isomorphic to 
a full 
subquiver of $\GGG$. $\GGG$ has components $\GGG^0$ and $\GGG^{\infty}$ 
consisting of peripheral arcs and the transjective component of the arcs from 
$\GG^P\cup \GG^I$ together with $\{\underline{\eta_i}\mid i=0\dots, n\}$. Write 
$\GGG_m^{Tr}$ for the full subquiver on the vertices of $\GG^P\cup\GG^I$ and 
the $\{\underline{\eta_i}\mid i=0\dots, n\}$. 

\begin{dfn}
We define $\GGG_m$ to be the full subquiver of $\GGG$ consisting of the vertices in 
$\GGG_m^{Tr}$, $\GGG_m^0$ and $\GGG_m^{\infty}$. 
\end{dfn}

Let $\mathcal J'_m(\mathcal{C}_{\tilde{A}})$ be the full subcategory of the cluster category 
$\mathcal{C}_{\tilde{A}}$ of type $Q_{g,h}$ consisting of all objects who correspond to 
vertices in $\GGG_m$. Then the following is a direct consequence of 
Theorems \ref{thm:iso-quivers} and \ref{thm:equ-categories}. 

\begin{cor}
The $k$-category of $\GGG_m$ is equivalent to $\mathcal J'_m(\mathcal{C}_{\tilde{A}})$. 
\end{cor}

%What I had wanted to say is that the k-category of the (truncated quiver) QQQ_m is 
%equivalent to the truncated version of the cluster category. 
%sorry about the confusion. 
%
%The AR quiver of the truncated version of the 
%cluster category is just the \Gamma_m, with the additional slice
%or, maybe better: the AR-quiver of the cluster category is the quiver \Gamma with the 
%additional slice.
%

\noindent \textbf{Acknowledgements}
KB thanks NTNU for supporting this project and the Mittag Leffler institute for inviting her 
to the program on Representation Theory in 2015.

\appendix
%%%%%%%%%%%%%%% NEW SECTION %%%%%%%%%%%%%%%%%%%%%
%
\section{Relations appearing in Theorem~\ref{thm:hauptsatz}}%\label{sec:relations-geometric}
%%%%%%%%%%%%%%% %%%%%%%%%%%%%%%%%%%%%

We use the geometric interpretation of relations to visualize the 
relations in Theorem~\ref{thm:hauptsatz}

Relations (a) and (b) of the theorem are the usual mesh relations, they have already 
been described above.

\subsection*{Relations (c1),(c2)} 

Relation (c1) has on one side $h$ compositions of elementary moves between preprojective 
arcs, followed by a long move and then $g$ elementary moves in the tube of rank $g$, 
we can write it as 
$$
\xymatrix@R=1.5pc@C=4pc@!0{ 
X_0\ar[r]^{\alpha} & X_1\ar[r]^{\alpha} & \dots\ar[r]^{\alpha} & X_h\ar@{-->}[r]^{\iota_0(0)} 
& Y_g\ar[r]^{\pi} & Y_{g-1}\ar[r]^{\pi} & \dots\ar[r]^{\pi} & Y_0  
% [i_{\partial'}, j_{\partial'}] \ar[rdd]_f \ar@{-->}@/^1pc/[rrrrdddd]^h&&  \\
%\\ 
%& [i_{\partial'}, (j+1)_{\partial'}] \ar@{-->}[rrrdd]_g &&&  &&  \\ 
%\\
% &&& & [i_{\partial'}, m_{\partial}] 
}
$$
dropping the subscripts on the $\alpha$ and $\pi$. 
The arcs corresponding to these vertices all have lifts starting at $0_{\partial}$, 
the $X_i$ end at $\partial'$ and the $Y_i$ correspond to peripheral arcs. More precisely, 
we can choose a lift for $X_0$, 
$X_0=[-ghm_{\partial},ghm_{\partial'}]$. 
Under the arrows $\alpha$, 
the arcs get lengthened 
by elementary moves fixing the ending point 
$-ghm_{\partial}$, up to the lift 
$[-ghm_{\partial},(ghm+h)_{\partial'}]$
$X_h$. 

To prove that in our geometric set-up the first relation of (c1) holds, we can use the commuting 
triangle on the left hand side of Subsection~\ref{sssec:P-Th} to see that $\iota_0(0)\circ\alpha$ 
is a long move $\mu_{h-1}$ from $X_{h-1}$ to $Y_g$. Using this relation again, we see that 
$\alpha\circ\mu_{h-1}$ is a long move $\mu_{h-2}$ from $X_{h-2}$ to $Y_g$, etc. This proves 
that the composition $\iota_0(0)\circ\alpha_P^g$ is equal to a long move $\mu_0:X_0\to Y_g$. 
Now we use the commuting triangle on the right hand side of Subsection~\ref{sssec:P-Th} $h$ times 
to 
iteratedly replace the long move from $X_0$ composed with $\pi$ by another long move from 
$X_0$ to $Y_i$, $i=g-1,\dots, 0$. 

The other three relations in (c1) and (c2) work completely analogously, using Subsections 
~\ref{sssec:P-Th}, ~\ref{sssec:Tg-I} and \ref{sssec:Th-I} respectively. 

The relations in (e) rely on relation (f), so we will first consider the latter. 

%%%%%%%%%%%%%%% NEW SUBSECTION %%%%%%%%%%%%%%%%%%%%%
%
\subsection*{Relation (f)} 

For this, we need more work. Consider a tube $\T_g$ of rank $g$. We want to describe the effect 
of a composition of $2g$ elementary moves from an indecomposable object $X$ of $\T_g$ to itself. 

Let $[a_{\partial},b_{\partial}]$ be the arc (viewed in $\U$), $b\ge a+2$. 
There are (at most) two 
elementary moves on $[a_{\partial},b_{\partial}]$. 

Going down (up) from a peripheral arc in $\U$ corresponds to moving the endpoint (resp. the starting point) 
of the corresponding arc one step to the left, thus making it shorter (longer). 
We write $f_d$ for the elementary move downwards, 
$f_d:[a_{\partial},b_{\partial}]\mapsto[a_{\partial},(b-1)_{\partial}]$ and $f_u$ for the one going up, 
$f_u:[a_{\partial},b_{\partial}]\mapsto [(a-1)_{\partial},b_{\partial}]$. Note that for $b=a+2$, the image 
under $f_d$ is a boundary segment and hence zero. 

For $g\ge 1$ 
we write $f^{\downarrow g}$ for the composition of 
$g$ elementary moves downwards and 
$f^{\uparrow g}$ for $g$ consecutive elementary moves upwards. 
We then abbreviate the composition of $g$ downwards with $g$ upwards 
elementary moves by $f^{\downarrow\uparrow g}$: 
$$
f^{\downarrow\uparrow g}([a_{\partial},b_{\partial}]):=f^{\uparrow g}\circ f^{\downarrow g}([a_{\partial},b_{\partial}])
$$ 

We will use the notations for these compositions of elementary moves within the universal 
cover $\U$, and also in the annulus. 

The effect of a composition of $g$ downwards with $g$ upwards moves on peripheral arcs 
at the lower boundary is the following: 

\begin{eqnarray*}
f^{\downarrow\uparrow g}[i_{\partial},j_{\partial}] & = & 
		\left\{
 		\begin{array}{ll} [(i-g)_{\partial}, (j-g)_{\partial}] & \mbox{if $j-i>g+1$} \\
		       0 & \mbox{else}. 
		\end{array}
		\right.
%\mbox{ unless $j-i\le g+1$. In that case, the image is zero.} 
\end{eqnarray*}

For arcs of $\T_h$, we define $f^{\uparrow\downarrow h}$ analogously, the effect of $h$ elementary moves 
up followed by $h$ elementary moves down. Since for endpoints at $\partial'$, elementary moves 
increase an endpoint by $+1$, we get 

\begin{eqnarray*}
f^{\uparrow\downarrow h}[i_{\partial},j_{\partial}] & = & 
		\left\{
 		\begin{array}{ll} [(i+h)_{\partial}, (j+h)_{\partial}] & \mbox{if $j-i>h+1$} \\
		       0 & \mbox{else}. 
		\end{array}
		\right.
%\mbox{ unless $j-i\le g+1$. In that case, the image is zero.} 
\end{eqnarray*}

%The following statement follows directly from the definitions. 
%\begin{rem} \label{rem:g-moves}
%Consider the stripe $\U$. Let $[i_{\partial},j_{\partial}]$ be an oriented arc in $\U$, $j\ge i+2$, 
%let $g\ge 1$. Then the following holds: 
%\begin{enumerate}
%\item
%$f_d^{g}[i_{\partial},j_{\partial}]=[i_{\partial},(j-g)_{\partial}]$ unless $j-i\le g+1$. 
%In that case, the image is zero. 
%\item
%$f_u^{g}[i_{\partial},j_{\partial}]=[(i-g)_{\partial},j_{\partial}]$ 
%\item 
%$f^{\downarrow\uparrow g}[i_{\partial},j_{\partial}]=[(i-g)_{\partial}, (j-g)_{\partial}]$ unless $j-i\le g+1$. In that case, the image is zero. 
%\end{enumerate}
%\end{rem}

In $P_{g,h}$, this shift by 
$g$ or by $+h$ is not visible, 
the effect of the compositions $f^{\downarrow\uparrow g}$ and $f^{\uparrow\downarrow h}$ 
on arcs in $P_{g,h}$ is to move 
both endpoints of a peripheral arc around the boundary once (or zero, if the arc is close to 
the mouth of the tube): 
% 
%
%\begin{lem}
%Let $\alpha=\pi[i_{\partial},j_{\partial}]$ and $\beta=\pi[k_{\partial'},l_{\partial'}]$ 
%be peripheral arcs. 
%Then the following holds: 
%\begin{enumerate}
%\item
%$f_d^g \pi[i_{\partial},j_{\partial}]=$ 
%				$\left\{\begin{array}{ll} 
%				\pi[i_{\partial},(j-g)_{\partial}]    & \mbox{ for $j-i\ge g+2$} \\
%                                        0 & \mbox{else} 
%                                \end{array}\right.$ 
%\item
%$f_u^g \pi[i_{\partial},j_{\partial}]=\pi[(i-g)_{\partial},j_{\partial}]$. 
%\item 
%In particular, 
$$
f^{\downarrow\uparrow g} \pi[i_{\partial},j_{\partial}]=\left\{
                  \begin{array}{ll} \pi[i_{\partial},j_{\partial}]& \mbox{if $j-i\ge g+2$} \\ 
                                            0 & \mbox{else} 
                  \end{array}
                   \right. 
$$
$$
f^{\uparrow\downarrow h} \pi[k_{\partial'},l_{\partial'}]=\left\{
                  \begin{array}{ll} \pi[k_{\partial'},l_{\partial'}]& \mbox{if $l-k\ge h+2$} \\ 
                                            0 & \mbox{else} 
                  \end{array}
                   \right. 
$$
%\end{enumerate}
%\end{lem}

Note that it follows from the (geometric) mesh relations, that the effect 
of $g$ downwards moves 
combined with $g$ upwords moves in $\T_g$ ($h$ upwards moves with 
$h$ downwards moves in $\T_h$) 
is independent of the order in which this moves are done. 

Relation (f) involves the four ``top'' vertices in $Q_m$ and arrows 
corresponding to long 
moves between them. We will first describe these vertices 
giving one lift in $\U$ for each of them. 
First recall that $(ghm,0)_P=\tau^{-ghm}(0,0)_P $, 
$(ghm,0)_I= \tau^{ghm}(0,0)_I $ and observe that 
$(0,g)_g$ has as lift $[0_{\partial},2_{\partial}]$, 
$(0,0)_h$ has as lift $[-2_{\partial'},0_{\partial'}]$ 
(cf. Figure~\ref{fig:quiver-again} for $m=2$ to determine the latter). 

$$
\begin{array}{lcl} 
\mbox{vertices of $Q_m$}  & \quad & \mbox{lifts in $\U$} \\ %, equivalences in $P_{g,h}$} \\
\hline
(ghm,0)_P % =  \tau^{-ghm}(0,0)_P 
     &  & [-ghm_{\partial},(h+ghm)_{\partial'}] \\
(2hm(n+1)+h,0)_h  && [-2_{\partial'},(2hm(n+1)+h)_{\partial'}] \\
%  && \hat{=} [(-2-h(m(n+1)+hm))_{\partial'},(h+ghm)_{\partial'}]\\
%  && \hat{=} [(-2-h(m(n+1)+hm+1))_{\partial'},ghm_{\partial'}]\\
(2gm(n+1)+g,g)_g && [0_{\partial}, (2gm(n+1) +g+2)_{\partial}] \\  
%  && \hat{=} [-ghm_{\partial},(2gm(n+1) -ghm+g+2)_{\partial}]\\ 
%  && \hat{=}  [-ghm_{\partial},(2+g(m(n+1) +gm+1))_{\partial}] \\ 
(ghm,0)_I %= \tau^{ghm}(0,0)_I  
     & & [(-2-ghm)_{\partial'},(2+ghm)_{\partial}] %\\ 
% && \hat{=}[(-h(m(n+1)+hm+1-j)-2)_{\partial'},(2-ghm-g+jg)_{\partial}] \\ 
% & & \hat{=} [(-2-ghm-hj)_{\partial'},(2+gmh-gj)_{\partial}]
\end{array}
$$

In terms of lifts in $\U$, the four arrows $\iota_*(0)$, $\kappa_*(0)$ 
of $Q_m$ 
can be described as long moves moving vertically between the two 
boundary components 
composed with a sequence of elementary moves. 

%%%%%%%%%%%%%%%%%%%%%%%%%%%%%%%%%%%%%%%
\subsubsection{$\iota_{\infty}(0)$: $(ghm,0)_P\to (2hm(n+1)+h,0)_h$}\label{sssec:iota-h}
%%%%%%%%%%%%%%%%%%%%%%%%%%%%%%%%%%%%%%%
This arrow is a long move from an arc $\partial\to \partial'$ to 
a peripheral arc at $\partial'$. In geometric terms it is a rotation around the common 
ending point on $\partial'$ of the involved arcs. In terms of lifts in $\U$ with 
a common ending point: 
$$
[-ghm_{\partial},(h+ghm)_{\partial'}] \stackrel{(\nu_r)_r\mu}{\longrightarrow}
[(-2-h(m(n+1)+hm))_{\partial'},(h+ghm)_{\partial'}]
$$
The effect of $\iota_{\infty}(0)$ is to first use a long move $\mu$, 
changing the starting 
point from $ghm_{\partial}$ by sending it 
vertically across to $hhm_{\partial'}$ 
and then to use $2+hm(n+1)$ elementary moves $\nu_r$ 
still fixing the ending point on $\partial'$ 
to send this starting 
point along $\partial'$ to the left by subtracting $2+hm(n+1)$. 

%%%%%%%%%%%%%%%%%%%%%%%%%%%%%%%%%%%%%%%
\subsubsection{$\iota_0(0)$: $(ghm,0)_P\to (2gm(n+1)+g,g)_g$}\label{ssec:iota-g}
%%%%%%%%%%%%%%%%%%%%%%%%%%%%%%%%%%%%%%%
The arrow $\iota_0(0)$ is a long move around a common starting point on $\partial$. 
In terms of lifts: 
$$
[-ghm_{\partial},(h+ghm)_{\partial'}]\stackrel{(\nu_r)_r\mu}{\longrightarrow}
[-ghm_{\partial},(2+g(m(n+1) +gm+1))_{\partial}]
$$
The effect of $\iota_0(0)$ is to first send the ending point 
of the arc from $(h+ghm)_{\partial'}$ to $(g(1+gm))_{\partial}$ (long move $\mu$) 
and then to send 
this point along $\partial$ to the right by adding $2+gm(n+1)$ (composition of 
$2+gm(n+1)$ elementary moves $\nu_r$ around the same starting point). 

%%%%%%%%%%%%%%%%%%%%%%%%%%%%%%%%%%%%%%%
\subsubsection{$\kappa_{\infty}(0)$: $(2hm(n+1)+h,0)_h\to (ghm,0)_I$}\label{sssec:kappa-h}
%%%%%%%%%%%%%%%%%%%%%%%%%%%%%%%%%%%%%%%
This arrow corresponds to a rotation around the common starting point on 
$\partial'$. 
In $\U$, 
$$
[(-2-h(m(n+1)+hm))_{\partial'},(h+ghm)_{\partial'}]\stackrel{(\nu_r)_r\mu}{\longrightarrow}
[(-2-h(m(n+1)+hm)_{\partial'}, 2-ghm_{\partial}]
$$
we can describe the effect of $\kappa_{\infty}(0)$. 
First, the endpoint of the arc on $\partial'$ is sent vertically across to 
$(g+g^2m)_{\partial}$ under the long move $\mu$ and then under $2-g-g(m(n+1))$ 
elementary moves $\mu_r$, the endpoint is 
sent to the right to obtain $[(-2-h(m(n+1)+hm)_{\partial'}, 2-ghm_{\partial}]$. 

%%%%%%%%%%%%%%%%%%%%%%%%%%%%%%%%%%%%%%%
\subsubsection{$\kappa_0(0)$: $(2gm(n+1)+g,g)_g\to (ghm,0)_I$}\label{sssec:kappa-g}
%%%%%%%%%%%%%%%%%%%%%%%%%%%%%%%%%%%%%%%
$\kappa_0(0)$ corresponds to a rotation around the common ending point on 
$\partial$. 
In $\U$
$$
[-ghm_{\partial},(2+g(m(n+1) +gm+1))_{\partial}]\stackrel{(\nu_r)_r\mu}{\longrightarrow}
[(-2+ghm+h)_{\partial'},(2+g(m(n+1)+gm+1))_{\partial}]
$$
We first have a long move $\mu$ fixing the endpoints, sending the starting point 
from $-ghm$ on ${\partial}$ across to $-h^2m$ on $\partial'$. 
This is then composed with  $-2+h+hm(n+1)$ 
elementary moves $\mu_r$ to send the 
new starting point to the right and get the desired result.

%%%%%%%%%%%%%%%%%%%%%%%%%%%%%%%%%%%%%%%
\subsubsection{$\rho_{\infty}^h\pi_{\infty}^h$ and $f^{\uparrow \downarrow h}$} 
%%%%%%%%%%%%%%%%%%%%%%%%%%%%%%%%%%%%%%%

In terms of arcs, the effect of the path $\rho_{\infty}^h\pi_{\infty}^h$ is 
$f^{\uparrow \downarrow h}$ and the 
effect of $\pi_0^g\rho_0^g$ is $f^{\downarrow\uparrow g}$.

%%%%%%%%%%%%%%%%%%%%%%%%%%%%%%%%%%%%%%%
\subsection*{Path $\kappa_{\infty}(0)(\rho_{\infty}^h\pi_{\infty}^h)^j\iota_{\infty}(0)$}
%%%%%%%%%%%%%%%%%%%%%%%%%%%%%%%%%%%%%%%
We use the lifts from above (Subsection~\ref{sssec:iota-h}) describing $\iota_{\infty}(0)$ 
and compose with the remaining paths: 

\begin{eqnarray*}
(\rho_{\infty}^h\pi_{\infty}^h)^j & : &  
[(-2-h(m(n+1)+hm))_{\partial'},(h+hgm)_{\partial'}]  \\ 
 & & \longrightarrow [(-2-h(m(n+1)+hm-j))_{\partial'}, (hgm+hj+h)_{\partial'}]
\end{eqnarray*} 

Applying $\kappa_{\infty}(0)$ (Subsection~\ref{sssec:kappa-h}), this has the image 
\begin{equation}\label{eq:path-1-in-f}
[(-2-h(m(n+1)+hm-j))_{\partial'},2-ghm+jg)_{\partial}]
\end{equation}

%%%%%%%%%%%%%%%%%%%%%%%%%%%%%%%%%%%%%%%
\subsection*{Path $\kappa_0(0)(\pi_0^g\rho_0^g )^{2m(n+1)+1-j}\iota_0(0)$}
%%%%%%%%%%%%%%%%%%%%%%%%%%%%%%%%%%%%%%%

The arrow $\iota_0(0)$ is described in Subsection~\ref{ssec:iota-g}, we compose with 
the remaining paths: 

\begin{eqnarray*}
(\rho_0^g\pi_0^g)^{2m(n+1)+1-j} & : &  
[-ghm_{\partial},(g(m(n+1)+gm+1)+2)_{\partial}]  \\ 
 & & \longrightarrow [(-g(2m(n+1)+hm+1-j))_{\partial}, (g(mh-j)+2)_{\partial}]
\end{eqnarray*} 

Finally, applying $\kappa_0(0)$ (Subsection~\ref{sssec:kappa-g}), we get 
\begin{equation}\label{eq:path-2-in-f}
[(\underbrace{-2-ghm-hj}_{:=x})_{\partial'},
  (\underbrace{2+ghm-gj}_{:=y})_{\partial}]
\end{equation}

With the projection to $P_{g,h}$ in mind, we translate the resulting arc in $\U$ 
so that the images (\ref{eq:path-2-in-f}) and (\ref{eq:path-1-in-f}) of the two paths 
have the same endpoint. To the endpoint $y_{\partial}$ of (\ref{eq:path-2-in-f})
we add $2g(j-hm)$ (on boundary $\partial$). 
On $\partial'$, the corresponding translation 
is by $+2h(j-hm)$: 
\begin{eqnarray*}
%\pi[(-2-ghm-hj)_{\partial'},(2+ghm-gj)_{\partial}] 
\pi[x_{\partial'},y_{\partial}] & = &  [(-2-ghm-hj+2h(j-hm))_{\partial'},(2-ghm+gj)_{\partial}]\\ 
  & = & [(-2+hj -h(2hm+gm))_{\partial'},(2-ghm+gj)_{\partial}] \\ 
  & = &  [(-2+hj -h(hm+m(n+1)))_{\partial'},(2-ghm+gj)_{\partial}] 
\end{eqnarray*}

Hence the images of these arcs in $P_{g,h}$ under the projection map are the same 
and relation (f) is satisfied. 

%%%%%%%%%%%%%%% NEW SUBSECTION %%%%%%%%%%%%%%%%%%%%%
%
\subsection*{Relation (e)} 
%%%%%%%%%%%%%%%%%%%%%%%%%%%%%%%%%%%%

We only show the first claim $\kappa_0(0)\iota_0(0) (ghm,\alpha_n)_P=0$, 
the second claim is completely analogous. 

By (f), with $j=2m(n+1)+1$, we know that 
$\kappa_0(0)\iota_0(0) =\kappa_{\infty}{\rho_{\infty}^h\pi_{\infty}^h}^h$. In terms of arcs in 
the annulus, $\rho_{\infty}^h\pi_{\infty}^h$ is the map $f^{\uparrow\downarrow h}$ from above, it sends 
any arc $\pi[i_{\partial'},j_{\partial'}]$ with $j-i\ge h+2$ to itself. Relation 
(e) concerns a path going through $(2m(n+1)h,0)_h$, the vertex corresponding to 
$\alpha:=\pi[-2_{\partial'},(2hm(n+1)-h)_{\partial'}]$ 
and this is high up in the tube $\GG^{\infty}$, 
hence the effect of $f^{\uparrow\downarrow h}$ is to send this arc to itself. Applying $f^{\uparrow\downarrow h}$ to 
the arc $2m(n+1)-1$ times still sends $\alpha$ to itself, however, the composition of 
all these elementary moves touches the mouth of the tube exactly once. In (e), this is 
precomposed with the elementary move corresponding to $(ghm,\alpha_n)_P$. 
But this means 
that we can use Case B) from Section~\ref{ssec:diamonds-long} to replace the elementary 
move in $\GG^P$ (and the long move from $\GG^P$) by a long move to the arc 
$\pi[-2_{\partial'},  (2hm(n+1)+h-1)_{\partial'}]$ followed 
by an elementary move to $\alpha$ within the tube. Hence the path from $\alpha$ 
to itself, going all the way down to the mouth, is precomposed with one downwards move 
(in $\GG^{\infty}$). Then we use the mesh relations within $\GG^{\infty}$ to push this 
all the way to the mouth, hitting the tube a second time just to the left of the other 
vertex at the mouth. That means that in our path, there are two shortest peripheral arcs. 
Between them, there is only a boundary segment (hence a zero object) and 
an arc of the form $[i_{\partial'},(i+3)_{\partial'}]$. 
The zero relation at the mouth proves the claim (Subsection \ref{ssec:zero-triangle}).


\small
\begin{thebibliography}{99}
%\bibitem[BM]
\bibitem{bm} Baur K., Marsh R.J. \emph{A geometric model of tube categories}, 
Journal of Algebra 362 (2012), 178--191. 

%\bibitem[BBM]
\bibitem{bbm} Baur, K. , Buan, A., Marsh R.J., \emph{Torsion pairs and rigid objects in tubes}, 
Algebr. Represent. Theory 17 (2014), no. 2, 565--591.  
%
%\bibitem[BM1]{bm1} Baur K., Marsh R. J. \emph{A Geometric Description
%    of the $m$-cluster categories of type $D_n$}, International
%  Mathematics Research Notices (2007) Vol. 2007 : article ID rnm011. 
%  
%\bibitem[BM2]{bm2} Baur K., Marsh R. J. \emph{A geometric description
%    of m-cluster categories}, Trans. Amer. Math. Soc. 360 (2008),
%  5789-5803. 

%\bibitem[BMRRT]
\bibitem{bmrrt} Buan A., Marsh R., Reineke M., Reiten I., Todorov G. 
\emph{Tilting theory and cluster combinatorics}, Advances in
mathematics, 204 (2), 572-618 (2006).

%\bibitem[Br]
\bibitem{br} Br\"ustle T., \emph{Darstellungsk\"ocher der erblichen Algebren 
vom Typ $\tilde{A}_n$}, Diplomarbeit, University of Zurich, 1990. \\
English translation by K. Baur, available at http://arxiv.org/abs/1312.2995

%\bibitem[BZ]
\bibitem{bz} Br\"ustle T., Zhang J. \emph{On the cluster category of a
  marked surface}, Algebra and Number Theory 5 (4), 529--566 (2011). 

%\bibitem[CCS]
\bibitem{ccs} Caldero P., Chapoton F., Schiffler
  R. \emph{Quivers with relations arising from clusters ($A_n$ case)},
  Trans. Amer. Math. Soc. 358 , no. 3, 1347-1364 (2006). 
  
%\bibitem[GR]
\bibitem{gr} Gabriel, P., Roiter, A.V., 
{\em Representations of finite-dimensional algebras}. 
Translated from the Russian 4.With a chapter by B. Keller. 
Reprint of the 1992 English translation. Springer-Verlag, Berlin,
1997. 

%\bibitem[G]
\bibitem{gehrig} Gehrig, B. {\em Geometric realizations of cluster categores} 
Master's thesis at ETHZ, 2010, 
available at 
http://www.uni-graz.at/$\sim$baurk/thesis-Gehrig.pdf

%\bibitem[K]
\bibitem{k} Keller B. \emph{On triangulated orbit categories},
  Doc. Math. 10 (2005), 551-581. 

%\bibitem[R]
\bibitem{ringel} Ringel, C.M. {\em Tame Algebras and Integral Quadratic Forms}, 
Springer Lecture Notes in Mathematics, 1984. 

%\bibitem[To]
\bibitem{to} Torkildsen H. A. \emph{A geometric realization of the $m$-cluster category of affine type $A$}, 
to be published in Communications in Algebra, arxiv: 
http://arxiv.org/pdf/1208.2138v1.pdf

%\bibitem[W]
\bibitem{w} Warkentin M. \emph{Fadenmoduln \"uber $\widetilde{A}_n$ und Cluster-Kombinatorik}, 
Diploma Thesis, University of Bonn, December 2008. Available from \\
http://nbn-resolving.de/urn:nbn:de:bsz:ch1-qucosa-94793

\end{thebibliography}
\normalsize
\end{document}